\documentclass{amsart} 

\usepackage[T1]{fontenc}
\usepackage[utf8]{inputenc}
\usepackage[english]{babel}
\usepackage{lmodern}
\usepackage{microtype}
\usepackage[a4paper,margin=1in]{geometry}
\usepackage{amsmath,amssymb,amsfonts,mathtools}
\usepackage{amsthm}
\usepackage{graphicx}
\graphicspath{{figs/}}
\usepackage[percent]{overpic}
\usepackage[section]{placeins}
\usepackage{enumitem}
\usepackage{csquotes}
\usepackage[backend=biber,style=alphabetic,maxnames=99,giveninits=true,doi=false,url=false,eprint=true,isbn=false]{biblatex}
\usepackage{aas_macros}
\usepackage[colorlinks=true,allcolors=blue]{hyperref}
\allowdisplaybreaks

\renewbibmacro{in:}{}
\DeclareFieldFormat{pages}{#1}

\numberwithin{equation}{section}
\numberwithin{figure}{section}

\newtheorem{theorem}{Theorem}[section]
\newtheorem{lemma}[theorem]{Lemma}
\newtheorem{proposition}[theorem]{Proposition}
\newtheorem{corollary}[theorem]{Corollary}
\theoremstyle{definition}
\newtheorem{definition}[theorem]{Definition}

\newtheorem{example}[theorem]{Example}
\theoremstyle{remark}
\newtheorem{remark}[theorem]{Remark}

\newcommand{\A}{\mathcal{S}}

\let\oldDelta\Delta
\renewcommand{\Delta}{\oldDelta} 

\newcommand{\M}{M}

\newcommand{\R}{\mathbb{R}}
\newcommand{\Z}{\mathbb{Z}}

\renewcommand{\S}{S}

\let\oldSigma\Sigma
\renewcommand{\Sigma}{\oldSigma}

\let\oldtau\tau
\renewcommand{\tau}{\oldtau}

\let\oldOmega\Omega 
\renewcommand{\Omega}{\oldOmega}

\let\oldalpha\alpha 
\renewcommand{\alpha}{\oldalpha}
\newcommand{\alphast}{\oldalpha_{\mathrm{st}}}

\let\olddxi\xi 
\renewcommand{\xi}{\olddxi}
\newcommand{\xist}{\olddxi_{\mathrm{st}}}

\newcommand{\X}{X}

\let\Re\oldRe
\newcommand{\Re}{R}

\newcommand{\G}{G}

\let\L\oldL
\newcommand{\L}{L}

\let\oldgamma\Gamma 
\renewcommand{\Gamma}{\oldgamma}

\newcommand{\MCG}{\mathrm{MCG}}

\renewcommand{\aa}{\mathfrak{a}}

\newcommand{\dd}{\mathfrak{d}}
\renewcommand{\ll}{\mathfrak{l}}

\title{An Algorithm to Legendrian Realize a Curve on a Ribbon Surface}

\author{Eric Stenhede}

\address{University of Vienna, 
Faculty of Mathematics,
Oskar-Morgenstern-Platz 1,
1090 Wien, Austria}
\email{eric.stenhede@univie.ac.at}

\date{\today}
\keywords{Legendrian realization, contact surgery, open books, Legendrian knots}
\def\subjclassname{\textup{2020} Mathematics Subject Classification}
\expandafter\let\csname subjclassname@1991\endcsname=\subjclassname

\subjclass{
53D35; 53D10, 57K10, 57R65, 57K10, 57K33
}

\begin{document}

\begin{abstract}
We give an explicit algorithm to Legendrian realize a homologically nontrivial simple closed curve on a ribbon surface of a Legendrian graph in the standard contact structure $(\R^3,\xist)$. As an application, we obtain an algorithm that converts an abstract open book whose monodromy is written as a product of Dehn twists along homologically nontrivial curves into a contact surgery diagram for the supported contact manifold. Along the way, we also record a uniqueness statement which is implicit in earlier work but, to our knowledge, was never written in the form needed here: any two Legendrian realizations of the same curve on a ribbon surface are Legendrian isotopic, and likewise for Legendrian knots lying on pages of open books and representing the same isotopy class on the page.
\end{abstract}

\maketitle

\setcounter{tocdepth}{1}
\tableofcontents

\section{Introduction}\label{sec:introduction}

A recurring theme in low-dimensional topology is the search for efficient and flexible ways to present geometric objects. In contact topology, two complementary languages have proved particularly effective for studying contact $3$--manifolds:
\begin{itemize}
    \item \emph{contact surgery diagrams}, in which one presents a contact manifold as the result of contact $(\pm 1)$--surgery on a Legendrian link in $(S^3,\xi_{\mathrm{st}})$, and
    \item \emph{open book decompositions}, which, by Giroux's correspondence, encode contact structures in terms of surfaces and mapping classes \cite{Ding_Geiges_2004,Giroux02}.
\end{itemize}
The relationship between these two descriptions has been studied extensively. In one direction, Stipsicz described an algorithm for constructing a compatible open book from a Legendrian surgery diagram \cite{MR2155241}, and related constructions also appear in \cite{MR1825664,MR2282016,Avdek13}. In the opposite direction, Avdek showed how to build a contact surgery diagram from an abstract open book with connected binding whose monodromy is written as a product of Dehn twists along the Lickorish generators \cite{Avdek13}.

The main difficulty in extending Avdek's algorithm to more general abstract open books where the binding might not be connected or the monodromy might not be given in terms of Dehn twists along the Lickorish generators is making the Legendrian realization principle computable. If the monodromy is given as a product of Dehn twists along arbitrary homologically nontrivial curves, then one needs an explicit way to realize those curves as Legendrian knots on pages of an open book. The Legendrian realization principle guarantees existence after a suitable perturbation of the surface, but its proof is non-constructive and global in nature \cite{Kanda1998,Honda00,Etnyre04}.

The goal of this article is to make this realization problem explicit in the setting of ribbon surfaces of Legendrian graphs in $(\R^3,\xist)$. More precisely, the core result of the paper, stated in Section~\ref{section:Stat_results}, is an algorithm that takes as input a Legendrian graph $\G\subset(\R^3,\xist)$ together with a homologically nontrivial simple closed curve $\aa$ on a ribbon surface of $\G$, and produces the front projection of a Legendrian realization of $\aa$ (Theorem~\ref{thm:algorithm}). As an application, Section~\ref{sec:applications_OB} yields an explicit procedure that converts an abstract open book whose monodromy is written as a product of Dehn twists along homologically nontrivial curves into a contact surgery diagram for the supported contact manifold (Corollary~\ref{corollary}).

Given the front projection of a ribbon surface in $(\R^3,\xist)$, it is important to notice that there is no naive local way to Legendrian realize a curve on it. Many of the crossings occurring in the front projection of the boundary link of a ribbon violate the Legendrian condition, so in general one cannot simply introduce cusps where needed and hope to obtain a Legendrian knot. Figure~\ref{fig:example_algorithm_intro} shows a concrete instance of this phenomenon.

\begin{figure}[htbp]
    \centering
    \begin{overpic}[scale=0.9]{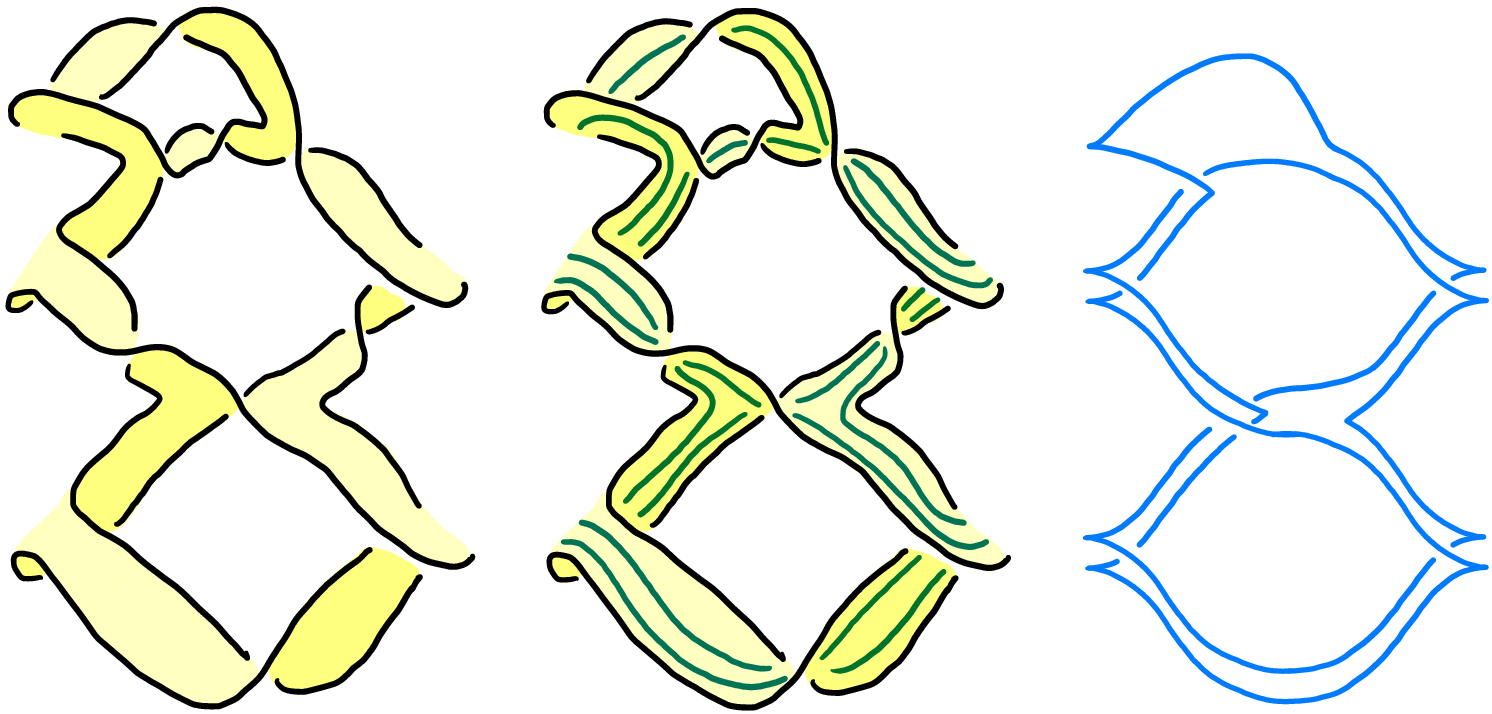}
        \put(2,23.5){$\rightarrow$}
    \end{overpic}
    \caption{Left: the front projection of a ribbon surface $\Sigma$. The crossing indicated by an arrow is one of the crossings that makes a naive Legendrian realization impossible. Middle: a homologically nontrivial curve $\aa$ on $\Sigma$. Right: the Legendrian realization of $\aa$ produced by the algorithm.}
    \label{fig:example_algorithm_intro}
\end{figure}

Very briefly, the algorithm works as follows. Starting from a Legendrian graph $\G\subset(\R^3,\xist)$, we use the handle decomposition of a ribbon surface $\Sigma$ induced by the vertices and edges of $\G$ to put the curve $\aa$ in a standard position. We then decompose $\aa$ into finitely many segments, each lying in a single handle. For every segment in a $1$--handle we assign a combinatorial quantity, the \emph{relative gain}, which governs the slope of the corresponding Legendrian segment in the front projection. By modifying $\aa$ within its isotopy class on the ribbon, we arrange that these relative gains balance globally, so that the resulting Legendrian segments glue together with matching $z$--coordinates and close up to a Legendrian knot.

Along the way, we also record a uniqueness statement that is implicit in earlier work but, to our knowledge, has not been written down in the form required here: any two Legendrian realizations of the same curve on a ribbon surface are Legendrian isotopic, and similarly for Legendrian knots lying on pages of open books and representing the same isotopy class on the page; see Theorems~\ref{thm:uniqueness_ribbon} and~\ref{thm:uniqueness_OB}. These uniqueness results are of independent interest beyond the algorithm itself, and they will also be used in forthcoming work of the author with collaborators on a contact version of Kirby's theorem~\cite{KegelStenhedeVertesiZuddasKirby}.

\begin{remark}
The results of this article form part of the author's PhD thesis~\cite{StenhedeThesis}.
\end{remark}

\subsection*{Acknowledgements}
I would like to thank Vera Vértesi and Marc Kegel for their support, encouragement, and valuable comments throughout this project. This project was initiated while I was participating in the ICERM semester programme \emph{Braids}, and I am grateful to ICERM for providing a stimulating environment in which this work began. I was supported by the Vienna School of Mathematics and by Vera Vértesi's FWF projects PAT7436924 and P~34318.
\section{Preliminaries}\label{sec:preliminaries}

\subsection{Mapping class groups and Dehn twists}

Let $\Sigma$ be a smooth, compact, oriented surface with non-empty boundary. We denote by $\mathrm{Diff}^+(\Sigma,\partial\Sigma)$ the group of orientation-preserving diffeomorphisms of $\Sigma$ that restrict to the identity on a neighborhood of the boundary, and by $\mathrm{Diff}^+_0(\Sigma,\partial\Sigma)$ the connected component containing the identity. The \emph{mapping class group} $\MCG(\Sigma,\partial \Sigma)$ is the quotient
\[
\MCG(\Sigma,\partial\Sigma)
:=\mathrm{Diff}^+(\Sigma,\partial\Sigma)/ \mathrm{Diff}^+_0(\Sigma,\partial\Sigma).
\]
This group is generated by Dehn twists along simple closed curves. We denote by $\tau_\aa^\pm$ the mapping class represented by a positive or negative Dehn twist along a simple closed curve $\aa$ in $\Sigma$.

\subsection{Contact structures, Legendrian graphs, and projections}

A \emph{contact form} on a smooth, oriented $3$--manifold $M$ is a $1$--form $\alpha\in\Omega^1(M)$ such that
\[
\alpha\wedge d\alpha > 0.
\]
A \emph{contact structure} $\xi$ on $M$ is a $2$--plane field that can be written as the kernel of a contact form, i.e.\ $\xi=\ker\alpha$ for some $\alpha$ as above. On $\R^3$ we consider the standard contact form
\[
\alphast := dz + x\,dy
\]
and the associated standard contact structure
\[
\xist := \ker\alphast.
\]
On $S^3\subset\R^4$ we use the standard contact structure
\[
\xist = \ker\bigl(x_1dy_1 - y_1dx_1 + x_2dy_2 - y_2dx_2\bigr)\big|_{S^3}.
\]
For any point $p\in S^3$ there is a contactomorphism
\[
(S^3\setminus\{p\},\,\xist|_{S^3\setminus\{p\}}) \cong (\R^3,\xist)
\]
(see \cite[Proposition~2.1.8]{Geiges_book}).

Given a contact form $\alpha$, its \emph{Reeb vector field} $R_\alpha$ is uniquely defined by
\[
d\alpha(R_\alpha,-) \equiv 0,
\qquad
\alpha(R_\alpha)\equiv 1.
\]
For the standard form $\alphast$ the Reeb vector field is $\partial_z$. A \emph{contact vector field} $\X$ for a contact structure $\xi$ is a vector field whose flow preserves $\xi$.

Let $(M,\xi)$ be a contact $3$--manifold. A smooth curve
\[
r\colon I \longrightarrow M,
\qquad I\subset \R \text{ or } I=S^1,
\]
is \emph{Legendrian} if $r'(s)\in\xi_{r(s)}$ for all $s\in I$. A properly embedded $1$--dimensional CW-complex whose edges are Legendrian segments and whose vertices have distinct oriented tangencies in the contact planes is called a \emph{Legendrian graph}. Two Legendrian graphs are \emph{Legendrian isotopic} if they are joined by a one-parameter family of Legendrian graphs preserving the cyclic order of the incident edges at every vertex.

In $(\R^3,\xist)$ there are two particularly useful ways to depict Legendrian objects: the \emph{front projection} onto the $(y,z)$--plane and the \emph{Lagrangian projection} onto the $(x,y)$--plane. We adopt the coordinate convention in Figure~\ref{fig:coordinate_system}.

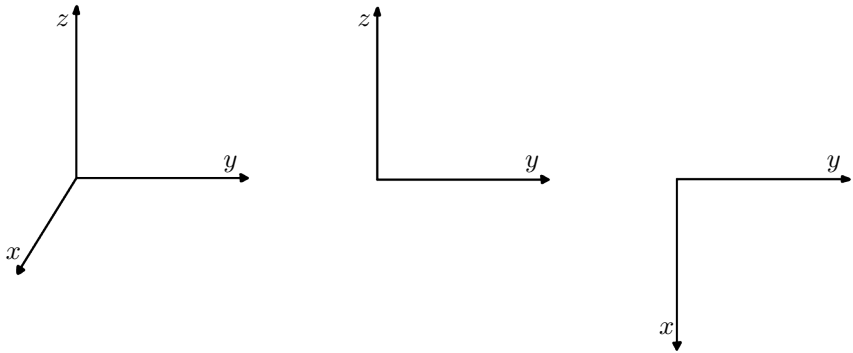
\begin{figure}[htbp]
    \centering
    \begin{overpic}[scale=1]{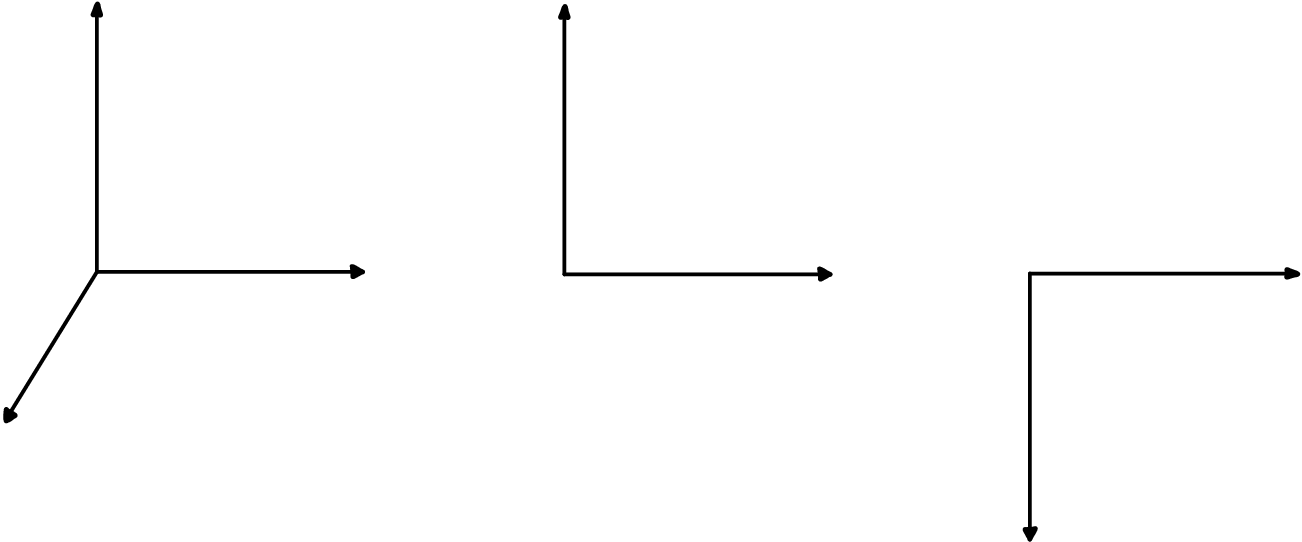}
        \put(-1,11){$x$}
        \put(5,39){$z$}
        \put(25,22){$y$}
        \put(41,39){$z$}
        \put(61,22){$y$}
        \put(77,2){$x$}
        \put(97,22){$y$}
    \end{overpic}
    \caption{Left: coordinates on $\R^3$. Center: the front projection. Right: the Lagrangian projection.}
    \label{fig:coordinate_system}
\end{figure}

If $\star$ is a Legendrian object in $(\R^3,\xist)$, we denote its front projection by $\star_F$ and its Lagrangian projection by $\star_L$.

Let $r=(r_x,r_y,r_z)\colon I\to(\R^3,\xist)$ be a Legendrian immersion. Its front projection is $r_F=(r_y,r_z)$. A point of the image of a generic front projection has one of the local models shown in Figure~\ref{fig:Legendrian_projection}: a regular point, a transverse crossing, a cusp, or a vertex of a Legendrian graph. Conversely, any planar graph in the $(y,z)$--plane whose points have neighborhoods modeled on Figure~\ref{fig:Legendrian_projection}(a)--(e) lifts uniquely to a generic Legendrian graph in $(\R^3,\xist)$.

\begin{figure}[htbp]
    \centering
    \begin{overpic}[scale=1]{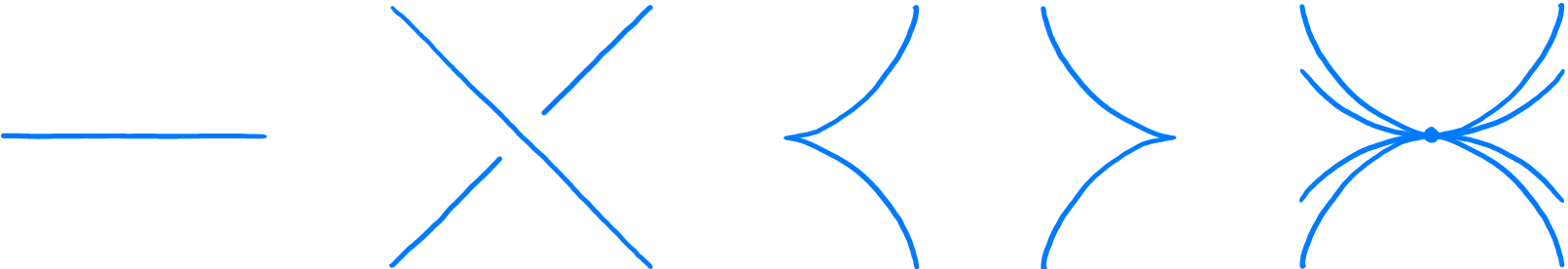}
        \put(-4,15){(a)}
        \put(21,15){(b)}
        \put(47,15){(c)}
        \put(62,15){(d)}
        \put(79,15){(e)}
        \put(83,10){$\boldsymbol\cdot$}
        \put(83,8){$\boldsymbol\cdot$}
        \put(83,6){$\boldsymbol\cdot$}
        \put(99,10){$\boldsymbol\cdot$}
        \put(99,8){$\boldsymbol\cdot$}
        \put(99,6){$\boldsymbol\cdot$}
    \end{overpic}
    \caption{Local models for points in the front projection of a generic Legendrian graph.}
    \label{fig:Legendrian_projection}
\end{figure}
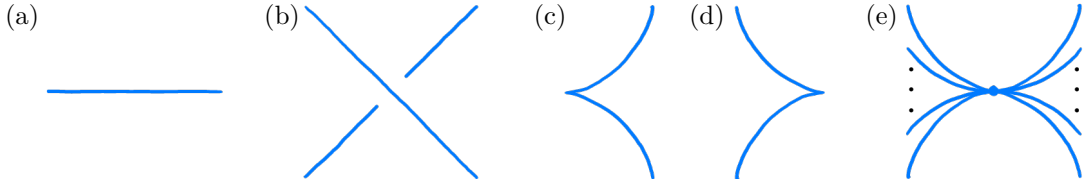

The Lagrangian projection $r_L$ of a Legendrian immersion $r$ is an immersed curve in the $(x,y)$--plane. The $z$--coordinate of $r$ can be recovered from $r_L$ and the $z$--coordinate of a single point $r(s_0)$ via
\[
r_z(s) = r_z(s_0) - \int_{s_0}^s r_x(u)\,r'_y(u)\,du.
\]
In particular, if $r$ is a closed Legendrian curve, then
\[
\oint_{r_L} x\,dy = 0.
\]
Moreover, $r$ is embedded if and only if every loop in $r_L$ (except, in the closed case, the full loop itself) encloses a nonzero oriented area. Figure~\ref{fig:Lagrangian_projection} shows the Lagrangian projections corresponding to the local models in Figure~\ref{fig:Legendrian_projection}.

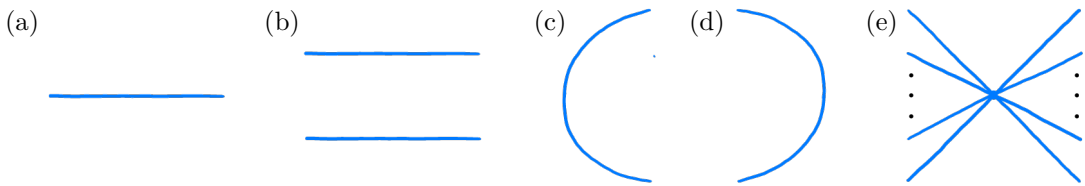
\begin{figure}[htbp]
    \centering
    \begin{overpic}[scale=1]{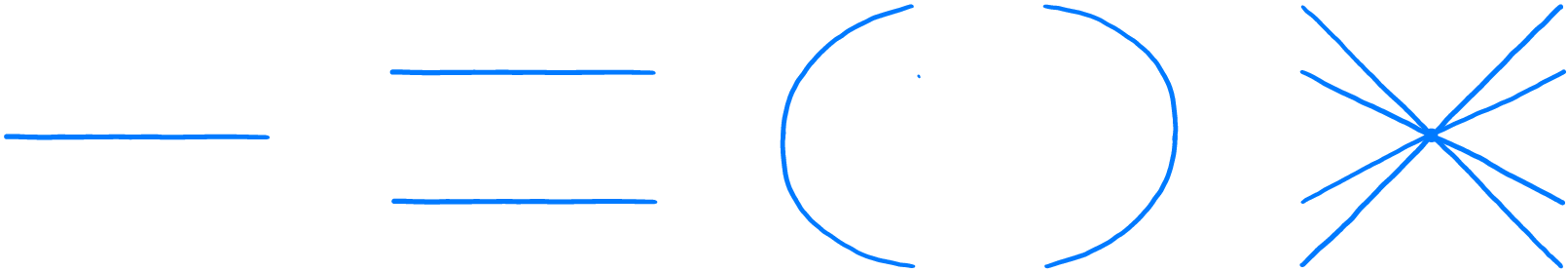}
        \put(-4,15){(a)}
        \put(21,15){(b)}
        \put(47,15){(c)}
        \put(62,15){(d)}
        \put(79,15){(e)}
        \put(83,10){$\boldsymbol\cdot$}
        \put(83,8){$\boldsymbol\cdot$}
        \put(83,6){$\boldsymbol\cdot$}
        \put(99,10){$\boldsymbol\cdot$}
        \put(99,8){$\boldsymbol\cdot$}
        \put(99,6){$\boldsymbol\cdot$}
    \end{overpic}
    \caption{Lagrangian projections corresponding to the front projections in Figure~\ref{fig:Legendrian_projection}. The $y$--coordinate agrees with that of the front projection, and the $x$--coordinate parametrizes the vertical direction (top to bottom is positive).}
    \label{fig:Lagrangian_projection}
\end{figure}

\subsection{Convex surfaces and Legendrian realization}

Given an oriented surface $\S$ in a contact manifold $(M,\xi)$, we denote by $\S_\xi$ its characteristic foliation. We say that $\S$ is \emph{convex} if there exists a contact vector field $\X$ transverse to $\S$. Equivalently, $\S$ admits a vertically invariant neighborhood: there is a neighborhood $\S\times(-\varepsilon,\varepsilon)\subset M$ of $\S\cong \S\times\{0\}$ in which the contact structure is invariant under translation in the second factor. In suitable coordinates the contact structure is then given as the kernel of a $1$--form
\[
\beta + u\,ds,
\]
where $\beta\in\Omega^1(\S)$, $u\in C^\infty(\S)$, and $s\in(-\varepsilon,\varepsilon)$ is the coordinate in the second factor. The \emph{dividing set} $\Gamma_\S$ is the set of points where the transverse contact vector field lies in $\xi$; in the vertically invariant model this is precisely the zero set of $u$.

A singular foliation $\mathcal{F}$ on $\S$ is \emph{divided} by a multicurve $\Gamma$ if $\Gamma$ is transverse to $\mathcal{F}$ and there exist an area form $\omega$ on $\S$ and a vector field $v$ defining $\mathcal{F}$ such that $\mathcal{L}_v\omega\neq 0$ on $\S\setminus\Gamma$, with $v$ pointing out of the positive region and into the negative region along $\Gamma$. If $\S$ is convex, then its characteristic foliation is divided by the dividing set.

The flexibility theorem of Giroux says that, once the dividing set is fixed, the characteristic foliation may be changed arbitrarily inside the class of foliations divided by it.

\begin{theorem}[Giroux's flexibility theorem \cite{Giroux91,Etnyre_VHM11}]\label{thm:Giroux_flexibility}
Let $\S$ be a convex surface in a contact manifold $(M,\xi)$ with respect to a contact vector field $\X$. If $\S$ has boundary, assume that $\partial\S$ coincides with the dividing set $\Gamma_\S$. Let $\mathcal{F}$ be a singular foliation on $\S$ divided by $\Gamma_\S$, and let $\mathcal{N}(\S)$ be any neighborhood of $\S$ of the form $\S\times(-\varepsilon,\varepsilon)$ in $M$. Then there exists an isotopy
\[
\psi_t\colon \S \longrightarrow \mathcal{N}(\S), \qquad t\in[0,1],
\]
of embeddings such that:
\begin{itemize}
    \item $\psi_0$ is the inclusion $\S\hookrightarrow M$;
    \item for all $t\in[0,1]$, the surface $\psi_t(\S)$ is convex with respect to $\X$ and has dividing set $\psi_t(\Gamma_\S)$;
    \item the characteristic foliation $(\psi_1(\S))_\xi$ coincides with $\psi_1(\mathcal{F})$.
\end{itemize}
Moreover, if $\mathcal{F}$ coincides with $\S_\xi$ on a neighborhood of $\Gamma_\S$ in $\S$, then the isotopy may be chosen fixed on a small enough neighborhood of $\Gamma_\S$.
\end{theorem}

A graph $C\subset\S$ is called \emph{nonisolating} if it is properly embedded, disjoint from $\partial\S$, and transverse to the dividing set when they intersect. In the situation most relevant for us, namely when the dividing set coincides with the boundary, this simply means that $C$ is disjoint from the boundary.

\begin{theorem}[Legendrian realization principle \cite{Honda00,Etnyre04}]\label{thm:Legendrian_realization}
Let $\S$ be a convex surface in a contact manifold $(M,\xi)$ with respect to a contact vector field $\X$. If $\S$ has boundary, assume that $\partial\S$ coincides with the dividing set $\Gamma_\S$. Let $C\subset\S$ be a nonisolating graph, and let $\mathcal{N}(\S)$ be any neighborhood of $\S$ of the form $\S\times(-\varepsilon,\varepsilon)$ in $M$. Then there exists an isotopy
\[
\psi_t\colon \S \longrightarrow \mathcal{N}(\S),\qquad t\in[0,1],
\]
of embeddings such that:
\begin{itemize}
    \item $\psi_0$ is the inclusion $\S\hookrightarrow M$, and $\psi_t$ is fixed on a neighborhood of $\Gamma_\S$;
    \item for all $t\in[0,1]$, the surface $\psi_t(\S)$ is convex with respect to $\X$ and has dividing set $\Gamma_{\psi_t(\S)}=\Gamma_\S$;
    \item the graph $\psi_1(C)\subset \psi_1(\S)$ is Legendrian.
\end{itemize}
\end{theorem}

We will also use the surface neighborhood theorem in the following form.

\begin{theorem}[Giroux's surface neighborhood theorem \cite{Giroux91}]\label{thm:neigh_thm}
Let $S_i$ be compact oriented surfaces embedded in contact $3$--manifolds $(M_i,\xi_i)$, for $i=0,1$, and let
\[
\phi: S_0 \to S_1
\]
be a diffeomorphism which identifies the characteristic foliations induced by $\xi_0$ and $\xi_1$ (with orientation). Then there is a contactomorphism of neighborhoods
\[
\psi : \mathcal{N}(S_0) \to \mathcal{N}(S_1),
\]
sending $S_0$ to $S_1$, and such that on $S_0$ the restriction $\psi|_{S_0}$ is isotopic to $\phi$ via an isotopy preserving the characteristic foliation.
\end{theorem}

\subsection{Contact surgery and open books}

Let $L$ be a Legendrian knot in a contact manifold $(M,\xi)$. A standard neighborhood $\mathcal{N}(L)$ of $L$ is a solid torus whose boundary is a convex surface. The slope of the dividing set on $\partial\mathcal{N}(L)$ determines the \emph{contact framing} of $L$. Removing the interior of $\mathcal{N}(L)$ and gluing back a solid torus by $(\pm 1/n)$--surgery with respect to the contact framing produces, in a canonical way, a new contact manifold; this is called \emph{contact $(\pm 1/n)$--surgery} on $L$. In this article we only consider contact $(\pm 1)$--surgery. If $L$ is equipped with such a surgery coefficient, we write $L^\pm$.

An \emph{open book} on a closed oriented $3$--manifold $M$ is a pair $(B,\pi)$ where $B\subset M$ is a link, called the \emph{binding}, and $\pi\colon M\setminus B\to S^1$ is a fibration which coincides with the angular coordinate on a neighborhood $B\times D^2$ of the binding. The closures of the fibers of $\pi$ are the \emph{pages}. An \emph{abstract open book} is a pair $(\Sigma,\phi)$ where $\Sigma$ is a compact oriented surface with non-empty boundary and $\phi\in\mathrm{Diff}^+(\Sigma,\partial\Sigma)$ is the identity near the boundary. Equivalent abstract open books correspond to diffeomorphic open books.

We say that a contact structure $\xi$ on $M$ is \emph{compatible with} (or \emph{supported by}) an open book $(B,\pi)$ if there exists a contact form $\alpha$ for $\xi$ such that the binding is positively transverse and $d\alpha$ restricts to a positive area form on the interior of every page. The fundamental correspondence due to Giroux can be stated as follows.

\begin{theorem}[Giroux's correspondence \cite{Giroux02,BHH24,LicVer1_24,LicVer2_24}]\label{thm:Giroux_correspondence}
Let $M$ be a closed, oriented $3$--manifold. There is a one-to-one correspondence between positive cooriented contact structures on $M$ up to isotopy and open books on $M$ up to isotopy and positive stabilization.
\end{theorem}

The version in terms of abstract open books identifies contact $3$--manifolds up to contactomorphism with abstract open books up to equivalence and positive stabilization.

The next theorem will only be used in the final application. Once ribbon surfaces have been introduced in Section~3, it says that suitable ribbons arise as pages of supporting open books.

\begin{theorem}[Giroux \cite{Giroux02}]\label{thm:contact_cell}
Let $G$ be the $1$--skeleton of a contact cell decomposition of a contact $3$--manifold $(M,\xi)$. Then any ribbon surface $R_G$ of $G$ is a page of an open book compatible with $\xi$.
\end{theorem}

Finally, we recall the effect of contact surgery on knots sitting in pages of an open book.

\begin{proposition}[\cite{MR1919715,Etnyre06}]\label{prop:Dehn_surgery_Dehn_twist}
Let
\[
  \boldsymbol{L}
  =
  \bigcup_{i=1}^n L_i^{\delta_i},
  \qquad
  \delta_i\in\{+,-\},
\]
be a contact surgery link whose components are embedded in different pages of an open book $(B,\pi)$ supporting a contact structure $\xi$ in a manifold $M$. Let $(\Sigma,h)$ be the corresponding abstract open book, where $\Sigma$ is a page chosen just before the page containing $L_1$. Then the contact manifold obtained by contact surgery on $\boldsymbol{L}$ is contactomorphic to the contact manifold supported by the abstract open book $(\Sigma, h\circ\tau_{\scriptscriptstyle{\boldsymbol{L}}})$, where
\[
    \tau_{\scriptscriptstyle{\boldsymbol{L}}}
    :=
    \tau_{\scriptscriptstyle{\overline{L_n}}}^{-\delta_n}
    \cdots
    \tau_{\scriptscriptstyle{\overline{L_1}}}^{-\delta_1}
    \in \MCG(\Sigma,\partial\Sigma)
\]
and $\overline{L_i}$ denotes the isotopy class on the page represented by $L_i$.
\end{proposition}

\section{Ribbon surfaces}

Throughout the remainder of this article, $\G$ denotes a Legendrian graph with no vertices of valency~$1$.

\begin{definition}\label{def:ribbon}
    Let $\G$ be a Legendrian graph in a contact manifold $(\M,\xi)$. A \emph{ribbon surface}\footnote{Giroux first introduced such surfaces in \cite{Giroux02} (without giving them a name). The formulation here is taken from \cite[Definition~2.2]{Avdek13}.} of~$\G$ is a compact oriented surface $\Sigma\subset\M$ such that:
    \begin{itemize}
        \item $\G\subset\Sigma$.
        \item There exists a contact form $\alpha$ for $(\M,\xi)$ whose Reeb vector field $\Re_\alpha$ is positively transverse to~$\Sigma$.
        \item There exists a vector field $v$ on $\Sigma$ that directs the characteristic foliation of $\Sigma$, is positively transverse to $\partial\Sigma$, and whose time-$t$ flow $\Phi^t_v$ satisfies
        \[
        \bigcap_{t\in(0,\infty)} \Phi_v^{-t}(\Sigma) = \G.
        \]
    \end{itemize}
    The graph $\G$ is called the \emph{Legendrian skeleton} of~$\Sigma$.
\end{definition}

An example of a ribbon surface is shown in Figure~\ref{fig:ribbon_example}.\footnote{It will become clear later why this is a ribbon surface: it arises from an algorithm that constructs ribbon surfaces from Legendrian graphs, together with an application of the cancellation lemma.}

\begin{figure}[htbp]
    \centering
    \begin{overpic}[scale=0.9]{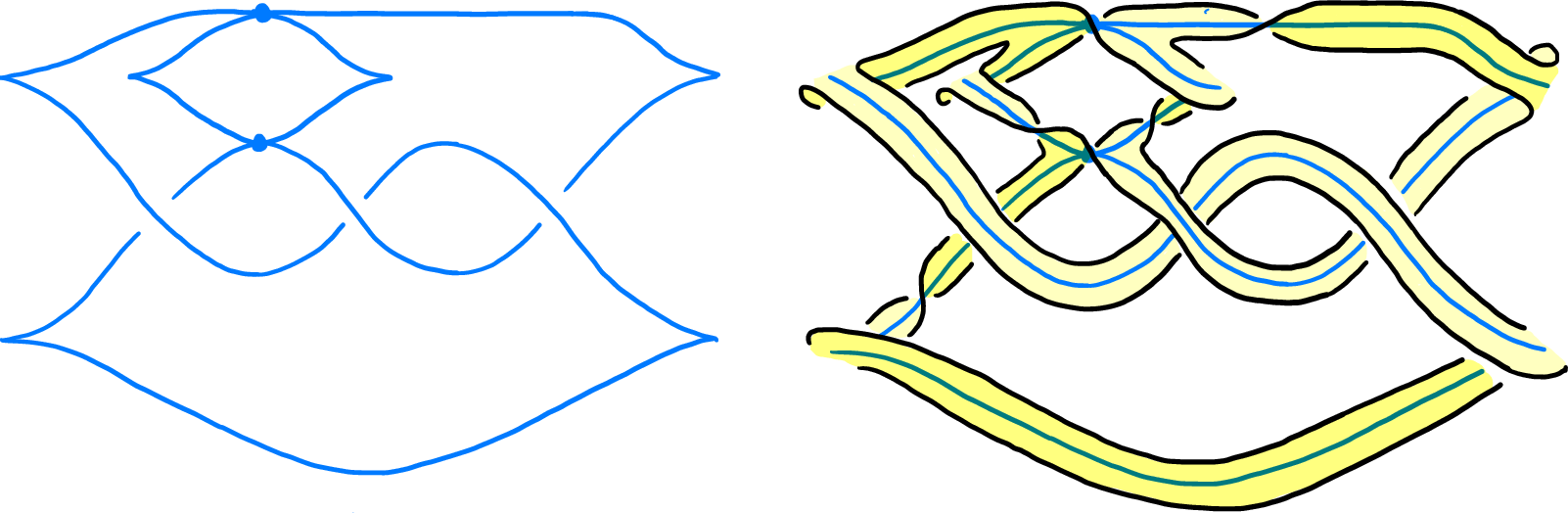}
    \end{overpic}
    \caption{Left: a Legendrian graph in $(\R^3,\xist)$. Right: a ribbon surface of the Legendrian graph.}\label{fig:ribbon_example}
\end{figure}

Every Legendrian graph $\G$ admits a ribbon surface: one can take a sufficiently small ``thickening'' of $\G$ in the direction of the contact structure. Any such ribbon is automatically convex, since it is transverse to a Reeb vector field. However, we can say more about this. 

\begin{lemma}\label{lem:ribbon_convex}
    Let $\Sigma$ be a ribbon surface. Then $\Sigma$ is convex with respect to a contact vector field whose dividing set $\Gamma_\Sigma$ coincides with the boundary $\partial\Sigma$.
\end{lemma}

Thus, from now on a ribbon surface will always be regarded as a convex surface with dividing set equal to its boundary.

\begin{proof}[Proof of Lemma~\ref{lem:ribbon_convex}]
    Use the notation of Definition~\ref{def:ribbon}. Flowing $\Sigma$ along the Reeb vector field $\Re_\alpha$ produces a vertically invariant neighborhood $\Sigma \times [-\varepsilon,\varepsilon]$ in which
    \[
        \alpha = \beta + ds,
    \]
    with $\beta \in \Omega^1(\Sigma)$, and $s \in [-\varepsilon,\varepsilon]$. We have that
    \[
        d\alpha|_\Sigma = d\beta =: \omega,
    \]
    and $\omega$ is a positive area form on $\Sigma$ since $\Re_\alpha$ is positively transverse to $\Sigma$. Up to rescaling the vector field $v$ defining the characteristic foliation, we may assume that $\beta = i_v d\beta$. Then
    \[
        \operatorname{div}_{d\beta}(v)\,d\beta
        =
        d(i_v d\beta)
        =
        d\beta,
    \]
    so $\operatorname{div}_{d\beta}(v) = 1 > 0$ everywhere. Together with the fact that $\partial \Sigma$ is a positive transverse link, this implies that the characteristic foliation of~$\Sigma$ is divided by the multicurve $\partial \Sigma$.
    
    The proof of \cite[Proposition~II.2.1]{Giroux91} shows that we can find a contact vector field transverse to $\Sigma$ which induces $\partial \Sigma$ as dividing set. For completeness, we reproduce the argument here (in the form adapted from \cite[Theorem 4.8.5]{Geiges_book}).

    Consider the $1$-form $\widetilde{\alpha}$ on $\Sigma \times [0,1]$ given by
    \[
        \widetilde{\alpha} = \beta + \widetilde{u}\,ds,
    \]
    where $\widetilde{u} \in C^\infty(\Sigma \times [-\varepsilon,\varepsilon])$ is a function to be chosen. The condition for $\widetilde{\alpha}$ to be a contact form is
    \[
        \widetilde{u}\,\operatorname{div}_\omega(v) - v(\widetilde{u}) > 0,
    \]
    which in our case reduces to
    \[
        \widetilde{u} - v(\widetilde{u}) > 0,
    \]
    since $\operatorname{div}_\omega(v) = 1$.

    Now use the flow of $-v$ to identify a collar neighborhood of $\partial \Sigma$ with $\partial \Sigma \times [0,\varepsilon]$. In these coordinates, the parameter $t \in [0,\varepsilon]$ is such that $v = -\partial_t$.

    Define $\widetilde{u}$ to be equal to $1$ outside $\partial \Sigma \times [0,\varepsilon]$. In this region the contact condition for $\widetilde{\alpha}$ is clearly satisfied.

    On $\partial \Sigma \times [0,\varepsilon]$, choose $\widetilde{u}$ of the form
    \[
        \widetilde{u}(t) = g(t)e^{-t}.
    \]
    Then
    \[
        \widetilde{u} - v(\widetilde{u})
        =
        \widetilde{u} + \frac{\partial \widetilde{u}}{\partial t}
        =
        \frac{\partial g(t)}{\partial t}\,e^{-t},
    \]
    so the contact condition is satisfied provided we choose $g$ with $\partial g / \partial t > 0$.

    We can easily find a smooth function $g$ with this property, and we may further assume that $g = 1/e^{-t}$ near $t = \varepsilon$ (so that $\widetilde{u}$ matches the value $1$ outside $\partial \Sigma \times [0,\varepsilon]$) and that $g = 0$ along $\partial \Sigma \times \{0\}$. The latter condition implies that $\widetilde{u} = 0$ precisely on $\partial \Sigma \times \{0\}$. Hence $\partial_s$ is a contact vector field for the contact form $\widetilde{\alpha}$, transverse to $\Sigma$ and inducing $\partial \Sigma$ as dividing set.

    To conclude, note that $\ker \widetilde{\alpha}$ induces on $\Sigma$ the same characteristic foliation as $\ker \alpha$. By Theorem~\ref{thm:neigh_thm}, this implies that $\Sigma$ is also convex with respect to $\ker \alpha$, with dividing set equal to $\partial \Sigma$ for a suitable contact vector field.
\end{proof}

A ribbon surface $\Sigma$ of a Legendrian graph $\G$ carries a natural handle decomposition (which depends on~$\G$). The $0$--handles are the thickenings of pairwise disjoint regular closed neighborhoods of the vertices of $\G$, and the $1$--handles are the connected components of the closure of the complement of the union of the $0$--handles. See Figure~\ref{fig:handle_dec_ribbon}.

\begin{figure}[htbp]
    \centering
    \begin{overpic}[scale=0.9]{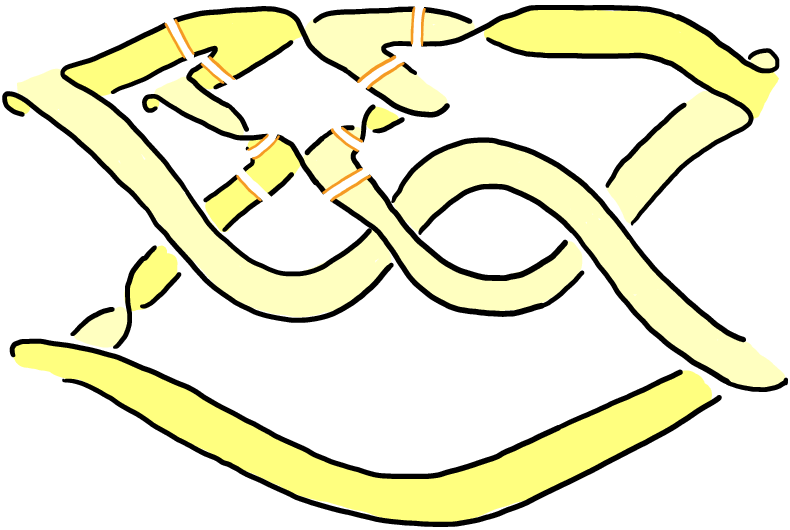}
    \end{overpic}
    \caption{The handle decomposition of the ribbon from Figure~\ref{fig:ribbon_example}.}\label{fig:handle_dec_ribbon}
\end{figure}

A ribbon surface is called \emph{generic}\footnote{This is non-standard terminology, adopted here for convenience.} if its characteristic foliation has the following form:
\begin{itemize}
    \item every $0$--handle contains exactly one singular point, which is a positive elliptic point;
    \item each $1$--handle contains finitely many singular points, all of which are positive elliptic or positive hyperbolic.
\end{itemize}
In practice, the characteristic foliation on the handles of a generic ribbon looks as in Figure~\ref{fig:generic_ribbon}.

\begin{figure}[htbp]
    \centering
    \begin{overpic}[scale=1]{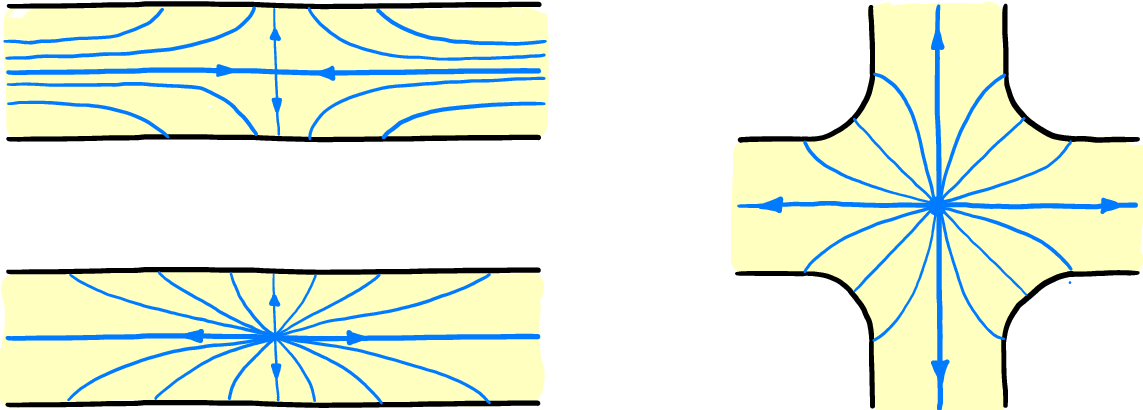}
    \end{overpic}
    \caption{Left: the characteristic foliation on a $1$--handle of a generic ribbon, obtained by concatenating pieces like those at the top-left and bottom-left. Right: the characteristic foliation on a $0$--handle. Here the vertex has valency~$4$, but the general case is analogous.}\label{fig:generic_ribbon}
\end{figure}

Since this foliation is divided by $\partial\Sigma$, Giroux's flexibility Theorem~\ref{thm:Giroux_flexibility} implies that any ribbon can be perturbed to a generic one.

\subsubsection{Ribbon surfaces in the standard contact structure}

We now specialize to $(\R^3,\xist)$. Given the front projection of a generic Legendrian graph $\G\subset(\R^3,\xist)$, there is an explicit procedure (adapted from \cite[Algorithm~2]{Avdek13}) for drawing the front projection of a ribbon surface~$\Sigma$ for~$\G$.

For each piece of the front projection of $\G$ as in Figure~\ref{fig:Legendrian_projection}, we associate a surface as indicated in the top row of Figure~\ref{fig:ribbon_local}.

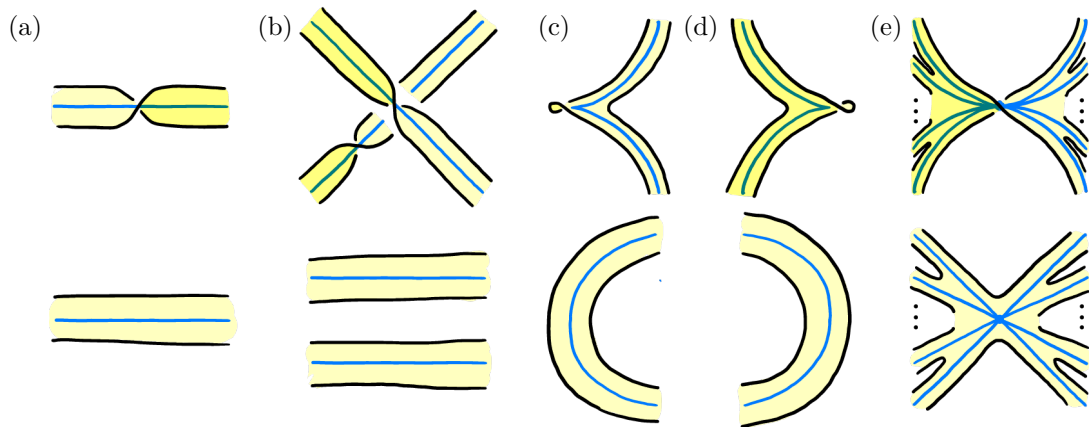
\begin{figure}[htbp]
    \centering
    \begin{overpic}[scale=1]{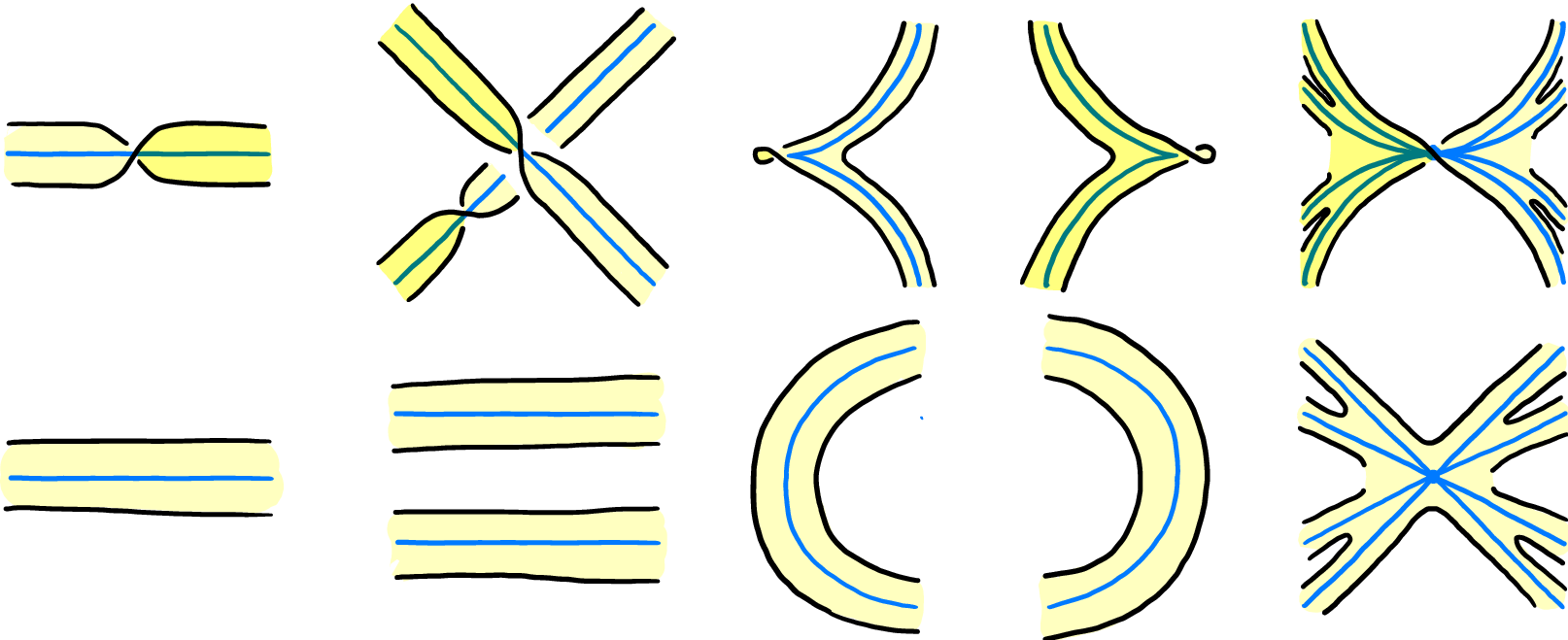}
        \put(-4,38){(a)}
        \put(20,38){(b)}
        \put(47,38){(c)}
        \put(61,38){(d)}
        \put(79,38){(e)}
        \put(83,11){$\boldsymbol\cdot$}
        \put(83,10){$\boldsymbol\cdot$}
        \put(83,9){$\boldsymbol\cdot$}
        \put(99,11){$\boldsymbol\cdot$}
        \put(99,10){$\boldsymbol\cdot$}
        \put(99,9){$\boldsymbol\cdot$}
        \put(83,31){$\boldsymbol\cdot$}
        \put(83,30){$\boldsymbol\cdot$}
        \put(83,29){$\boldsymbol\cdot$}
        \put(99,31){$\boldsymbol\cdot$}
        \put(99,30){$\boldsymbol\cdot$}
        \put(99,29){$\boldsymbol\cdot$}
    \end{overpic}
    \caption{Top row: front projection of the ribbon near a non-vertical edge, a crossing, cusps, and a vertex. Bottom row: corresponding Lagrangian projections.}\label{fig:ribbon_local}
\end{figure}

If the vertex in Figure~\ref{fig:Legendrian_projection}(e) only has edges coming from one side, the ribbon looks locally as in Figure~\ref{fig:ribbon_local2}.

\begin{figure}[htbp]
    \centering
    \begin{overpic}[scale=1]{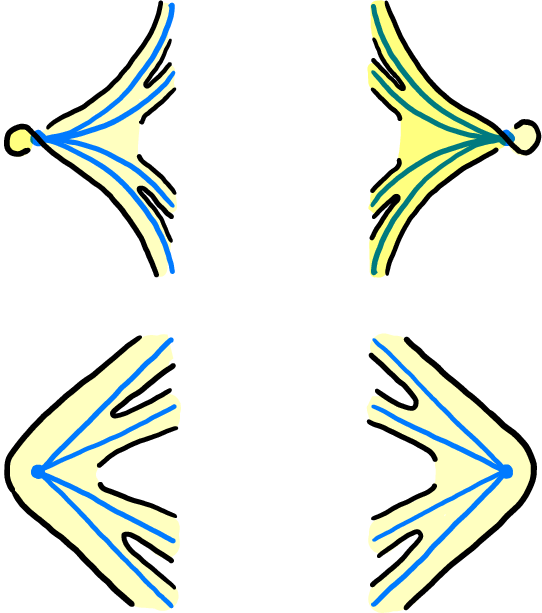}
        \put(-4,96){(f)}
        \put(49,96){(g)}
    \end{overpic}
    \caption{Top row: front projection of the ribbon near a one-sided vertex. Bottom row: the corresponding Lagrangian projections.}\label{fig:ribbon_local2}
\end{figure}

The front projection of a surface does not fully determine the surface, but it provides a good intuition for how $\Sigma$ looks. One additional piece of information is that we impose the ribbon $\Sigma$ to be everywhere transverse to $\partial_z$ (a Reeb vector field), and we choose an orientation so that the lighter side is positive (that is, we orient $\Sigma$ using the orientation of $\R^3$ and $\partial_z$). In particular, the Lagrangian projection of $\Sigma$ is an immersion and looks like the bottom rows of Figures~\ref{fig:ribbon_local} and~\ref{fig:ribbon_local2}.

We can (and will) arrange that on each local piece in Figure~\ref{fig:ribbon_local} the characteristic foliation has exactly one critical point: a positive hyperbolic point for the piece in~(a), and a positive elliptic point for the others. In the front projection, we can imagine this critical point as being located at the crossing of the ribbon surface's boundary.

Figure~\ref{fig:ribbon_example0} shows an example of this construction applied to a Legendrian graph $\G$.

\begin{figure}[htbp]
    \centering
    \begin{overpic}[scale=0.9]{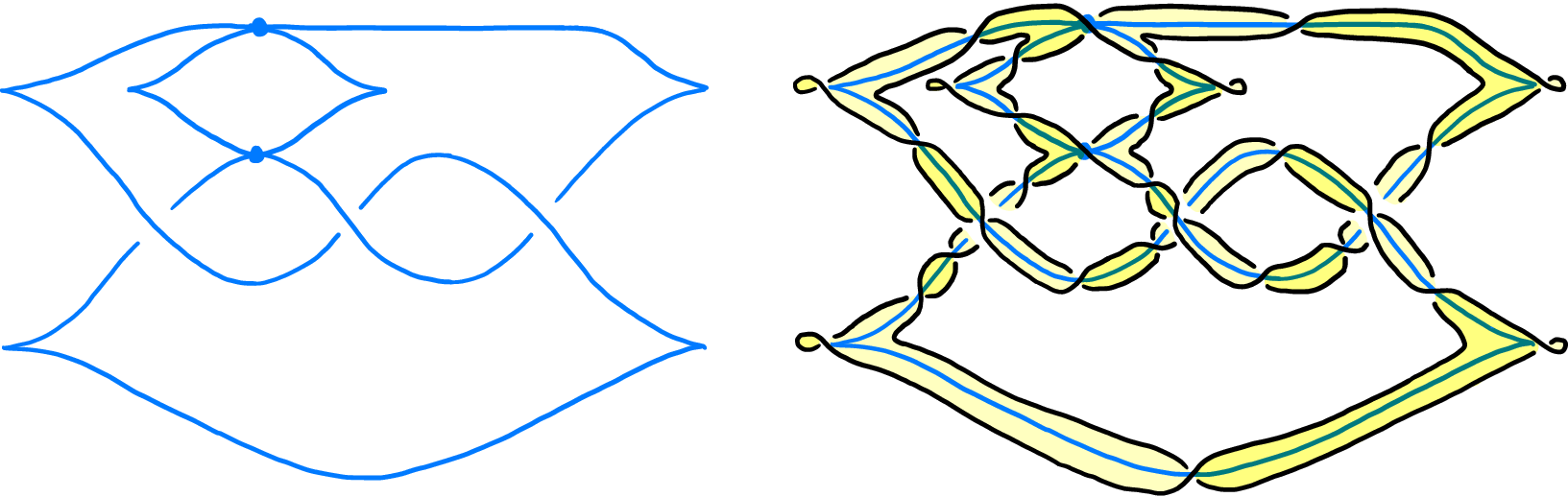}
    \end{overpic}
    \caption{Left: a Legendrian graph. Right: a ribbon surface of the graph obtained by the local procedure above.}\label{fig:ribbon_example0}
\end{figure}

Note that, in particular, we assume that the surfaces created in this way are generic ribbon surfaces.

The elliptic and hyperbolic points in the $1$--handles satisfy the hypotheses of the Elimination Lemma (see \cite[Lemma~2.15]{Giroux00}\footnote{See \cite[Lemma 31]{Massot14} for an English version.}), so they can be canceled in pairs by a small perturbation of~$\Sigma$. The local picture is shown in Figure~\ref{fig:ribbon_simplif_local}.

\begin{figure}[htbp]
    \centering
    \begin{overpic}[scale=1]{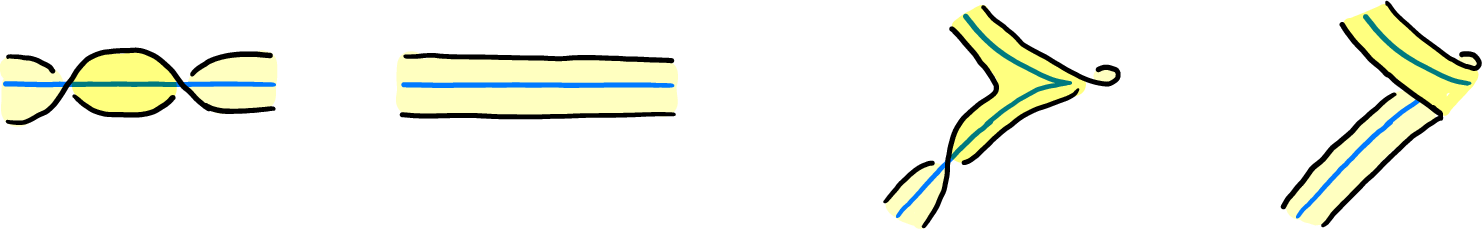}
    \end{overpic}
    \caption{Local simplifications of the ribbon via elliptic–hyperbolic cancellation. Horizontal and vertical reflections are also allowed (with crossings and shading adjusted).}\label{fig:ribbon_simplif_local}
\end{figure}

A global example is given in Figure~\ref{fig:ribbon_simpli_ex}. In each $1$--handle at least one hyperbolic singularity must remain.

\begin{figure}[htbp]
    \centering
    \begin{overpic}[scale=0.9]{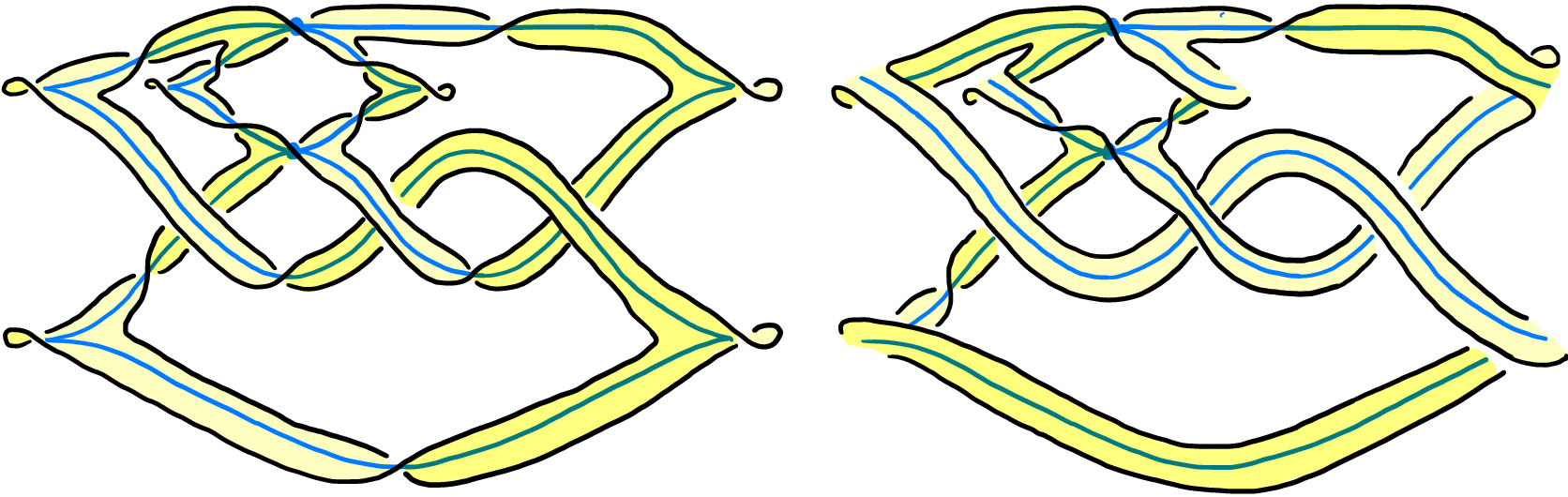}
    \end{overpic}
    \caption{Simplifications of the ribbon in Figure~\ref{fig:ribbon_example0}. Left: after canceling elliptic points near crossings. Right: after also canceling elliptic points near cusps.}\label{fig:ribbon_simpli_ex}
\end{figure}

\subsubsection{Legendrian realization for ribbon surfaces}

The Legendrian realization principle (Theorem~\ref{thm:Legendrian_realization}) gives a sufficient condition (being nonisolating) for a graph on a convex surface to be Legendrian realizable. While this condition is well known to be necessary, we provide a proof in this specific context for the sake of completeness. To this end, note that since a ribbon surface $\Sigma$ is a convex surface whose dividing set coincides with its boundary, a curve $\aa$ in $\Sigma$ is nonisolating if and only if it is homologically nontrivial.\footnote{In this article, we prefer to phrase results using the term ``homologically nontrivial'' rather than ``nonisolating,'' as it is independent of the choice of contact vector field that makes the ribbon convex.}

\begin{lemma}\label{lem:inverse_Leg_real}
    Let $\L$ be a Legendrian knot in a contact manifold $(M,\xi)$. Suppose that $\L$ is embedded in a convex surface $\Sigma$ whose dividing set $\Gamma_\Sigma$ coincides with the boundary $\partial\Sigma$. Then $\L$ is homologically nontrivial in~$\Sigma$.
\end{lemma}

\begin{proof}
    Assume by contradiction that $\L$ is homologically trivial. Then $\L$ bounds a subsurface $\widetilde{\Sigma}\subset\Sigma$. In a vertically invariant neighborhood $\Sigma\times(-\varepsilon,\varepsilon)$ we may write $\xi=\ker(\beta+u\,ds)$, with $\beta\in\Omega^1(\Sigma)$. Since $\L$ is Legendrian, we have
    \[
    \int_\L \beta = 0.
    \]
    On the other hand, the contact condition and the assumption that $u>0$ on $\widetilde{\Sigma}$ imply that $d\beta$ is a positive area form on~$\widetilde{\Sigma}$, so
    \[
    \int_{\widetilde{\Sigma}} d\beta > 0.
    \]
    This contradicts Stokes' theorem, which would give
    \[
    0 = \int_\L\beta = \int_{\widetilde{\Sigma}} d\beta.
    \]
\end{proof}

Thus homological nontriviality is a necessary condition. For our algorithm it is convenient to use the following equivalent combinatorial characterization.

\begin{lemma}\label{lem:odd_number}
    Let $\Sigma$ be a compact orientable surface with non-empty boundary, equipped with a handle decomposition. A simple closed curve $\aa\subset\Sigma$ is homologically nontrivial if and only if there exists a $1$--handle whose cocore intersects $\aa$ an odd number of times.
\end{lemma}

\begin{proof}
    First, we note that a simple closed curve on a surface with nonempty boundary is trivial in $H_1(\Sigma,\partial\Sigma;\Z)$ if and only if it is trivial in $H_1(\Sigma,\partial\Sigma;\Z_2)$. 

    Consider the algebraic intersection pairing
    \[
    \langle\cdot,\cdot\rangle\colon
    H_1(\Sigma,\partial\Sigma;\Z_2)\times H_1(\Sigma;\Z_2)\to\Z_2,
    \]
    which counts transverse intersection points modulo~$2$.\footnote{By Lefschetz duality this is dual to the nondegenerate pairing $H^1(\Sigma;\Z_2)\times H^1(\Sigma,\partial\Sigma;\Z_2)\to H^2(\Sigma,\partial\Sigma;\Z_2)\cong\Z_2$ induced by the cup product.}

    Suppose first that there is a cocore $\dd$ of a $1$--handle intersecting $\aa$ transversely in an odd number of points. Then $[\aa]\in H_1(\Sigma;\Z_2)$ and $[\dd]\in H_1(\Sigma,\partial\Sigma;\Z_2)$ satisfy
    \[
    \langle[\aa],[\dd]\rangle = 1,
    \]
    so $[\aa]\neq 0$.

    Conversely, assume that all cocores $\dd_i$ of $1$--handles intersect $\aa$ in an even number of points. Since the pairing is nondegenerate and $[\aa]\neq 0$, there exists some properly embedded arc $\dd$ with $\langle[\aa],[\dd]\rangle=1$. We can express $[\dd]$ as a $\Z_2$–linear combination
    \[
    [\dd] = \sum_i c_i[\dd_i],\qquad c_i\in\Z_2,
    \]
    of cocores. If every $\langle[\aa],[\dd_i]\rangle$ were zero, then
    \[
    1
    = \langle[\aa],[\dd]\rangle
    = \Big\langle[\aa],\sum_i c_i[\dd_i]\Big\rangle
    = \sum_i c_i \langle[\aa],[\dd_i]\rangle
    = 0,
    \]
    a contradiction. Hence some cocore intersects $\aa$ an odd number of times.
\end{proof}

\section{Uniqueness of the Legendrian realization}

In this section, we drop the assumption that we are working in $(\R^3,\xist)$. This section is not technically needed for the later parts of the article, but it complements the algorithm and the results here will also be used in forthcoming work of the author with collaborators on a contact version of Kirby's theorem~\cite{KegelStenhedeVertesiZuddasKirby}.

All the necessary ingredients already appear in \cite{Giroux02,Etnyre06,MR2557137}; what is missing is a brief analysis of how two isotopies of a contact manifold that induce the same path of contact structures can differ. This analysis forms the core of the second part of the proof of Theorem~\ref{thm:uniqueness_ribbon}.

\subsubsection{Uniqueness of the Legendrian realization for ribbon surfaces}

\begin{definition}\label{def:leg_real}
    Let $\Sigma$ be a compact, oriented surface. Let $\xi$ be a vertically invariant contact structure on $\Sigma\times I$, where either $I=[-\varepsilon,\varepsilon]$ or $I=\R$.
    
    We call an isotopy $\phi_t$ of $\Sigma=\Sigma\times\{0\}$ in $\Sigma\times I$ \emph{admissible} if $\phi_t(\Sigma)$ is always transverse to $\partial_z$ and is the identity on a neighborhood of $\partial\Sigma$ (if $\partial\Sigma\neq\emptyset$).
    
    Let $\aa$ be the isotopy class of a simple closed curve on $\Sigma$. We say that a Legendrian knot $L\subset\Sigma\times I$ is a \emph{Legendrian realization} of $\aa$ if there is an admissible isotopy $\phi_t$ such that $L\subset\phi_1(\Sigma)$ and the isotopy class of $\phi_1^{-1}(L)$ on $\Sigma$ coincides with $\aa$. 
    
    When we say that $(L,\phi_t)$ is a Legendrian realization of $\aa$, we are assuming we have a pair where $L$ is a Legendrian realization of $\aa$ and $\phi_t$ is an admissible isotopy of $\Sigma$ with the properties just described.
\end{definition}

The result of the Legendrian realization principle (Theorem~\ref{thm:Legendrian_realization}) is indeed a Legendrian realization of $\aa$, justifying the name.

\begin{theorem}\label{thm:uniqueness_ribbon}
    Let $\Sigma$ be a compact, oriented surface with $\partial\Sigma\neq\emptyset$. Let $\xi=\ker\alpha$ be a vertically invariant contact structure on $\Sigma\times I$, where $I=[-\varepsilon,\varepsilon]$ or $I=\R$, for which $\partial_z$ is a Reeb vector field, and assume $\partial\Sigma$ (identified with $\partial(\Sigma\times\{0\})$) is a positive transverse link.
    
    Then any two Legendrian realizations of $\aa$ are Legendrian isotopic.
\end{theorem}

\begin{remark}
    Let $\Sigma$ be a compact, oriented surface in a contact manifold $(M,\xi)$. Assume that $\Sigma$ has nonempty positive transverse boundary and is transverse to a Reeb vector field $R_\alpha$ (for example, $\Sigma$ is a ribbon surface). Consider a neighborhood $\Sigma\times[-\varepsilon,\varepsilon]$ obtained by flowing $\Sigma$ using $R_\alpha$. Then Lemma~\ref{lem:ribbon_convex} shows that $\Sigma$ is also convex with respect to a contact vector field $X$ that makes the dividing set $\Gamma_\Sigma$ equal to $\partial \Sigma$. 
    
    Moreover, by carefully analyzing the proof of Lemma~\ref{lem:ribbon_convex} we can see that we are free to choose any small enough collar neighborhood $A$ of $\partial\Sigma$ and then construct $X$ so that 
    \begin{itemize}
        \item $X$ is vertically invariant;
        \item $X_p\in T_p(\partial\Sigma\times[-\varepsilon,\varepsilon])$ for all $p\in\partial\Sigma$;
        \item $X=\partial_z$ on $(\Sigma\setminus A)\times\R$. 
    \end{itemize}
    This in particular tells us that, given a neighborhood of $\Sigma$ obtained by flowing $\Sigma$ with $X$, it contains a neighborhood of $\Sigma$ obtained by flowing $\Sigma$ with $R_\alpha$ and vice versa. Moreover, given a Legendrian realization $(L,\phi_t)$ with respect to $R_\alpha$, we can find such a contact vector field $X$ so that $(L,\phi_t)$ is also a Legendrian realization for $X$.

    We can probably say much more about this relation, but since we do not need it we stop here.
\end{remark}

Before proving Theorem \ref{thm:uniqueness_ribbon} we need a couple more definitions and a Lemma.

We say that an isotopy $\phi_t$ of $(\Sigma\times\R,\xi)$ is a \emph{Moser isotopy} if it is obtained by applying the Moser trick to a family $\alpha_t$ of contact forms defining $\xi_t:=T\phi_t^{-1}(\xi)$. More precisely, we are assuming that the Moser trick applied to $\alpha_t$ gives an isotopy $\widetilde{\phi}_t$ such that $\xi_t=T\widetilde{\phi}_t(\xi_0)$, and we are imposing that $\widetilde{\phi}_t=\phi_t^{-1}$.

It is important to note that not all isotopies are Moser isotopies. For example, a contact isotopy $\phi_t$ of $\xi$ will never be of this form, as in this case the path $\xi_t$ would be constant and so $\widetilde{\phi}_t=\mathrm{Id}$.

We say that a contact form $\alpha$ is \emph{compatible} with $\Sigma\times \R$ if 
\begin{enumerate}
    \item $\alpha$ is vertically invariant;
    \item $R_\alpha$ is positively transverse to $\Sigma$;
    \item $\partial\Sigma$ is a positive transverse link for $\ker\alpha$.
\end{enumerate}

If we are in the hypothesis of Theorem \ref{thm:uniqueness_ribbon}, then the contact form $\alpha$ (whose Reeb vector field is $\partial_z$) is clearly compatible with $\Sigma\times I$.

The following is an obvious adaptation of Etnyre's proof of Giroux's intuition that two contact structures compatible with the same open book are isotopic \cite{Giroux02,Etnyre06}. More precisely, the space $C$ of contact forms compatible with an open book is connected and simply connected, and the same holds in our case.

\begin{lemma}\label{lem:Ccontra}
Let $\Sigma$ be an oriented, compact surface with $\partial\Sigma\neq\emptyset$. The space $C$ of contact forms compatible with $\Sigma\times\R$ is connected and simply connected.
\end{lemma}

\begin{proof}
Since we only consider $z$–invariant forms, any $\alpha\in C$ can be written
uniquely as
\[
  \alpha = \lambda + h\,dz,
\]
where $\lambda$ is the pullback of a $1$–form on $\Sigma$ and $h:\Sigma\to\R$ is a
function, both independent of $z$.

\medskip\noindent
\textbf{Step 1. Characterization of $C$.}

For a contact form $\alpha$, the following are clearly equivalent:
\begin{enumerate}
  \item[(a)] $\alpha$ is a positive contact form and $R_\alpha$ is positively
  transverse to $\Sigma$.
  \item[(b)] $\alpha$ is a positive contact form and $d\alpha|_{T(\Sigma)}$
  is a positive area form on each slice.
\end{enumerate}

This gives a different way of characterizing forms in $C$.

\medskip\noindent
\textbf{Step 2. The ray $\alpha+t\,dz$ stays inside $C$.}

If $\alpha\in C$ and $t\ge0$, then $\alpha_t:=\alpha+t\,dz$ also lies in $C$.

Indeed, vertical invariance is clear. Since $d(\alpha_t)=d\alpha$, we still have
$d\alpha_t|_{T\Sigma}=d\alpha|_{T\Sigma}$, which is a positive area form because $d\alpha$ is.

On $\partial\Sigma$, a positive tangent vector $v$ is horizontal, so $dz(v)=0$ and
\[
  \alpha_t(v)=\alpha(v)>0.
\]
Thus the boundary remains a positive transverse link.

To check the contact condition, compute
\[
  \alpha_t\wedge d\alpha_t
  = \alpha_t\wedge d\alpha
  = \alpha\wedge d\alpha + t\,dz\wedge d\alpha.
\]
The first term on the right is clearly positive since $\alpha$ is a contact form. The second term is also nonnegative, since $d\alpha$ is a volume form on the pages, $dz$ vanishes on the pages, and $dz$ is positive on the oriented normal to the pages. Hence $\alpha_t\wedge d\alpha_t$ is a positive volume form for all $t\ge 0$.

So for each $\alpha\in C$, the ray $\{\alpha+t\,dz\mid t\ge0\}$ lies entirely in $C$.

\medskip\noindent
\textbf{Step 3. Convex combinations after adding a large multiple of $dz$.}

Let $\alpha_0,\alpha_1\in C$ and write
\[
  \alpha_i = \lambda_i + h_i\,dz,\qquad i=0,1,
\]
with $z$–independent $\lambda_i$ and $h_i$.

Fix $R>0$ and set
\[
  \beta_i := \alpha_i + R\,dz = \lambda_i + (h_i+R)\,dz,\qquad i=0,1.
\]
By Step~2, $\beta_i\in C$ for all $R\ge0$.

For $s\in[0,1]$, define the convex combination
\[
  \beta_s := (1-s)\beta_0 + s\beta_1
           = \lambda_s + (h_s+R)\,dz,
\]
where
\[
  \lambda_s := (1-s)\lambda_0 + s\lambda_1,\qquad
  h_s := (1-s)h_0 + s h_1.
\]

We now show that for $R$ sufficiently large, every $\beta_s$ lies in $C$.

\smallskip\noindent
\emph{Boundary condition.}
Along $\partial\Sigma$ we have $dz(v)=0$ for tangent vectors $v$, so
\[
  \beta_s(v) = (1-s)\alpha_0(v) + s\alpha_1(v),
\]
which is positive since each $\alpha_i(v)>0$. Thus the boundary remains a positive
transverse link for all $s$.

\smallskip\noindent
\emph{Positivity of $d\beta_s$.}
On $\Sigma$, we have
\[
  d\beta_s|_{T\Sigma} = d\lambda_s = (1-s)d\lambda_0 + s d\lambda_1.
\]
Each $d\lambda_i$ is a positive area form, so their convex combination $d\lambda_s$ is
also a positive area form. In particular, there exists a smooth positive function
$f_s$ such that
\[
  d\lambda_s = f_s\,\omega
\]
for some fixed background area form $\omega$ on $\Sigma$. By compactness of
$\Sigma\times[0,1]$, there exists $m>0$ with $f_s\ge m$ everywhere, so
\[
  d\lambda_s \ge m\,\omega
\]
pointwise for all $s$.

\smallskip\noindent
\emph{Contact condition for large $R$.}
Compute
\[
  d\beta_s = d\lambda_s + dh_s\wedge dz,
\]
and hence
\[
  \beta_s\wedge d\beta_s
  = \lambda_s\wedge d\lambda_s
    + \lambda_s\wedge dh_s\wedge dz
    + (h_s+R)\,dz\wedge d\lambda_s.
\]
Write this as
\[
  \beta_s\wedge d\beta_s = A_s + (h_s+R)\,B_s,
\]
where $B_s:=dz\wedge d\lambda_s$ is a positive volume form (since $d\lambda_s$ is a
positive area form) and $A_s:=\lambda_s\wedge d\lambda_s+\lambda_s\wedge dh_s\wedge dz$.

Consider the volume form $\Omega=\omega\wedge dz$ on $\Sigma\times\R$. By compactness of
$\Sigma$ and the fact that everything is vertically invariant, there exists $C>0$ such that
\[
  |A_s|\le C \Omega
\]
for all $s$.\footnote{This notation is not really valid, because in general there is no total order on volume forms, but it is clear what it means in this case.} Since $h_s$ is a convex combination of the two fixed functions $h_0,h_1$,
there exists $H>0$ with $|h_s|\le H$ for all $s$.

On the other hand $d\lambda_s\ge m\omega$ implies
\[
  B_s = dz\wedge d\lambda_s \ge m\,\omega\wedge dz = m\,\Omega.
\]
Therefore
\[
  (h_s+R)\,B_s \ge (R-H)\,m\,\Omega.
\]

Combining the estimates, we get
\[
  \beta_s\wedge d\beta_s
  \ge -C\,\Omega + (R-H)m\,\Omega
  = \bigl((R-H)m - C\bigr)\,\Omega.
\]
If we choose $R$ so large that
\[
  (R-H)m - C > 0,
\]
then $\beta_s\wedge d\beta_s$ is a positive volume form everywhere for all $s\in[0,1]$.
Thus each $\beta_s$ is a positive contact form.

We have already seen that $d\beta_s|_{T\Sigma}$ is a positive area form, so by the
characterization above the Reeb field $R_{\beta_s}$ is positively transverse to
$\Sigma$. Together with the boundary condition, and the fact that all these contact forms are obviously vertically invariant, this shows $\beta_s\in C$
for all $s\in[0,1]$.

In particular, for any $\alpha_0,\alpha_1\in C$ there exists $R\ge0$ such that the path
\[
  s\longmapsto (1-s)(\alpha_0+Rdz)+s(\alpha_1+Rdz)
\]
lies entirely in $C$. This is exactly what we have just proved.

\medskip\noindent
\textbf{Step 4. Path–connectedness of $C$.}

Let $\alpha_0,\alpha_1\in C$. Choose $R\ge0$ as in Step~3 and set
$\beta_i=\alpha_i+Rdz$.

Consider the following concatenated path in $C$:
\begin{enumerate}
  \item From $\alpha_0$ to $\beta_0$ along the ray
  \[
    t\longmapsto \alpha_0+tR\,dz,\qquad t\in[0,1],
  \]
  which lies in $C$ by Step~2.
  \item From $\beta_0$ to $\beta_1$ along the convex path
  \[
    s\longmapsto \beta_s=(1-s)\beta_0+s\beta_1,\qquad s\in[0,1],
  \]
  which lies in $C$ by Step~3.
  \item From $\beta_1$ to $\alpha_1$ along
  \[
    t\longmapsto \alpha_1 + (1-t)R\,dz,\qquad t\in[0,1],
  \]
  again in $C$ by Step~2.
\end{enumerate}
Thus any two points of $C$ can be joined by a path in $C$, so $C$ is path–connected.

\medskip\noindent
\textbf{Step 5. Simple connectedness of $C$.}

Let $\{\alpha_\theta\}_{\theta\in S^1}\subset C$ be a continuous loop. We will show it is
null–homotopic in $C$.

\smallskip\noindent
\emph{Stage 1: push the loop in the $dz$–direction.}
For each $\theta\in S^1$ and $t\in[0,1]$ define
\[
  \alpha_{\theta,t} := \alpha_\theta + tR\,dz.
\]
By Step~2, $\alpha_{\theta,t}\in C$ for all $(\theta,t)$. For $t=0$ we recover the
original loop, and for $t=1$ we get the loop
\[
  \theta\longmapsto \beta_\theta:=\alpha_\theta+Rdz.
\]
This gives a homotopy in $C$ from $\{\alpha_\theta\}$ to $\{\beta_\theta\}$.

\smallskip\noindent
\emph{Stage 2: contract the pushed–up loop by convex combinations.}
Fix some basepoint $\theta_0\in S^1$ and consider
\[
  H(\theta,s) := (1-s)\,\beta_\theta + s\,\beta_{\theta_0},\qquad
  (\theta,s)\in S^1\times[0,1].
\]
For each fixed $\theta$, this is exactly the kind of convex path treated in
Step~3, now with parameters $(\theta,s)$.

The estimates in Step~3 depend continuously on $\theta$ and $s$, and
$S^1$ is compact. Thus there exists a single $R>0$ such that for this $R$ all forms $H(\theta,s)$
lie in $C$. For $s=0$ we recover $\beta_\theta$, and for $s=1$ we have the constant
loop $\theta\mapsto\beta_{\theta_0}$. Hence $\{\beta_\theta\}$ is null–homotopic in $C$.

Combining Stages~1 and~2, we obtain a homotopy in $C$ from the original loop
$\{\alpha_\theta\}$ to a constant loop. Thus every loop in $C$ is null–homotopic and
$\pi_1(C)=0$.

\medskip\noindent
Therefore $C$ is both connected and simply connected.
\end{proof}

A good part of the argument in the following proof comes from \cite{MR2557137}.

\begin{proof}[Proof of Theorem \ref{thm:uniqueness_ribbon}]
We need to show that two Legendrian realizations $L$ and $L'$ of $\aa$ are Legendrian isotopic. That is, we need to find an isotopy $\psi_t:S^1\rightarrow\Sigma\times I$, $t\in[0,1]$, such that 
\begin{itemize}
    \item $\psi_0(S^1)=L$;
    \item $\psi_1(S^1)=L'$;
    \item $\psi_t(S^1)$ is Legendrian for all $t$.
\end{itemize}

To achieve this, we will proceed as follows. First, if we are working on $(\Sigma\times[-\varepsilon,\varepsilon],\xi=\ker\alpha)$, then we think of this manifold as being a subset of $\Sigma\times\R$, where $\alpha$ is extended by imposing it to be vertically invariant. Next we will find an (ambient) isotopy of $\Sigma\times\R$ that sends $L$ to $L'$ through Legendrian knots. By restricting the isotopy to $L$, this gives the desired Legendrian isotopy. Of course, in the case where the original manifold was $\Sigma\times[-\varepsilon,\varepsilon]$, we need to ensure that the image of $L$ under this isotopy is always contained in $\Sigma\times[-\varepsilon,\varepsilon]$.

If $\phi_t$ is an admissible isotopy of $\Sigma$ in $\Sigma\times\R$ as in Definition \ref{def:leg_real}, we can extend it to an (ambient) isotopy, still denoted $\phi_t$, of $\Sigma\times\R$ by simply setting
\[
  \phi_t(x,z)=(\phi_t(x)_\Sigma,\phi_t(x)_\R+z),
\]
where $\phi_t(x)_\Sigma$ and $\phi_t(x)_\R$ are the coordinates of $\phi_t(x)\in\Sigma\times\R$. We will always do this and still call it $\phi_t$. (We will sometimes abuse notation and infer properties of the ambient isotopy from properties of the isotopy of $\Sigma$, while calling both by the same name. It will be clear from the context which one we are referring to.)

Note that if $\phi_t$ is an admissible isotopy (we now think of these as ambient isotopies as explained above), and we consider the family of contact forms
\[
  \alpha_t:=T\phi_t^{-1}(\alpha),
\]
then each $\alpha_t$ is compatible with $\Sigma\times\R$. This is true because:
\begin{itemize}
    \item condition (1) in the definition of compatibility is guaranteed by the fact that both $\phi_t$ and $\alpha$ are vertically invariant;
    \item condition (2) is guaranteed by the fact that $R_{\alpha_t}=T\phi_t^{-1}(\partial_z)$, and this is transverse to $\Sigma$ because $\phi_t(\Sigma)$ is transverse to $\partial_z$;
    \item condition (3) is guaranteed by the fact that each $\alpha_t$ agrees with $\alpha$ on $A\times\R$, where $A$ is a neighborhood of $\partial\Sigma$ on which $\phi_t$ is the identity.
\end{itemize}
Conversely, given a path $\alpha_t$ of contact forms compatible with $\Sigma\times\R$, we have that:
\begin{itemize}
    \item we can apply the Moser trick to them and obtain an isotopy $\phi_t$ (this is not automatic, as $\Sigma\times\R$ is not compact). It follows from the fact that the time–dependent vector field defining $\phi_t$ is vertically invariant;
    \item the isotopy $\phi_t$ is vertically invariant, i.e.
    \[
      \phi_t(x,z)=(\phi_t(x,0)_\Sigma,\phi_t(x)_\R+z);
    \]
    \item $\phi_t(\Sigma)$ is transverse to $\partial_z$.
\end{itemize}
These properties are all consequences of the fact that each $\alpha_t$ is vertically invariant.

Given all of this, it is easy to check that there is a surjection between:
\begin{itemize}
    \item the space of paths $\alpha_t$ of contact forms compatible with $\Sigma\times\R$
\end{itemize}
onto:
\begin{itemize}
    \item the space of vertically invariant Moser isotopies $\phi_t$ such that $\phi_t(\Sigma)$ is transverse to $\partial_z$.
\end{itemize}
It is then clear why having a control over the first space (given by Lemma~\ref{lem:Ccontra}) has consequences over the second space.

In the following, every time we say that an isotopy $\phi_t$ is a Moser isotopy, we are also implicitly assuming that the Moser trick was applied to the specific path of contact forms $\alpha_t:=T\phi_t^{-1}(\alpha)$ (in principle, two different paths of contact forms for the same path of contact structures might yield different isotopies).

\medskip\noindent
\textbf{(a)} Assume first that $(L,\phi_t)$ and $(L',\phi_t')$ are two Legendrian realizations of $\aa$ such that $\phi_t$ and $\phi_t'$ are Moser isotopies. Then $L$ and $L'$ are Legendrian isotopic. Call $\alpha_t$ and $\alpha_t'$ the paths of contact forms associated to $\phi_t$ and $\phi_t'$. Note that $\alpha_0=\alpha_0'=\alpha$.

The proof of this claim follows as in the proof of \cite[Theorem~2.7]{MR2557137}. The key ingredient is Lemma~\ref{lem:Ccontra}. 

The steps of the proof are the following. We do not give the details as they are identical to those in the reference given. We, however, stress the key ingredients so that one can check that everything works in our setting.

\smallskip\noindent
(a.1) Let $a$ be a curve whose isotopy class is $\aa$, and let $\phi_t$ and $\phi_t'$ be two admissible Moser isotopies such that $\phi_1(a)$ and $\phi_1'(a)$ are Legendrian.

\smallskip\noindent
(a.1.1) Assume that $\alpha_1=\alpha_1'$ (so, in particular, we are assuming that $\xi_1=\xi_1'$). Then $\phi_1(a)$ is Legendrian isotopic to $\phi_1'(a)$. This follows from the fact that $C$ is simply connected and the fact that the contact forms in the paths $\alpha_t$ and $\alpha_t'$ are all compatible with $\Sigma\times\R$. More precisely, there exists a path $\alpha_t^s$ of paths such that
\begin{itemize}
    \item $\alpha_t^0=\alpha_t$,
    \item $\alpha_t^1=\alpha_t'$,
    \item $\alpha_0^s=\alpha$, and
    \item $\alpha_1^s=\alpha_1=\alpha_1'$.
\end{itemize}
This gives rise to a path of Moser isotopies that achieves what we want.

\smallskip\noindent
(a.1.2) If $\alpha_1\neq\alpha_1'$, the proof that $\phi_1(a)$ is Legendrian isotopic to $\phi_1'(a)$ uses the fact that there exists a path $\widetilde{\alpha}_t$ of contact forms compatible with $\Sigma\times\R$ such that 
\begin{itemize}
    \item $\widetilde{\alpha}_0=\alpha_1$;
    \item $\widetilde{\alpha}_1=\alpha_1'$;
    \item the curve $a$ is Legendrian for all the $\widetilde{\alpha}_t$.
\end{itemize}
This is shown in \cite[Lemma~2.9]{MR2557137}.

In practice, one considers the path from $\alpha$ to $\alpha_1'$ obtained by concatenating a path from $\alpha$ to $\alpha_1$ with $\widetilde{\alpha}_t$. The Moser isotopy associated to this path first makes $a$ into $\phi_1(a)$ and then transforms $\phi_1(a)$ by a Legendrian isotopy to a knot Legendrian isotopic to $\phi_1'(a)$ (by step (a.1.1)).

\smallskip\noindent
(a.2) Let $a_t$ be a path of curves all representing the homology class determined by $\aa$. Assume that $\phi_t$ and $\phi_t'$ are admissible Moser isotopies such that $\phi_1(a_0)$ and $\phi_1'(a_1)$ are Legendrian. Then $\phi_1(a_0)$ and $\phi_1'(a_1)$ are Legendrian isotopic. 

The proof of this relies on the fact that there exists a path of contact forms $\widetilde{\alpha}_t$ compatible with $\Sigma\times\R$ such that $a_t$ is Legendrian for $\widetilde{\alpha}_t$ \cite[Lemma~2.8]{MR2557137}. We can now consider a path $\alpha_t$ from $\alpha$ to $\alpha_1=\widetilde{\alpha}_0$ concatenated with the path $\widetilde{\alpha}_t$. The Moser isotopy associated to this path first makes $a_0$ into a Legendrian knot $L$ that is Legendrian isotopic to $\phi_1(a_0)$ by (a.1.1), and then Legendrian isotopes $L$ to a Legendrian knot that is Legendrian isotopic to $\phi_1'(a_1)$ again by (a.1.1), proving the theorem in this case.

\medskip\noindent
\textbf{(b)} Next assume that $(L,\phi_t)$ is any type of Legendrian realization (that is, we are not assuming that $\phi_t$ is a Moser isotopy). We will show that $L$ is Legendrian isotopic to a Legendrian realization coming from a Moser isotopy, thereby reducing the general case to (a).

Consider the path of contact forms
\[
  \alpha_t := T\phi_t^{-1}(\alpha)
\]
and the corresponding path of contact structures $\xi_t:=\ker\alpha_t$ on $\Sigma\times\R$. Let us write $\xi_0=\xi$.

By the Moser trick, there exists a Moser isotopy
\[
  \widetilde{\phi}_t : \Sigma\times\R \to \Sigma\times\R,
  \qquad \widetilde{\phi}_0 = \mathrm{id},
\]
such that
\[
  T\widetilde{\phi}_t^{-1}(\xi_0) = \xi_t
  \quad\text{for all } t.
\]

Recall that $\widetilde{\phi}_t$ exists because each $\alpha_t$ is vertically invariant and so the time–dependent vector field given by the Moser trick is also vertically invariant and thus can be integrated. Moreover, note that $\widetilde{\phi}_t$ is admissible because the path $\alpha_t$ is compatible with $\Sigma\times\R$ and all $\alpha_t$ coincide with $\alpha$ on $A\times\R$, where $A$ is a neighborhood of $\partial\Sigma$ on which $\phi_t$ is the identity. This last condition ensures that $\widetilde{\phi}_t$ is the identity on $A\times\R$ for all $t$.

By definition of $\xi_t$ and of $\widetilde{\phi}_t$, we have
\[
  T \phi_t^{-1}(\xi_0) = \xi_t = T \widetilde{\phi}_t^{-1}(\xi_0)
  \quad\text{for all } t\in[0,1].
\]
The “difference” isotopy
\[
  F_t := \widetilde{\phi}_t \circ \phi_t^{-1}
\]
is then a contact isotopy of $(\Sigma\times\R,\xi_0)$, because
\[
  T F_t(\xi_0)
  = T(\widetilde{\phi}_t\circ \phi_t^{-1})(\xi_0)
  = T \widetilde{\phi}_t(\xi_t)
  = \xi_0.
\]
Hence $F_t$ is a path of contactomorphisms of $(\Sigma\times\R,\xi_0)$, with $F_0 = \mathrm{id}$.

Let $a := \phi_1^{-1}(L) \subset \Sigma$. By assumption $(L,\phi_t)$ is a Legendrian realization of $\aa$, so $a$ is a curve on $\Sigma$ representing the class $\aa$. Consider the knot
\[
  L' := \widetilde{\phi}_1(a) \subset \widetilde{\phi}_1(\Sigma).
\]
Then:
\begin{itemize}
    \item $L'$ is Legendrian because $L'=F_1(L)$ and $F_1$ is a contactomorphism. Thus $(L',\widetilde{\phi}_t)$ is a Legendrian realization of $\aa$ associated to the Moser isotopy $\widetilde{\phi}_t$.
    \item Moreover, since $F_t$ is a contact isotopy of $(\Sigma\times\R,\xi_0)$ with $F_0=\mathrm{id}$, the path $t\mapsto F_t(L)$ gives a Legendrian isotopy from $L$ to $L'$.
\end{itemize}

Thus any Legendrian realization $(L,\phi_t)$ is Legendrian isotopic to a Legendrian realization $(L',\widetilde{\phi}_t)$ whose isotopy $\widetilde{\phi}_t$ is a Moser isotopy. 

Finally, a careful (but not difficult) analysis of all these isotopies shows that in all cases the Legendrian isotopy between $L$ and $L'$ stays inside $\Sigma\times[-\varepsilon,\varepsilon]$ if we started with Legendrian realizations in this manifold.

Combined with (a), this implies that any two Legendrian realizations of $\aa$ lying in the same vertically invariant neighborhood $\Sigma\times\R$ are Legendrian isotopic.
\end{proof}

\subsubsection{Uniqueness of the Legendrian realization for open book decompositions}

Notice that the open book setting is very similar to our setting.

\begin{definition}\label{def:admiss_is_op}
    Let $(M,\xi)$ be a contact manifold and let $(B,\pi)$ be an open book compatible with $\xi$.
    
    In this setting, we call an isotopy $\phi_t$ of $M$ \emph{admissible} if
    \begin{itemize}
        \item $\phi_t$ is the identity on a neighborhood of $B$;
        \item there exists a contact form $\alpha$ for $\xi$ such that
        \begin{itemize}
            \item $\alpha$ is compatible with $(B,\pi)$ in the sense that the Reeb vector field $R_\alpha$ is always positively transverse to the pages and positively tangent to $B$;
            \item for each $z\in S^1$, the surface $\phi_t(\pi^{-1}(z))$ is always transverse to $R_\alpha$.
        \end{itemize}
    \end{itemize}

    Let $\aa$ be the isotopy class of a simple closed curve on a page $\Sigma$ of $(B,\pi)$. We say that a Legendrian knot $L\subset M\setminus B$ is a \emph{Legendrian realization} of $\aa$ if there is an admissible isotopy $\phi_t$ such that $L\subset\phi_1(\Sigma)$ and the isotopy class of $\phi_1^{-1}(L)$ on $\Sigma$ coincides with $\aa$.
\end{definition}

Let $\phi_t$ be an admissible isotopy of $(M,(B,\pi),\xi)$ as in Definition~\ref{def:admiss_is_op}, and let $\alpha$ be the corresponding contact form. We say that $\phi_t$ is a \emph{Moser isotopy} if it is obtained by applying the Moser trick to the path of contact forms
\[
  \alpha_t := T\phi_t^{-1}(\alpha).
\]

In this language, in \cite[Theorem~2.7]{MR2557137} it is proved that, given a homologically nontrivial isotopy class $\aa$ of a simple closed curve on a page $\Sigma$, we have:
\begin{itemize}
    \item there exists a Legendrian realization of $\aa$ given by a Moser isotopy;
    \item any two Legendrian realizations of $\aa$ given by a Moser isotopy are Legendrian isotopic.
\end{itemize}

The proof uses the fact that the isotopies are Moser because it uses that the space of contact forms compatible with a given open book is convex \cite{Giroux02,Etnyre06}.
    
However, a simple adaptation of the argument in part~(b) of the proof of Theorem~\ref{thm:uniqueness_ribbon}, together with the proof of \cite[Theorem 2.7]{MR2557137}, shows the following:

\begin{theorem}[Uniqueness of Legendrian knots on open books]\label{thm:uniqueness_OB}
    Let $(B,\pi)$ be an open book supporting a contact manifold $(M,\xi)$. Let $\Sigma$ be a page and let $\aa$ be the isotopy class of a homologically nontrivial simple closed curve on $\Sigma$. Then:
    \begin{itemize}
        \item There exists a Legendrian knot $L$ embedded on a page of an open book $(B',\pi')$ supporting $\xi$ such that
        \begin{itemize}
            \item $(B,\pi)$ and $(B',\pi')$ are isotopic through open books compatible with the same contact form $\alpha$ for $\xi$;
            \item if we denote by $\phi_t$ the isotopy between $(B,\pi)$ and $(B',\pi')$, then $\phi_1^{-1}(L)\subset \Sigma$ defines the same isotopy class $\aa$.
        \end{itemize}
        \item If $L'$ is another Legendrian knot that is embedded on a page of an open book $(B'',\pi'')$ such that
        \begin{itemize}
            \item $(B,\pi)$ and $(B'',\pi'')$ are isotopic through open books compatible with the same contact form $\alpha'$ for $\xi$;
            \item if we denote by $\phi_t'$ the isotopy between $(B,\pi)$ and $(B'',\pi'')$, then $\phi_1'^{-1}(L')\subset \Sigma$ defines the same isotopy class $\aa$,
        \end{itemize}
        then $L'$ is Legendrian isotopic to $L$.
    \end{itemize}
\end{theorem}

We need this stronger version (in fact, Theorem~\ref{thm:uniqueness_ribbon}) of the result in \cite{MR2557137} for later applications to contact Kirby calculus. Indeed, one may encounter two Legendrian knots that lie in admissible isotopies of the same ribbon surface without literally arising as Legendrian realizations of a fixed curve on the ribbon via a Moser isotopy. In that situation, Theorem~\ref{thm:uniqueness_ribbon} still implies that the knots are Legendrian isotopic, whereas the result of \cite{MR2557137} alone would not suffice.

\section{An algorithm to Legendrian realize a curve on a ribbon surface}

\subsection{Statement of the result}\label{section:Stat_results}

Before stating the main theorem, we formalize the notion of an ``abstract ribbon'', which allows us to separate the combinatorics on a surface from its concrete embedding in $(\R^3,\xist)$.

Let $\Sigma$ be a ribbon surface of a Legendrian graph $\G$. We write $[\Sigma]$ for the isotopy class of~$\Sigma$ through embeddings that fix $\G$ pointwise. It is straightforward to check that if $\Sigma$ and $\Sigma'$ are ribbon surfaces of the same Legendrian graph $\G$, then $[\Sigma]=[\Sigma']$. This leads to:

\begin{definition}\label{def:abstract_ribbon}
    The \emph{abstract ribbon} $\A$ of a Legendrian graph $\G$ is the class $[\Sigma]$, where $\Sigma$ is any ribbon surface of~$\G$.
\end{definition}

The abstract ribbon inherits a well-defined handle decomposition from any representative $\Sigma$, since the handle decomposition only depends on~$\G$.

We can now state the main result.

\begin{theorem}\label{thm:algorithm}
    There exists an algorithm with the following input and output:

    \medskip\noindent
    \textbf{Input:}
    \begin{itemize}
        \item a generic Legendrian graph $\G\subset(\R^3,\xist)$, given by its front projection;
        \item a homologically nontrivial simple closed curve $\aa$ on the abstract ribbon $\A$ of~$\G$.
    \end{itemize}

    \noindent
    \textbf{Output:} the front projection of a generic Legendrian knot $\L$ which is a Legendrian realization of~$\aa$ in the following sense. There exists a generic ribbon surface $\Sigma$ of~$\G$ and an isotopy $\psi_t\colon\Sigma\to\R^3$, $t\in[0,1]$, such that:
    \begin{enumerate}
        \item $\psi_0$ is the inclusion $\Sigma\subset\R^3$, and $\psi_t$ is fixed on a neighborhood of~$\partial\Sigma$;
        \item each surface $\psi_t(\Sigma)$ is transverse to~$\partial_z$ and hence convex;
        \item for every $\varepsilon>0$, the algorithm can be carried out so that
        \[
        \mathrm{Im}\psi_t\subset\Sigma\times[-\varepsilon,\varepsilon]\qquad\text{for all }t\in[0,1],
        \] where the second factor is parametrized by the $z$-coordinate.
        \item $\L$ is a Legendrian knot properly embedded in $\psi_1(\Sigma)$;
        \item the curve $\psi_1(\aa)$ is smoothly isotopic to $\L$ in the surface $\psi_1(\Sigma)$.
    \end{enumerate}
\end{theorem}

We first describe the algorithm step by step, and then justify that it produces the Legendrian realization claimed in Theorem~\ref{thm:algorithm}.

\begin{remark}
    By Theorem \ref{thm:uniqueness_ribbon}, $L$ is the unique, up to Legendrian isotopy, Legendrian realization of the isotopy class of $\aa$.
\end{remark}

\subsection{Description of the algorithm}

We first fix notation. The data of the algorithm consists of:
\begin{itemize}
    \item $\G$, a generic Legendrian graph in $(\R^3,\xist)$, described by its front projection;
    \item $\A$, the abstract ribbon of~$\G$;
    \item $\aa$, a homologically nontrivial simple closed curve on~$\A$.
\end{itemize}

As recalled above, the ribbon surfaces of~$\G$ (and hence the abstract ribbon~$\A$) come with a natural handle decomposition. We will denote a handle, either $0$-- or $1$--handle, by~$H$.

If $H$ is a $1$--handle, we call $\G\cap H$ the \emph{Legendrian core} of~$H$ and denote it by~$\boldsymbol{e}$.\footnote{Here the letter $\boldsymbol{e}$ stands for ``edge'', even though in general $\boldsymbol{e}$ is only a subinterval of an edge of~$\G$.} If $H$ is a $0$--handle, we call $\G\cap H$ the \emph{Legendrian skeleton} of~$H$, consisting of several Legendrian segments meeting at the unique vertex~$v$ contained in~$H$; we usually denote these segments by $\boldsymbol{e}_1,\dots,\boldsymbol{e}_n$. See Figure~\ref{fig:notation_algorithm}. Note that in this example, $\boldsymbol{e_1} \cup \boldsymbol{e} \cup \boldsymbol{e_3}$ forms an edge of $\G$.

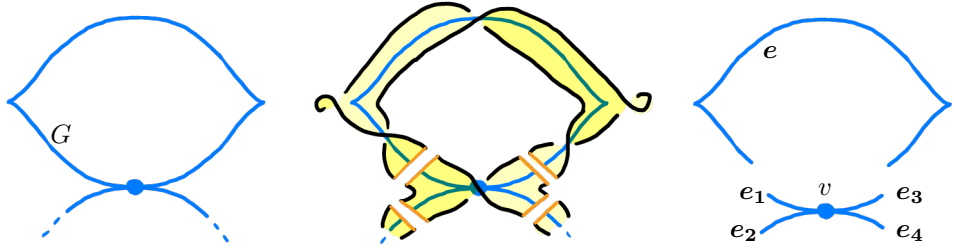
\begin{figure}[htbp]
    \centering
    \begin{overpic}[scale=1]{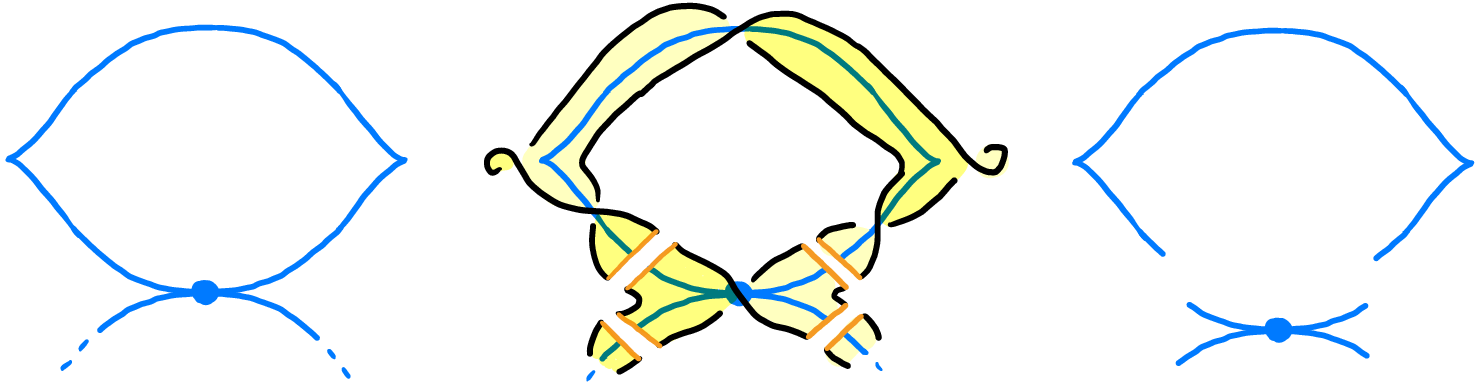}
        \put(5,11){$\G$}
        \put(85.8,5.5){$v$}
        \put(80,20){$\boldsymbol{e}$}
        \put(77.5,5){$\boldsymbol{e_1}$}
        \put(76.5,1){$\boldsymbol{e_2}$}
        \put(94,5){$\boldsymbol{e_3}$}
        \put(94,1){$\boldsymbol{e_4}$}
    \end{overpic}
    \caption{Left: a portion of a Legendrian graph $\G$ containing a vertex $v$ and an edge. Middle: a portion of a ribbon surface subdivided into handles. Right: the Legendrian core $\boldsymbol{e}$ of the $1$--handle and the Legendrian skeleton of the $0$--handle.}\label{fig:notation_algorithm}
\end{figure}

We now make two preliminary choices:
\begin{itemize}
    \item Choose a $1$--handle of $\A$ whose cocore intersects $\aa$ an odd number of times (this exists by Lemma~\ref{lem:odd_number}). Denote this distinguished $1$--handle by~$H_\star$.
    \item Choose an orientation for the core $\boldsymbol{e}$ of each $1$--handle~$H$.
\end{itemize}

Given such a choice, each $1$--handle $H$ can be identified with the abstract handle shown in Figure~\ref{fig:identification_handle}, with the blackboard orientation agreeing with the orientation of~$H$ and the core $\boldsymbol{e}$ oriented from left to right.

\begin{figure}[htbp]
    \centering
    \begin{overpic}[scale=1]{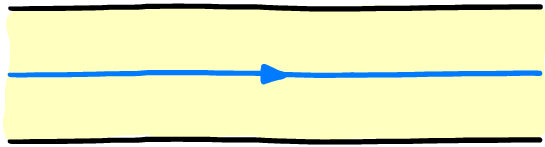}
        \put(102,11){$\boldsymbol{e}$}
    \end{overpic}
    \caption{An abstract $1$--handle, with oriented core~$\boldsymbol{e}$.}\label{fig:identification_handle}
\end{figure}

In all subsequent pictures of abstract $1$--handles we adopt this convention. For example, if a $1$--handle $H$ looks as on the left of Figure~\ref{fig:identification_handle_ex}, then the green and orange segments correspond to the green and orange segments on the right, after identifying $H$ with the abstract handle.

\begin{figure}[htbp]
    \centering
    \begin{overpic}[scale=1]{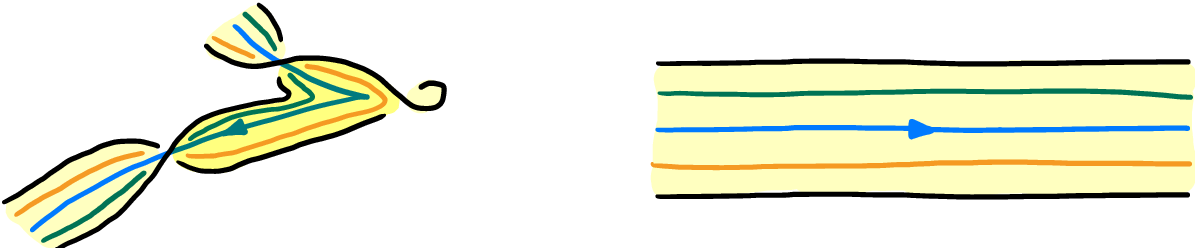}
    \end{overpic}
    \caption{Left: a $1$--handle $H$ with a green and an orange segment and Legendrian core~$\boldsymbol{e}$. Right: the corresponding segments on the abstract handle of Figure~\ref{fig:identification_handle}.}\label{fig:identification_handle_ex}
\end{figure}

\subsubsection*{\textbf{STEP 1}: subdividing the curve}

Isotope the curve $\aa$ (on the abstract ribbon~$\A$) to minimize its intersections with the attaching regions of the $1$-handles. After this isotopy, $\aa$ decomposes into a concatenation of segments, each contained in a single handle and meeting its boundary only in the attaching regions. These segments alternate between $0$-- and $1$--handles, see Figure~\ref{fig:minimize_intersections}.

\begin{figure}[htbp]
    \centering
    \begin{overpic}[scale=1]{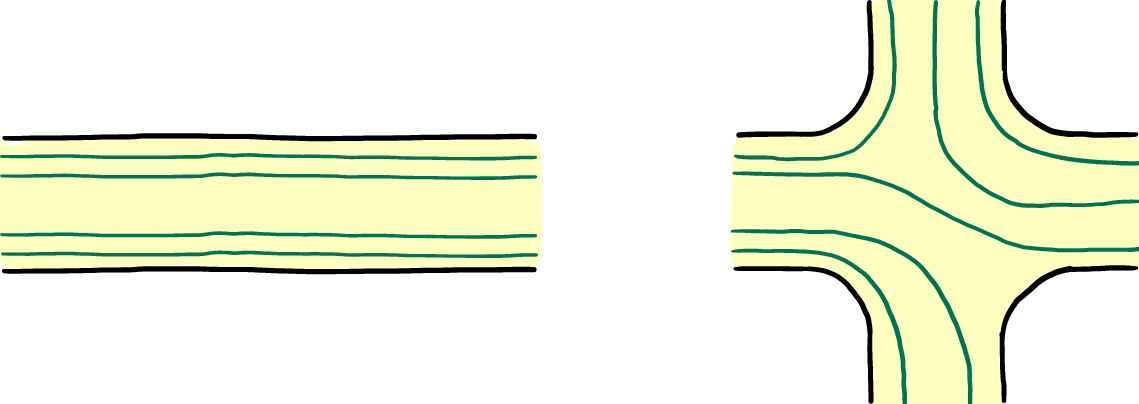}
        \put(24,15){$\boldsymbol\cdot$}
        \put(24,16.5){$\boldsymbol\cdot$}
        \put(24,18){$\boldsymbol\cdot$}
    \end{overpic}
    \caption{Left: a $1$--handle of $\A$ and the segment of $\aa$ it contains after Step~1. Right: an example of how $\aa$ may look on a $0$--handle after Step~1.}\label{fig:minimize_intersections}
\end{figure}

We now:

\begin{itemize}
    \item choose an orientation of~$\aa$;
    \item choose one segment of $\aa$ lying in a $1$--handle and call it~$\ll_1$.
\end{itemize}

Let $\ll_2$ be the next segment encountered when traveling along $\aa$ in the direction of its orientation, then $\ll_3$, and so on, until we return to~$\ll_1$. In total we obtain segments $\ll_1,\dots,\ll_n$ with $n$ even. The odd-indexed segments lie in $1$--handles, and the even-indexed segments lie in $0$--handles. Each segment inherits an orientation from~$\aa$.

\subsubsection*{\textbf{STEP 2}: relative gain (displacement index) in a 1--handle}

Let $H$ be a $1$--handle, and let $\ll_{j_1},\dots,\ll_{j_k}$ be the segments of $\aa$ contained in~$H$. 
To each segment $\ll_{j_i}$ we assign a rational number $\Delta(\ll_{j_i})$, called the 
\emph{relative gain} (or \emph{displacement index}) of~$\ll_{j_i}$.

We first define an intermediate quantity $\widetilde{\Delta}(\ll_{j_i})$. 
After identifying $H$ with the abstract handle of Figure~\ref{fig:identification_handle}, 
we may speak of the lowest segment, the second lowest, and so on.  
If $\ll_{j_i}$ is the $h$-th lowest segment, set
\[
    a(\ll_{j_i}) := h - \frac{k+1}{2}.
\]
Then define
\[
    \widetilde{\Delta}(\ll_{j_i}) :=
    \begin{cases}
        \lfloor a(\ll_{j_i}) \rfloor, & a(\ll_{j_i}) < 0,\\[0.3em]
        0,                            & a(\ll_{j_i}) = 0,\\[0.3em]
        \lceil a(\ll_{j_i}) \rceil,   & a(\ll_{j_i}) > 0.
    \end{cases}
\]

When the number of segments is odd, the lowest segment has value $-d$, the middle segment has value $0$, 
and the highest segment has value~$d$, with all other values determined by the formula above. 
If there is an even number of segments, then no segment has $\widetilde{\Delta}=0$. 
See Figure~\ref{fig:displacement_index_def1}.

\begin{figure}[htbp]
    \centering
    \begin{overpic}[scale=1]{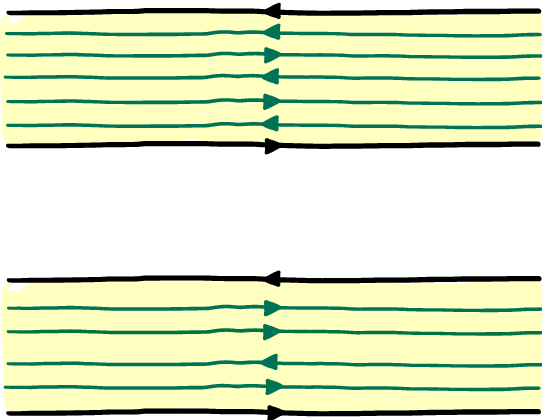}
        \put(100,69){$\leftarrow \widetilde{\Delta}=2$}
        \put(100,61){$\leftarrow \widetilde{\Delta}=0$}
        \put(100,52){$\leftarrow \widetilde{\Delta}=-2$}
        \put(-33,65){$\widetilde{\Delta}=1\rightarrow$}
        \put(-39.5,56.5){$\widetilde{\Delta}=-1\rightarrow$}
    
        \put(100,18){$\leftarrow \widetilde{\Delta}=2$}
        \put(100,8){$\leftarrow \widetilde{\Delta}=-1$}
        \put(-33,14){$\widetilde{\Delta}=1\rightarrow$}
        \put(-39.5,4){$\widetilde{\Delta}=-2\rightarrow$}
    \end{overpic}
    \caption{Top: a $1$--handle with an odd number of segments of $\aa$. Bottom: a $1$--handle with an even number of segments. The values $\widetilde{\Delta}$ are indicated.}\label{fig:displacement_index_def1}
\end{figure}
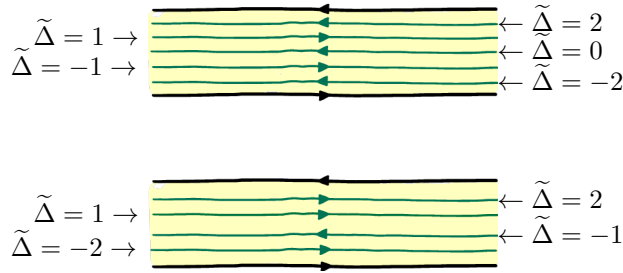

We then define the relative gain $\Delta(\ll_{j_i})$ by
\[
\Delta(\ll_{j_i}) := \delta_{j_i}\,\widetilde{\Delta}(\ll_{j_i}),
\]
where
\[
\delta_{j_i} :=
\begin{cases}
    1, & \text{if the orientation of $\ll_{j_i}$ agrees with the orientation of the core $\boldsymbol{e}$ of $H$,}\\
   -1, & \text{otherwise.}
\end{cases}
\]
For the segments in Figure~\ref{fig:displacement_index_def1}, the corresponding values of $\Delta$ are shown in Figure~\ref{fig:displacement_index_def2} (recall that the core $\boldsymbol{e}$ is oriented from left to right).

\begin{figure}[htbp]
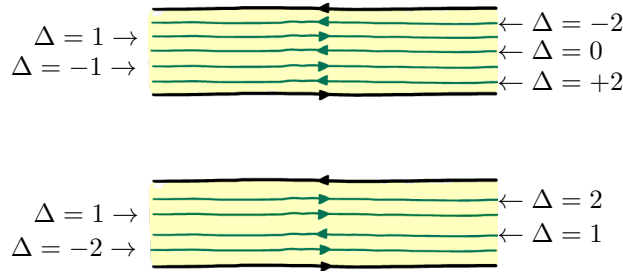

    \centering
    \begin{overpic}[scale=1]{displacement_index_def}
        \put(100,69){$\leftarrow \Delta=-2$}
        \put(100,61){$\leftarrow \Delta=0$}
        \put(100,52){$\leftarrow \Delta=+2$}
        \put(-33,65){$\Delta=1\rightarrow$}
        \put(-39.5,56.5){$\Delta=-1\rightarrow$}
        \put(100,18){$\leftarrow \Delta=2$}
        \put(100,8){$\leftarrow \Delta=1$}
        \put(-33,14){$\Delta=1\rightarrow$}
        \put(-39.5,4){$\Delta=-2\rightarrow$}
    \end{overpic}
    \caption{The relative gains (displacement indices) of the segments from Figure~\ref{fig:displacement_index_def1}.}\label{fig:displacement_index_def2}
\end{figure}

\begin{remark}
    Intuitively, $\Delta(\ll)$ measures how much the Legendrian realization of a segment $\ll$ will ``gain'' or ``lose'' height in the front projection relative to the core~$\boldsymbol{e}$. Segments further away from the center of the handle have larger $|\Delta|$, and reversing the orientation of the segment flips the sign. The precise geometric meaning of $\Delta$ will appear in the justification of the algorithm, where it will be related to the integral $\int x\,dy$ along the segment.

    The following discussion provides an informal explanation. Consider a $1$-handle $H$ of a ribbon surface $\Sigma$ of a Legendrian graph $\G$ (see, for example, Figure~\ref{fig:heuristic_delta}).

    \begin{figure}[htbp]
        \centering
        \begin{overpic}[scale=0.8]{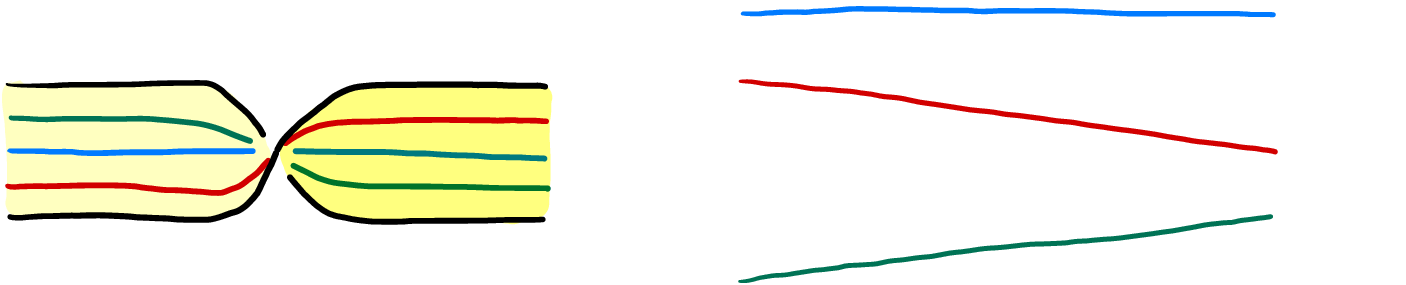}    
        \end{overpic}    
        \caption{Left: The front projection of the handle $H$ and three segments on it (one being the Legendrian core of the handle). Right: The Legendrian realization of these three segments.}\label{fig:heuristic_delta}
    \end{figure}

    Denote by $\boldsymbol{e}$ the Legendrian core of the handle $H$ (the central blue segment in Figure~\ref{fig:heuristic_delta}). This is one of the inputs to our algorithm, as we know the front projection of $\G$. Assume there are three segments of $\aa$ in this handle $H$. In Step~1, we first straighten these segments to look like the three segments (green, blue, and red) in Figure~\ref{fig:heuristic_delta} (left). The central blue segment is the Legendrian core, the red segment is ``a bit closer to us'' (in the sense that it has a higher $x$-coordinate), and the green segment is ``a bit further away from us.'' 

    The Legendrian realization of the blue segment is a copy of itself, possibly translated in the $\partial_z$ direction. Intuitively, the Legendrian realizations of the green and red segments will look similar to the blue segment but slightly rotated due to the contact condition (the $(y,z)$-slope of the contact planes decreases as $x$ increases), and possibly translated in the $\partial_z$ direction (see Figure~\ref{fig:heuristic_delta}, right). The relative gain $\Delta$ measures how much these Legendrian realizations lose or gain height relative to the blue segment when traversed in their oriented direction. For example, if all the segments are oriented from left to right, the red segment loses some height, while the green segment gains some height.

    The amount by which we need to translate these Legendrian realizations in the $\partial_z$ direction depends on the entire curve $\aa$ and will be encoded by the quantity $P$ defined later in Step~4.
\end{remark}

\subsubsection*{\textbf{STEP 3}: balancing the total gain}

Define
\[
\theta := \sum_{i=1}^n \Delta(\ll_i),
\]
where by convention $\Delta(\ll_i):=0$ if $\ll_i$ lies in a $0$--handle. Ultimately, $\theta$ will correspond to the integral $\oint_\aa x\,dy$, so we need to modify the $\Delta$’s so that $\theta$ becomes zero.

We do this by adjusting the relative gains of the segments in the distinguished handle $H_\star$. Let $\ll_{1_1},\dots,\ll_{1_k}$ be the segments of $\aa$ that lie in $H_\star$, with $k$ odd. Let $k_+$ be the number of those segments whose orientation agrees with the orientation of the core $\boldsymbol{e}$ of $H_\star$, and $k_-$ the number oriented in the opposite direction. Since $k$ is odd, $k_+\neq k_-$.

We redefine the relative gain of the segments in $H_\star$ by
\begin{equation}\label{eq:redefine_displ}
    \begin{cases}
      \Delta_{\text{old}}(\ll_{1_*}) := \Delta(\ll_{1_*}), \\[0.4em]
      \Delta(\ll_{1_*}) = \Delta_{\text{old}}(\ll_{1_*}) - \dfrac{\theta}{k_+ - k_-}, & \text{if $\ll_{1_*}$ is oriented as $\boldsymbol{e}$},\\[0.6em]
      \Delta(\ll_{1_*}) = \Delta_{\text{old}}(\ll_{1_*}) + \dfrac{\theta}{k_+ - k_-}, & \text{otherwise.}
    \end{cases}
\end{equation}

\begin{remark}\label{rk:displ_sum_zero}
Let $\{H_j\}_j$ be the collection of $1$--handles, and assume $H_1=H_\star$. For each $j$ let $k_j$ be the number of segments of $\aa$ in $H_j$, and denote these segments by $\ll_{j_1},\dots,\ll_{j_{k_j}}$. Then, after the redefinition~\eqref{eq:redefine_displ}, we have
\begin{equation}\label{eq:displ_sum_zero}
\begin{split}
\sum_{i=1}^n \Delta(\ll_i)
&= \sum_{i=0}^{\frac{n}{2}-1} \Delta(\ll_{1+2i})
 = \sum_j \sum_{i=1}^{k_j} \Delta(\ll_{j_i}) \\
&= \sum_{i=1}^{k_1}\Delta(\ll_{1_i}) + \sum_{j\neq 1}\sum_{i=1}^{k_j}\Delta(\ll_{j_i}) \\
&= \Big(-k_+\frac{\theta}{k_+-k_-} + k_-\frac{\theta}{k_+-k_-}\Big) + \sum_{i=1}^{k_1}\Delta_{\text{old}}(\ll_{1_i}) + \sum_{j\neq 1}\sum_{i=1}^{k_j}\Delta(\ll_{j_i}) \\
&= -\theta + \sum_{i=1}^{k_1}\Delta_{\text{old}}(\ll_{1_i}) + \sum_{j\neq 1}\sum_{i=1}^{k_j}\Delta(\ll_{j_i}) \\
&= -\theta + \theta = 0,
\end{split}
\end{equation}
where in the last step we used the original definition of~$\theta$. Thus after this adjustment we have forced the total sum of the relative gains to vanish.
\end{remark}

\subsubsection*{\textbf{STEP 4}: prominence}

We now introduce a discrete ``height bookkeeping'' function, the \emph{prominence}, defined at the endpoints of the segments. Informally, $P$ measures the relative height of the Legendrian realization of $\aa$ compared to the Legendrian graph~$\G$.

If $\ll_i$ is a segment, we denote its starting point by $\partial^{-}\ll_i$ 
and its endpoint by $\partial^{+}\ll_i$. Define $P$ inductively by
\[
\begin{cases}
 P(\partial^-\ll_1)=0,\\[0.3em]
 P(\partial^+\ll_i) = P(\partial^-\ll_i) + \Delta(\ll_i), & \text{if $i$ is odd (segment in a $1$--handle)},\\[0.3em]
 P(\partial^+\ll_i) = P(\partial^-\ll_i), & \text{if $i$ is even (segment in a $0$--handle)}.
\end{cases}
\]

Since $\partial^+\ll_n = \partial^-\ll_1$, the prominence is well-defined precisely because $\sum_i\Delta(\ll_i)=0$ by~\eqref{eq:displ_sum_zero}.

For segments in $0$--handles we will simply speak of ``the prominence of the segment'', since the two endpoints have the same value of~$P$.

\subsubsection*{\textbf{STEP 5}: Legendrian realization in a 1--handle}

We now explain how to Legendrian realize the segments of $\aa$ that lie in a fixed $1$--handle~$H$. We will use the front projection of the Legendrian core $\boldsymbol{e}$ as a template.

\begin{remark}
    In this step and the next, for readability we will not distinguish notationally between subsets of $\R^3$ and their front projections.
\end{remark}

Let $\ll_{j_1},\dots,\ll_{j_k}$ be the segments of $\aa$ in~$H$, and let $\boldsymbol{e}$ be the Legendrian core of~$H$ (see Figure~\ref{fig:example_1handles1}).

\begin{figure}[htbp]
    \centering
    \begin{overpic}[scale=1]{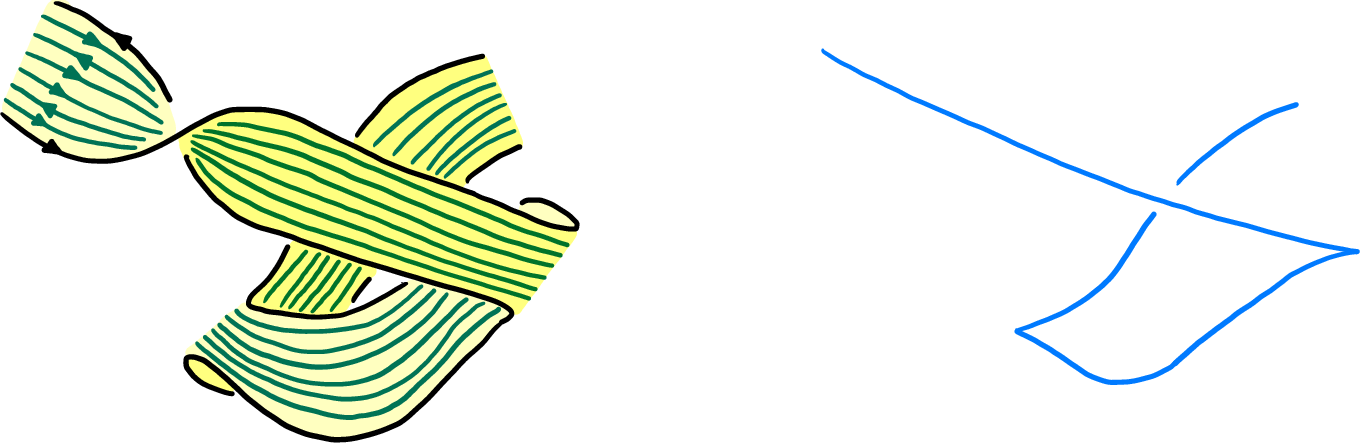}
        \put(19,27){$H$}
        \put(68,27){$\boldsymbol{e}$}
    \end{overpic}
    \caption{Left: a $1$--handle $H$ with six segments of $\aa$. Right: the Legendrian core $\boldsymbol{e}=\G\cap H$.}\label{fig:example_1handles1}
\end{figure}

Remove from $\boldsymbol{e}$ a small open subsegment disjoint from cusps and crossings. Denote the two resulting components by $\boldsymbol{e}^+$ and $\boldsymbol{e}^-$, see Figure~\ref{fig:example_1handles2} (left).\footnote{Choosing which component is called $\boldsymbol{e}^+$ and which $\boldsymbol{e}^-$ is arbitrary; the final result does not depend on this choice.}

\begin{figure}[htbp]
    \centering
    \begin{overpic}[scale=1]{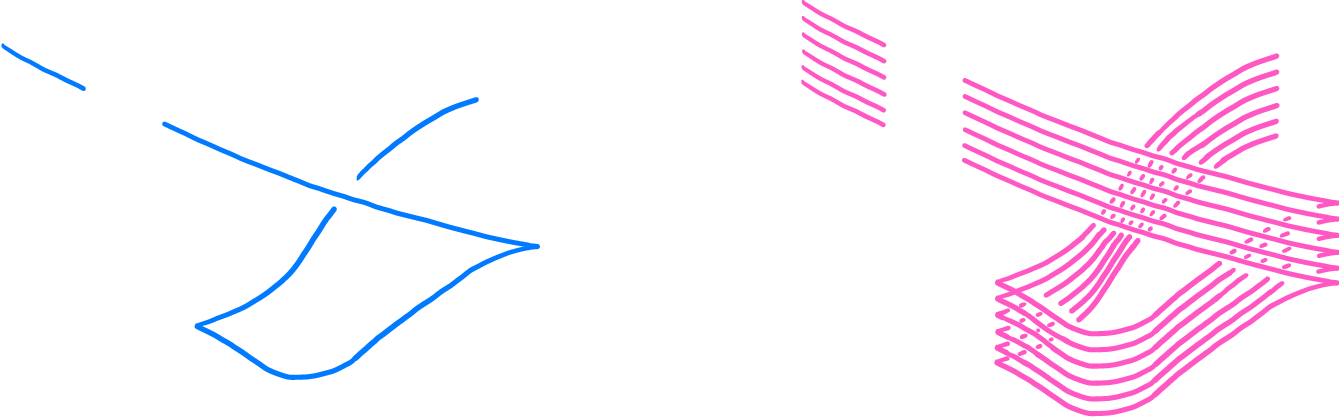}
        \put(4,27){$\boldsymbol{e}^+$}
        \put(15,22){$\boldsymbol{e}^-$}
    \end{overpic}
    \caption{Left: the two components $\boldsymbol{e}^+$ and $\boldsymbol{e}^-$ of $\boldsymbol{e}$ after removing a small open interval. Right: $k$ vertical translates of $\boldsymbol{e}^+\sqcup\boldsymbol{e}^-$.}\label{fig:example_1handles2}
\end{figure}

Create $k$ vertical translates of $\boldsymbol{e}^+\sqcup\boldsymbol{e}^-$ in the front projection (Figure~\ref{fig:example_1handles2}, right). These copies of $\boldsymbol{e}^+ \sqcup \boldsymbol{e}^-$ will form part of the Legendrian realization of the segments $\ll_{j_1}, \dots, \ll_{j_k}$. What remains is to connect the copies of $\boldsymbol{e}^+$ to the copies of $\boldsymbol{e}^-$ with a Legendrian braid.

For each segment $\ll_{j_*}$, denote by $p^\pm_*$ the endpoint of $\ll_{j_*}$ lying on the same side as $\boldsymbol{e}^\pm$, i.e.\ the endpoint in the component of the attaching region containing an endpoint of $\boldsymbol{e}^\pm$ (see Figure~\ref{fig:endpoints_ex}).

\begin{figure}[htbp]
    \centering
    \begin{overpic}[scale=1]{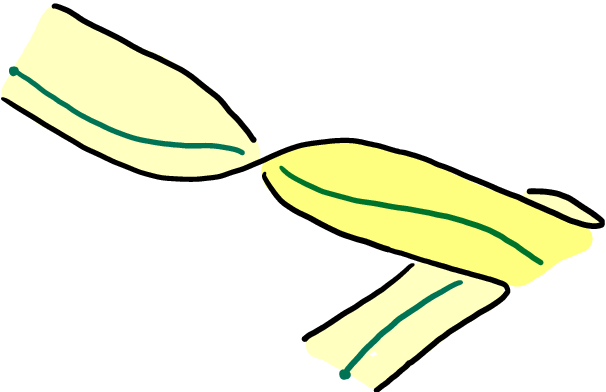}
        \put(-4,56){$p^+_*$}
        \put(15,46){$\ll_*$}
        \put(48,0){$p^-_*$}
    \end{overpic}
    \caption{A segment $\ll_*$ in $H$ with endpoints $p^\pm_*$ on the sides of $\boldsymbol{e}^\pm$.}\label{fig:endpoints_ex}
\end{figure}

Label the translates of $\boldsymbol{e}^\pm$ by $\boldsymbol{e}^\pm(1),\dots,\boldsymbol{e}^\pm(k)$ as follows: the copy labeled $\boldsymbol{e}^\pm(l)$ lies higher in $z$ than the copy $\boldsymbol{e}^\pm(m)$ if
\begin{itemize}
    \item $P(p^\pm_l) > P(p^\pm_m)$, or
    \item $P(p^\pm_l) = P(p^\pm_m)$ but $P(p^\mp_l) > P(p^\mp_m)$.
\end{itemize}

For each $i=1,\dots,k$, connect $\boldsymbol{e}^+(i)$ to $\boldsymbol{e}^-(i)$ by a straight segment in the front projection. The resulting Legendrian arc is declared to be the Legendrian realization of~$\ll_{j_i}$.

\begin{remark}
    Informally, the Legendrian realization of a segment $\ll_{j_*}$ looks like a copy of the core $\boldsymbol{e}$, starting at relative height $P(\partial^-\ll_{j_*})$ and ending at relative height $P(\partial^+\ll_{j_*})$, with a jump of size $\Delta(\ll_{j_*})$ along the way. The combinatorics of the braid record how these jumps for different segments interleave.
\end{remark}

Figure~\ref{fig:example_1handles3} shows a concrete choice of prominences for the segments in Figure~\ref{fig:example_1handles1}, and Figure~\ref{fig:example_1handles4} and~\ref{fig:example_1handles3} show the corresponding braid.

\begin{figure}[htbp]
    \centering
    \begin{overpic}[scale=1]{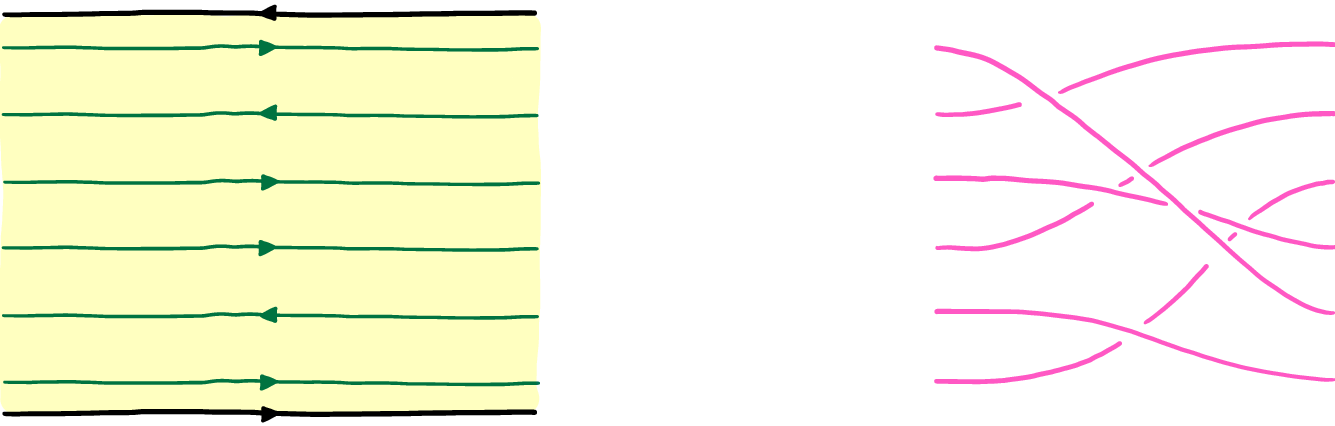}
        \put(-14,2.5){$P=4\rightarrow$}
        \put(-14,7.5){$P=2\rightarrow$}
        \put(-14,12.5){$P=3\rightarrow$}
        \put(-14,17.5){$P=3\rightarrow$}
        \put(-14,22.5){$P=2\rightarrow$}
        \put(-14,27.5){$P=0\rightarrow$}

        \put(41,2.5){$\leftarrow P=1$}
        \put(41,7.5){$\leftarrow P=0$}
        \put(41,12.5){$\leftarrow P=2$}
        \put(41,17.5){$\leftarrow P=4$}
        \put(41,22.5){$\leftarrow P=4$}
        \put(41,27.5){$\leftarrow P=3$}

        \put(67,2.5){$0$}
        \put(67,7.5){$2$}
        \put(67,12.5){$2$}
        \put(67,17.5){$3$}
        \put(67,22.5){$3$}
        \put(67,27.5){$4$}

        \put(101,2.5){$0$}
        \put(101,7.5){$1$}
        \put(101,12.5){$2$}
        \put(101,17.5){$3$}
        \put(101,22.5){$4$}
        \put(101,27.5){$4$}
    \end{overpic}
    \caption{Left: prominences at the endpoints of the segments in Figure~\ref{fig:example_1handles1}. Right: schematic picture of the braid, with strands labeled by the prominence on each side.}\label{fig:example_1handles3}
\end{figure}

\begin{figure}[htbp]
    \centering
    \begin{overpic}[scale=1]{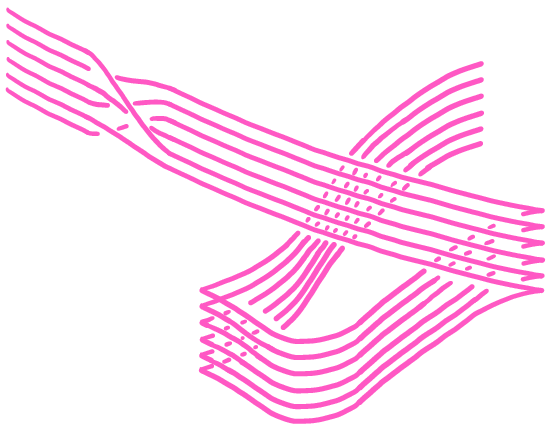}
    \end{overpic}
    \caption{Legendrian realization of the segments of Figure~\ref{fig:example_1handles1}, using the prominence data of Figure~\ref{fig:example_1handles3}.}\label{fig:example_1handles4}
\end{figure}

Note that the data in Figure~\ref{fig:example_1handles1} is not sufficient to determine the Legendrian braid we should place in Figure~\ref{fig:example_1handles2}, right. Indeed, only from Figure~\ref{fig:example_1handles1}, we can at most compute the displacement index $\Delta$ of all the segments. However, to determine the prominence, we need to know how the rest of the curve $\aa$ looks outside the $1$-handle we are considering. For this reason, we must specify the prominence of at least one endpoint of each segment.

\subsubsection*{\textbf{STEP 6}: Legendrian realization in a 0--handle}\label{step_6}

Let $H$ be a $0$--handle, and let $\ll_{j_1},\dots,\ll_{j_k}$ be the segments of $\aa$ contained in~$H$. Let $\boldsymbol{e}_1,\dots,\boldsymbol{e}_n$ be the Legendrian segments in the Legendrian skeleton of $H$, meeting at the vertex~$\boldsymbol{v}$ (Figure~\ref{fig:Leg_real_0handle0}).

\begin{figure}[htbp]
    \centering
    \begin{overpic}[scale=1]{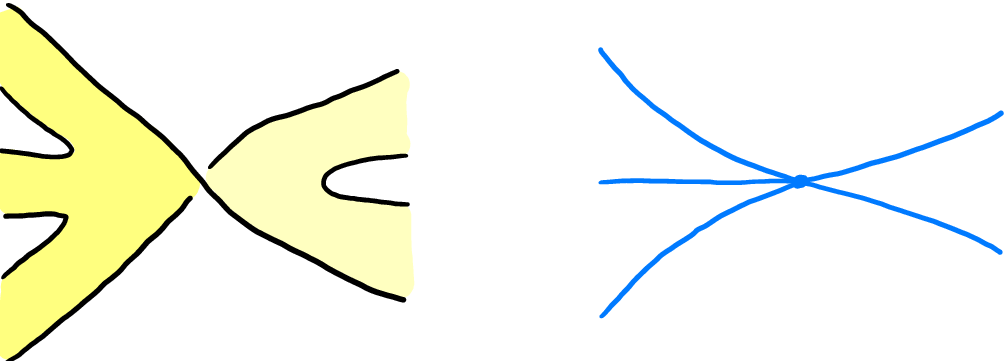}
        \put(78,20){$\boldsymbol{v}$}
        \put(54,5){$\boldsymbol{e_1}$}
        \put(54,17){$\boldsymbol{e_2}$}
        \put(54,30){$\boldsymbol{e_3}$}
        \put(100,10){$\boldsymbol{e_4}$}
        \put(100,25){$\boldsymbol{e_5}$}
        \put(18,25){$H$}
    \end{overpic}
    \caption{A $0$--handle $H$ and its Legendrian skeleton.}\label{fig:Leg_real_0handle0}
\end{figure}
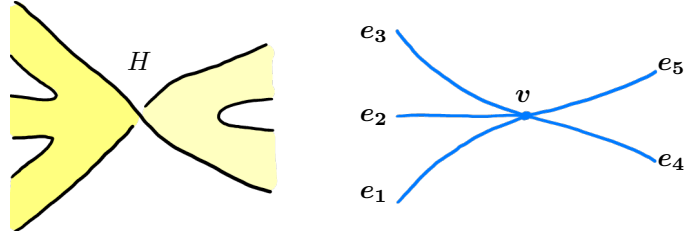

The Legendrian realization of a segment $\ll_*$ in $H$ is obtained by concatenating two of the Legendrian segments $\boldsymbol{e}_1,\dots,\boldsymbol{e}_n$, possibly translated in the $\partial_z$ direction and slightly perturbed. More precisely, if $\ll_*$ is the segment of $\aa$ that joins a $1$--handle $H_1$ to a $1$--handle $H_2$, then its Legendrian realization joins the corresponding pieces of the edges of~$\G$ passing through $H_1$ and~$H_2$.

The relative position of the Legendrian realizations is determined by the prominence. Since both endpoints of a segment $\ll_*$ in a $0$--handle have the same prominence, we write simply $P(\ll_*l)$ for that common value.

We place a segment $\ll_{j_l}$ above a segment $\ll_{j_m}$ if $P(\ll_{j_l}) > P(\ll_{j_m})$. For example, if
\[
P(\ll_{j_4}) < P(\ll_{j_2}) < P(\ll_{j_3}) < P(\ll_{j_1}) < P(\ll_{j_5}),
\]
then the Legendrian realizations are stacked vertically as in Figure~\ref{fig:Leg_real_0handle1}.

\begin{figure}[htbp]
    \centering
    \begin{overpic}[scale=1]{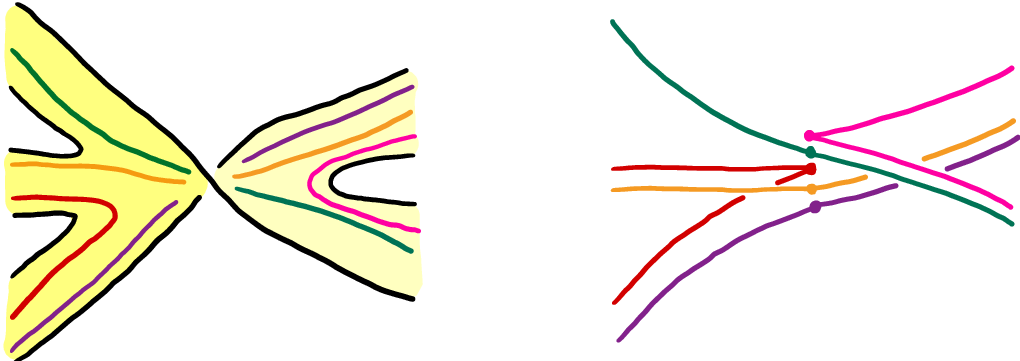}
        \put(-4,30){$\ll_{j_1}$}
        \put(-4,20){$\ll_{j_2}$}
        \put(-4,15){$\ll_{j_3}$}
        \put(-4,0){$\ll_{j_4}$}
        \put(41,20){$\ll_{j_5}$}
    \end{overpic}
    \caption{Left: a $0$--handle $H$ with several segments of $\aa$. Right: their Legendrian realizations when all prominences are distinct. The dots indicate the relative vertical position.}\label{fig:Leg_real_0handle1}
\end{figure}
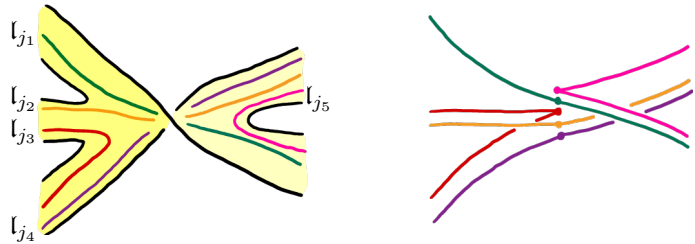

In the case where $P(\ll_{j_l}) = P(\ll_{j_m})$, we again draw the Legendrian realization of these segments as the concatenation of two of the Legendrian segments of the Legendrian skeleton of $H$, but we perturb them slightly to avoid overlaps. The required perturbation visually matches the relative position of the segments $\ll_{j_l}$ and $\ll_{j_m}$ in $H$. See Figure~\ref{fig:Leg_real_0handle2} for an example.

\begin{figure}[htbp]
    \centering
    \begin{overpic}[scale=1]{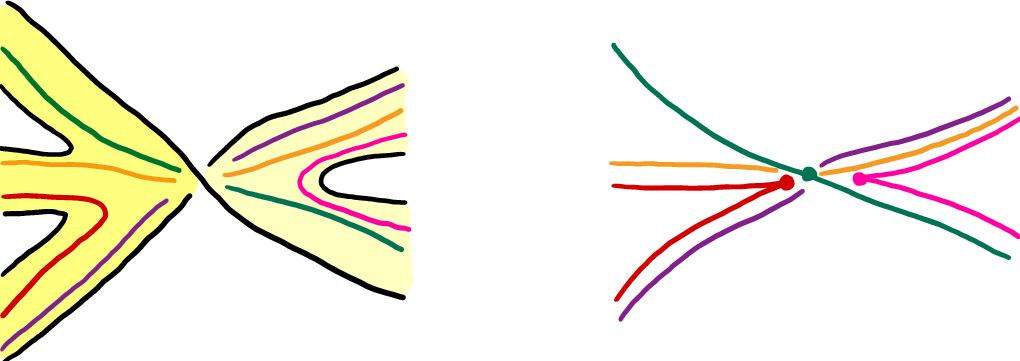}
    \end{overpic}
    \caption{Left: a $0$--handle with several segments of $\aa$ of equal prominence. Right: their Legendrian realizations after a small perturbation.}\label{fig:Leg_real_0handle2}
\end{figure}

\subsubsection*{\textbf{STEP 7}: connecting the pieces}

Finally, we connect the Legendrian realizations of consecutive segments. For each $i$, we join the endpoint of the Legendrian realization of $\ll_i$ to the starting point of the Legendrian realization of $\ll_{i+1}$ by a straight Legendrian segment.

\subsection{An example}\label{sec:example_algo}

In this section we illustrate how the algorithm is applied in practice. Consider the Legendrian graph $G$ in Figure~\ref{fig:example_algorithm1}.

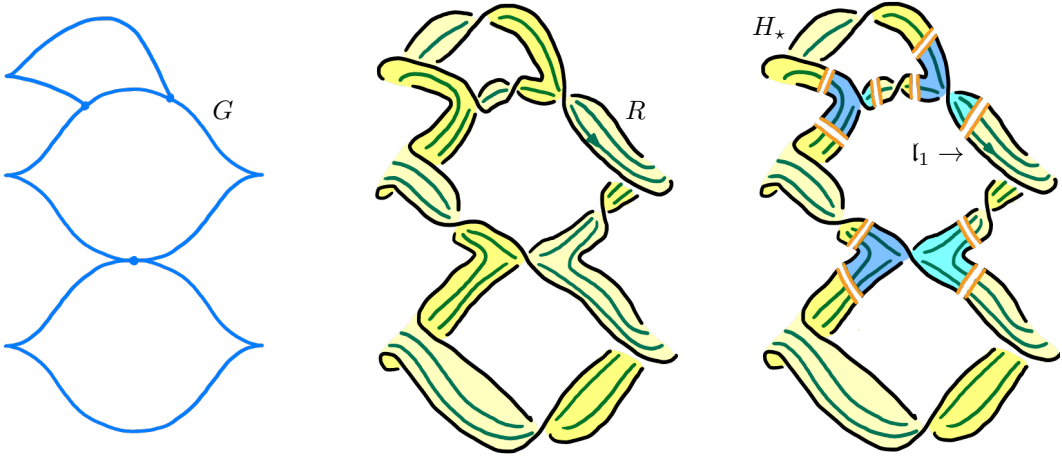
\begin{figure}[htbp]
    \centering
    \begin{overpic}[scale=1]{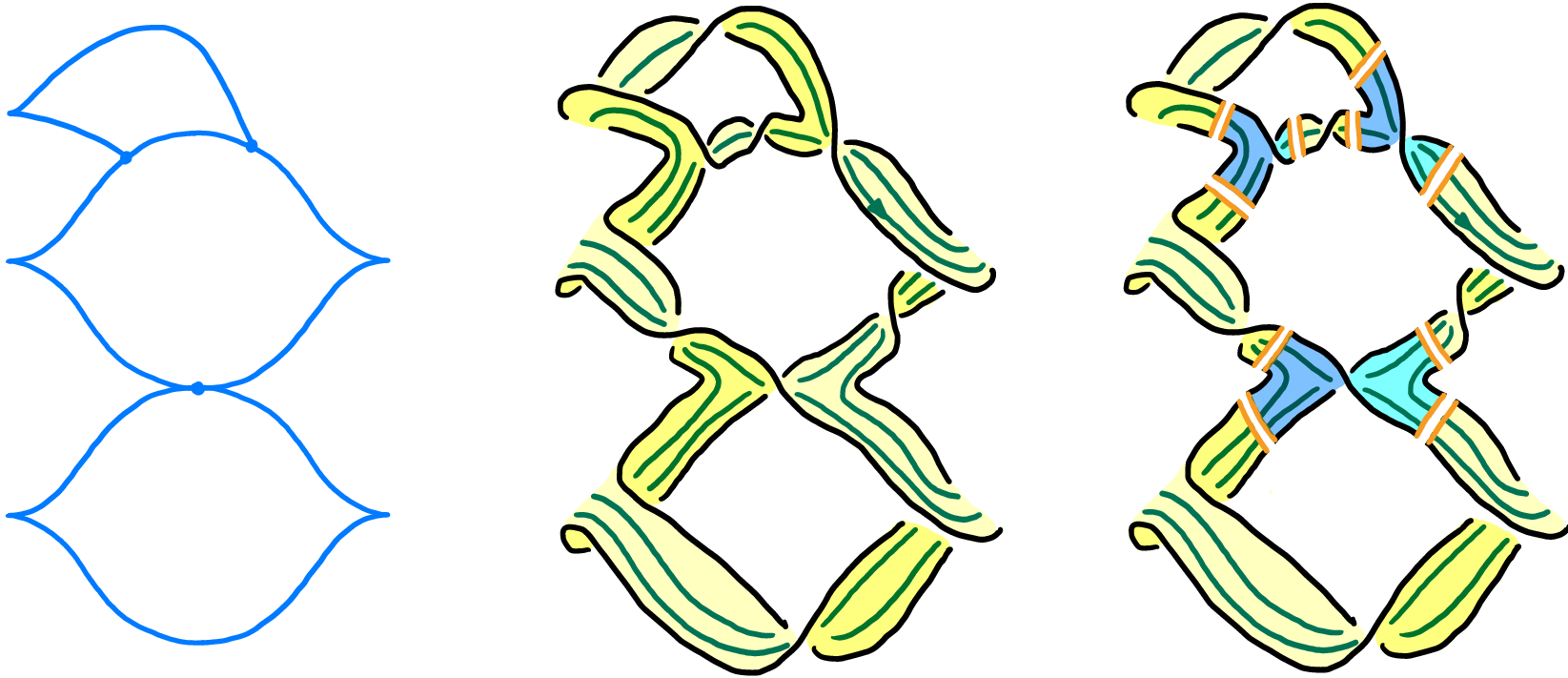}
      \put(20,32){$G$}
      \put(59,32){$R$}
      \put(71,40){$H_\star$}
      \put(86,28){$\ll_1\rightarrow$}
    \end{overpic}
    \caption{Left: the Legendrian graph $G$. Middle: a ribbon surface $R$ of $G$ and an oriented, homologically nontrivial simple closed curve $\aa$ on $R$ (in green). Right: the handle decomposition of $R$ and the curve $\aa$ subdivided into segments as in Step~1.}
    \label{fig:example_algorithm1}
\end{figure}

In the middle of Figure~\ref{fig:example_algorithm1} we see a ribbon surface $R$ of $G$ with an oriented, homologically nontrivial simple closed curve $\aa$ (drawn in green). On the right of Figure~\ref{fig:example_algorithm1} we see the handle decomposition of $R$ (three $0$--handles in blue, and six $1$--handles in yellow), together with the curve $\aa$ subdivided into segments as prescribed after Step~1. In particular, Figure~\ref{fig:example_algorithm1} also indicates which segment is $\ll_1$ and which $1$--handle is the distinguished handle $H_\star$.

If we compute the relative gain of the segments contained in the $1$--handles, we find that all such segments have the same strictly negative relative gain, except for:
\begin{itemize}
  \item the segment in $H_\star$, whose relative gain must later be adjusted; and
  \item the segment in the unique $1$--handle that contains only one segment and is not $H_\star$ (it is $\ll_7$ and its relative gain is $0$).
\end{itemize}

More concretely, set the prominence of $\partial^-\ll_1$ to be zero, that is $P(\partial^-\ll_1)=0$. Then
\[
  P(\partial^+\ll_1)=-\varepsilon,\qquad
  P(\ll_2)=-\varepsilon,\qquad
  P(\partial^-\ll_3)=-\varepsilon,\qquad
  P(\partial^+\ll_3)=-2\varepsilon,
\]
and so on, until we reach $\ll_7$, the segment that is alone in the $1$--handle which is not $H_\star$. At that point
\[
  P(\ll_6)=P(\partial^-\ll_7)=P(\partial^+\ll_7)=P(\ll_8)=-3\varepsilon.
\]
From there on, segments in the $1$--handles continue to have relative gain $-\varepsilon$ until we reach the second–to–last segment, $\ll_{15}$, which lies in $H_\star$. We have $P(\partial^-\ll_{15})=-5\varepsilon$, and since at the end we must return to prominence~$0$, we set
\[
  \Delta(\ll_{15}) = 5\varepsilon,
\]
so that $P(\partial^+\ll_{15})=0$. Consequently $P(\ll_{16})=0$, which agrees with $P(\partial^-\ll_1)$, as required.

For the discussion below, it is helpful to look at Figures~\ref{fig:example_algorithm2.1}, \ref{fig:example_algorithm2.2}, \ref{fig:example_algorithm3.1}, and \ref{fig:example_algorithm3.2} while reading. In these figures we see, successively, the Legendrian realizations of $\ll_1$, then of $\ll_1$ and $\ll_2$, and so on, until we obtain the Legendrian realization of the entire curve $\aa$.

\begin{figure}[htbp]
    \centering
    \begin{overpic}[scale=0.9]{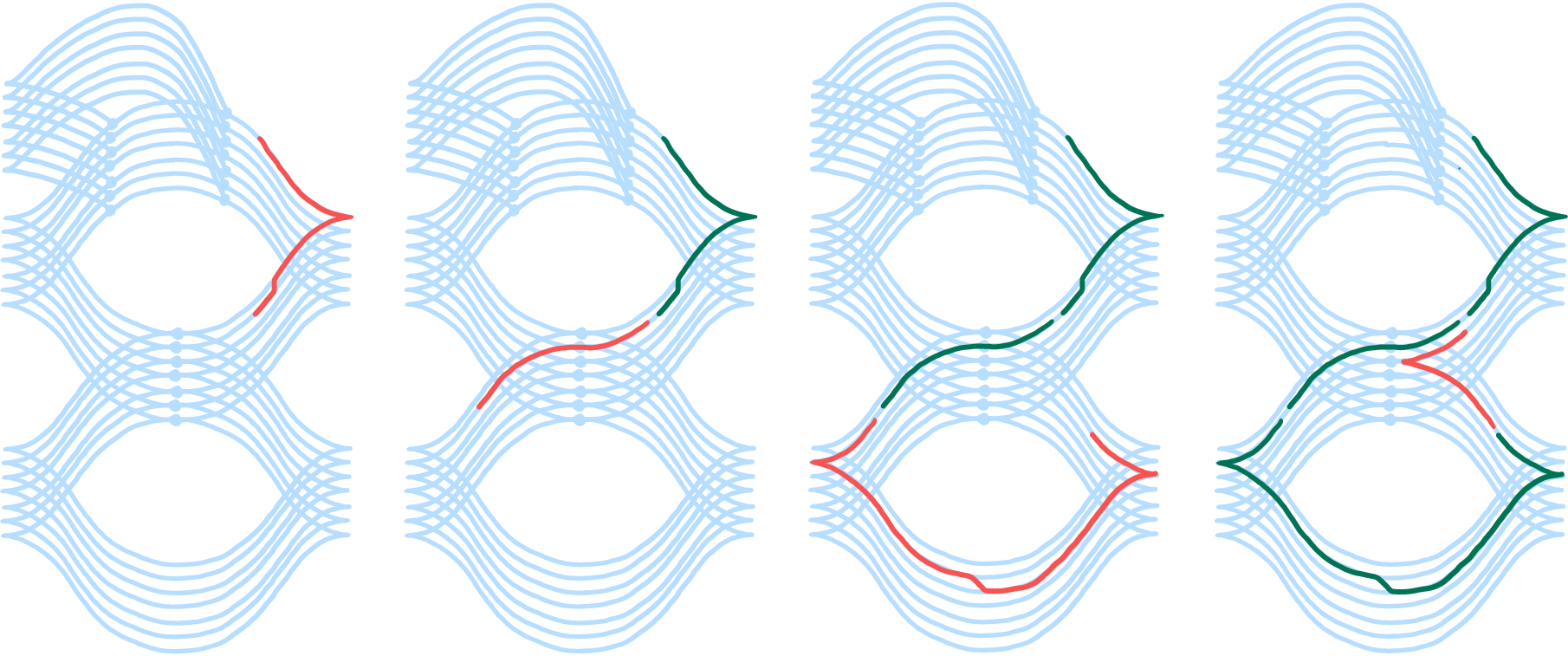}
      \put(20,36){$G\rightarrow$}
    \end{overpic}
    \caption{From left to right: the Legendrian realizations of the segments $\ll_1$ through $\ll_4$.}
    \label{fig:example_algorithm2.1}
\end{figure}

\begin{figure}[htbp]
    \centering
    \begin{overpic}[scale=0.9]{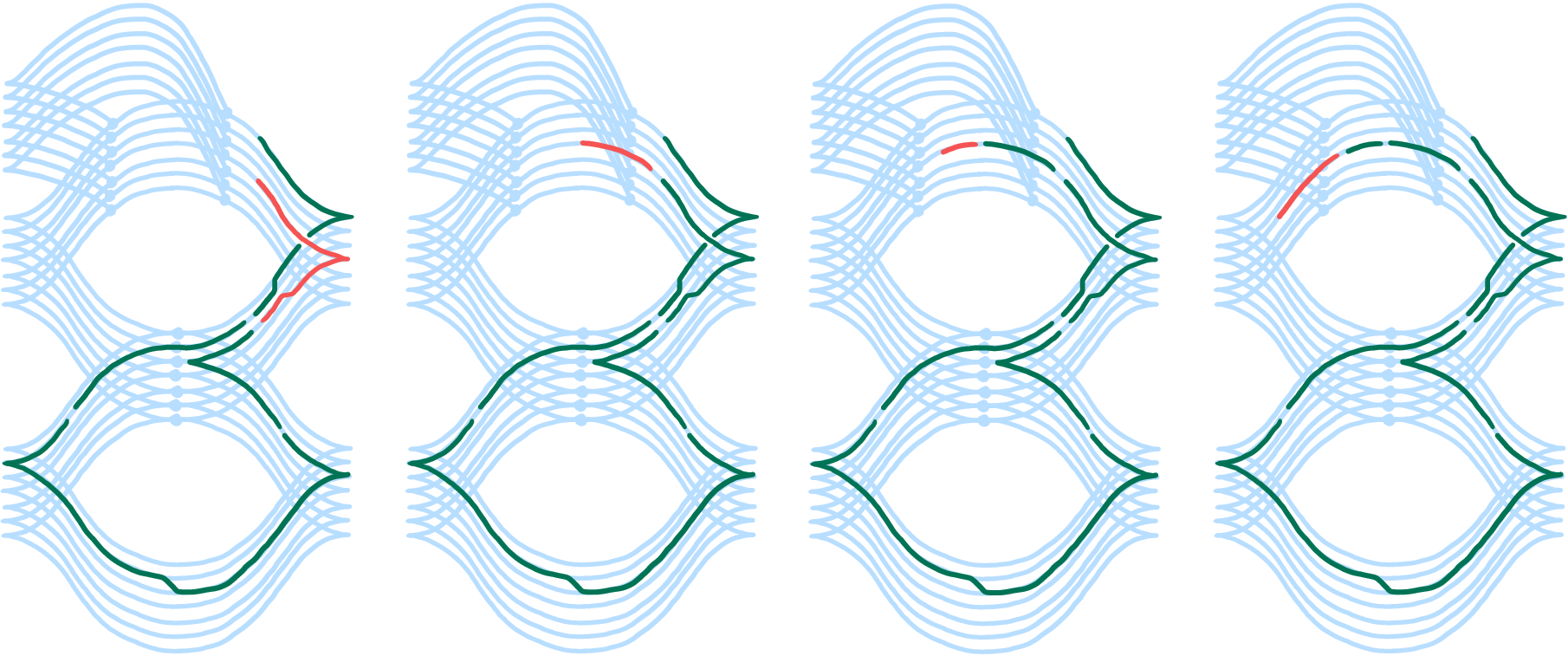}
    \end{overpic}
    \caption{From left to right: the Legendrian realizations of the segments $\ll_5$ through $\ll_8$.}
    \label{fig:example_algorithm2.2}
\end{figure}

\begin{figure}[htbp]
    \centering
    \begin{overpic}[scale=0.9]{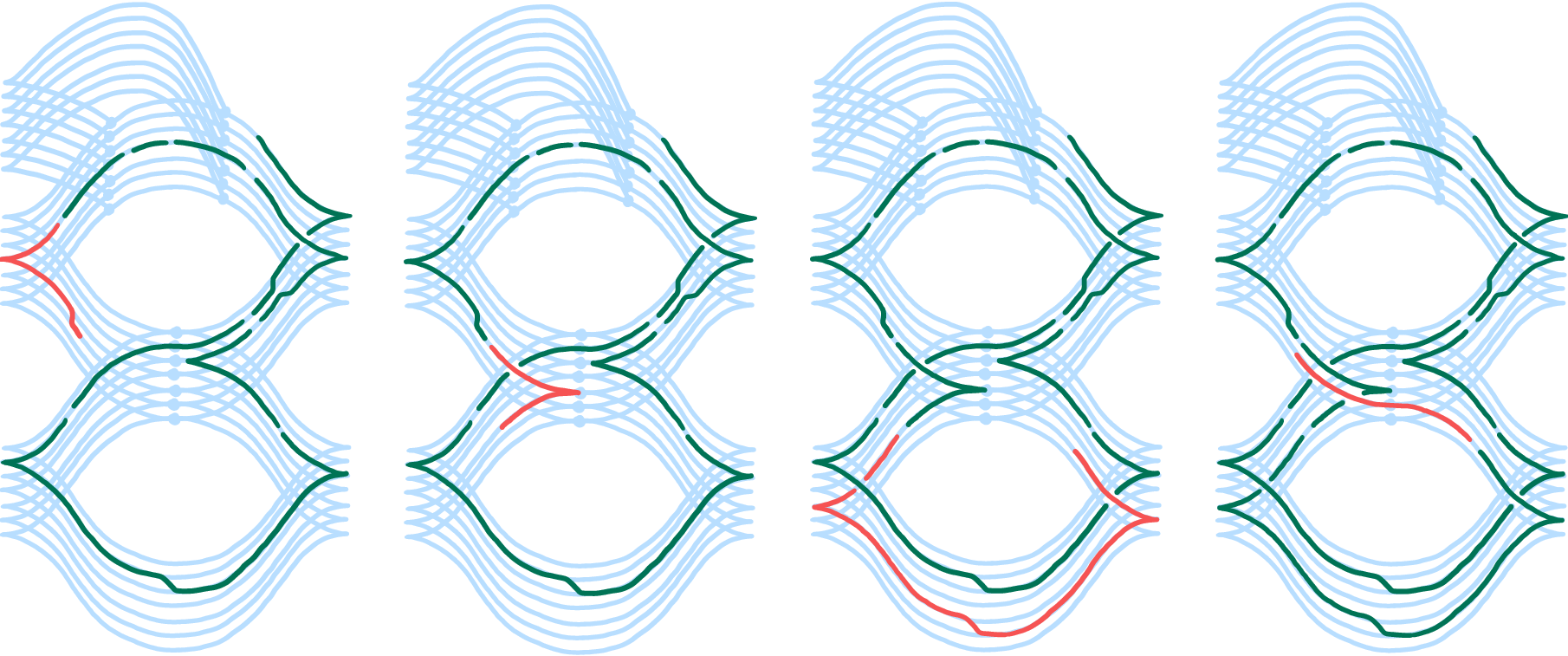}
    \end{overpic}
    \caption{From left to right: the Legendrian realizations of the segments $\ll_9$ through $\ll_{12}$.}
    \label{fig:example_algorithm3.1}
\end{figure}

\begin{figure}[htbp]
    \centering
    \begin{overpic}[scale=0.9]{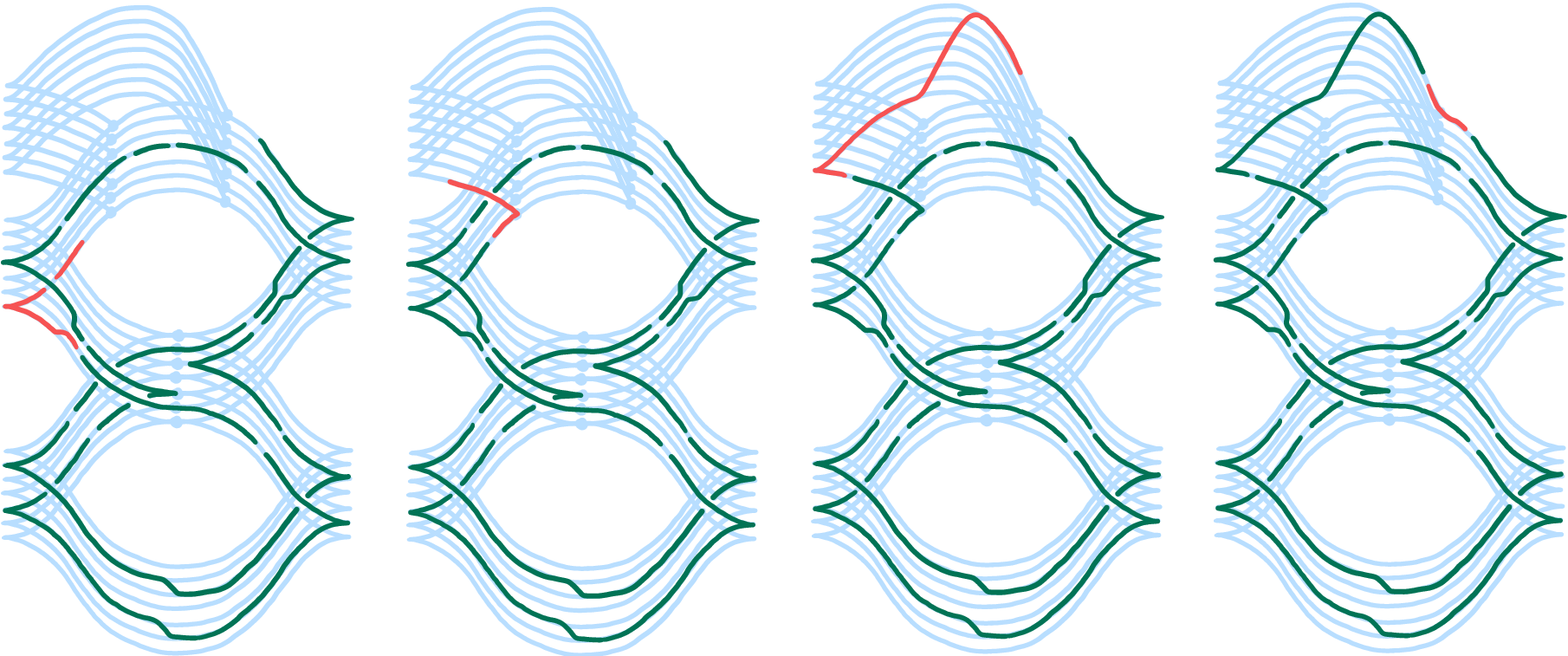}
    \end{overpic}
    \caption{From left to right: the Legendrian realizations of the segments $\ll_{13}$ through $\ll_{16}$.}
    \label{fig:example_algorithm3.2}
\end{figure}

To draw the Legendrian knot explicitly, we can apply the algorithm in a slightly different order from the one described in the general construction. Instead of first Legendrian realizing all segments in the $1$--handles and then all segments in the $0$--handles, we Legendrian realize the segments in the order in which they appear along $\aa$: first $\ll_1$, then $\ll_2$, and so on. This works in the present example because the prominence is monotone up to the second–to–last segment. In general one can still proceed in this way, but one may run out of space in the front projection if a new segment needs to be drawn between two previously drawn segments.

Denote by $G(s)$ the graph $G$ translated vertically by an amount $s$. We draw
\[
  G(0),\; G(-\varepsilon),\; G(-2\varepsilon),\;\dots,\; G(-5\varepsilon).
\]
These are the light blue graphs in Figures~\ref{fig:example_algorithm2.1}--\ref{fig:example_algorithm3.2}.

Consider first the Legendrian realization of $\ll_1$. Initially it coincides with a portion of an edge of $G(0)$, and then it ``jumps'' by an amount $-\varepsilon$ so that it becomes a portion of an edge of $G(-\varepsilon)$; see the first picture in Figure~\ref{fig:example_algorithm2.1}. The Legendrian realization of $\ll_2$ stays at prominence $-\varepsilon$. The Legendrian realization of $\ll_3$ then jumps by $-\varepsilon$ to land on $G(-2\varepsilon)$, and so on, until the Legendrian realization of the segment $\ll_{15}$ jumps from $G(-5\varepsilon)$ back to $G(0)$; see the second–to–last picture in Figure\ref{fig:example_algorithm3.2}.

In Figure~\ref{fig:example_algorithm4} we see the final Legendrian knot produced by the algorithm.

\begin{figure}[htbp]
    \centering
    \begin{overpic}[scale=1]{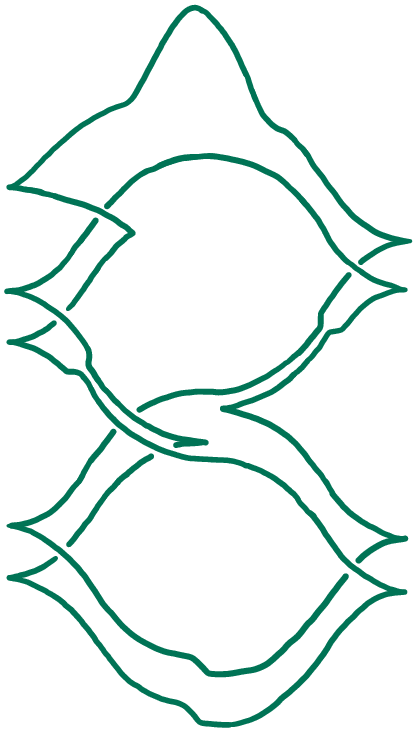}
    \end{overpic}
    \caption{The Legendrian realization of the curve $\aa$ from Figure~\ref{fig:example_algorithm1} obtained by applying the algorithm.}
    \label{fig:example_algorithm4}
\end{figure}

\subsection{Justification of the algorithm}\label{section:justif_algorithm}

The algorithm clearly produces a knot $L$ (and not a link): we connect 
the Legendrian realization of a segment $\ll_i$ to the Legendrian realization 
of the consecutive segment $\ll_{i+1}$ of $\aa$, and $\aa$ is a connected 
closed curve. Moreover, $L$ is a generic Legendrian knot, as all 
the front projections we have produced are generic and never violate the 
Legendrian condition.

We now make precise what remains to be shown and encapsulate it in a single statement.

\begin{proposition}[Validity of the algorithm]\label{prop:validity_algorithm}
Let $\G \subset (\R^3,\xist)$ be a generic Legendrian graph, let $\A$ be the abstract ribbon 
surface of~$\G$, and let $\aa$ be a homologically nontrivial simple closed curve 
on $\A$. Let $L$ be the Legendrian knot produced by the 
algorithm of Theorem~\ref{thm:algorithm}. Then the following hold:
\begin{itemize}
\item There exists a generic ribbon surface $\Sigma$ of $\G$ that can be perturbed 
      relative to the boundary, through surfaces transverse to $\partial_z$, so as to contain $L$.
\item The isotopy class of $L$ in the perturbation of $\Sigma$ agrees with the 
      isotopy class of $\aa$.
\item The perturbation of $\Sigma$ can be chosen to be arbitrarily $C^0$–small 
      if the algorithm is applied carefully.
\end{itemize}
\end{proposition}

The rest of this section is devoted to the proof of Proposition~\ref{prop:validity_algorithm}.
We split the argument into several lemmas corresponding to the subsequent subsections.

\subsubsection{Choosing the ribbon surface}

\begin{lemma}\label{lem:choosing_ribbon}
There exists a generic ribbon surface $\Sigma$ of~$\G$ with the following properties:
\begin{itemize}
\item In each $1$--handle there is exactly one hyperbolic singularity of the characteristic foliation.
\item This hyperbolic singularity is located precisely in the region where the braid of Step~5 
      is placed (compare Figure~\ref{fig:example_1handles1} and Figure~\ref{fig:example_1handles4}).
\end{itemize}
\end{lemma}

\begin{proof}
This can be achieved by first applying the algorithm that produces a ribbon surface from a Legendrian graph, then applying the simplifications of Figure~\ref{fig:ribbon_simplif_local}, and finally sliding the remaining hyperbolic points along the core of each $1$--handle until each lies precisely in the region where the braid of Step~5 is placed.
\end{proof}

\begin{example}\label{ex:ribbon_comp}
    The ribbon in Figure~\ref{fig:example_1handles1} is compatible with the braid in Figure~\ref{fig:example_1handles4}.
    By contrast, the ribbons in Figure~\ref{fig:definition_sigma} are not: the first has too many singularities in the $1$--handle, and in the other two the remaining hyperbolic point sits in the wrong place relative to the braid.
    
    \begin{figure}[htbp]
    \centering
    \begin{overpic}[scale=0.9]{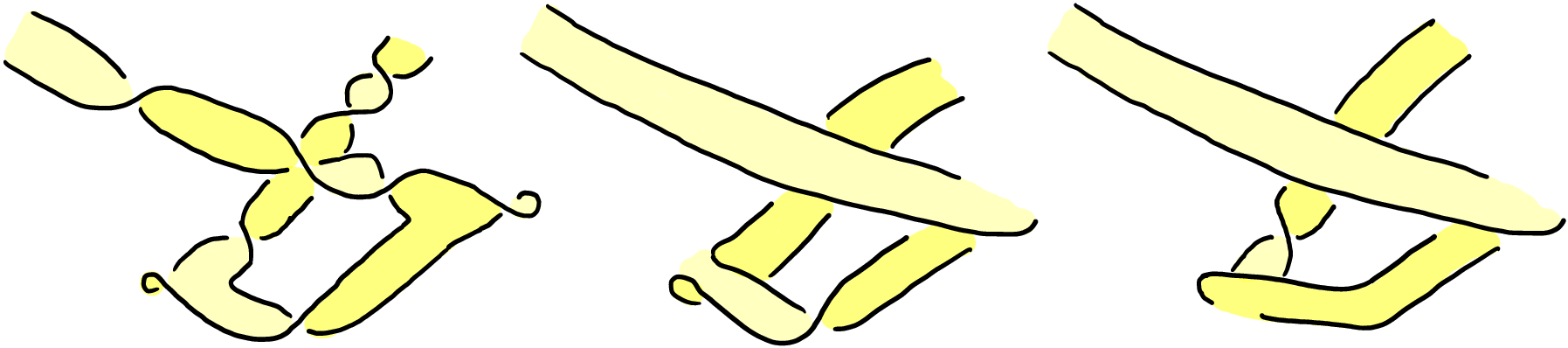}
    \end{overpic}
    \caption{Ribbon surfaces that are not adapted to the braid in Figure~\ref{fig:example_1handles4}: too many singularities (left) or hyperbolic point in the wrong region (middle and right).}\label{fig:definition_sigma}
    \end{figure}
\end{example}

The choice of where to place the braid in Step~5 determines a preferred representative $\Sigma$ of the abstract ribbon, but Theorem~\ref{thm:algorithm} imposes no restriction on~$\Sigma$ beyond genericity, so this choice causes no issues.

\subsubsection{Legendrian realization of individual segments}

Let $\ll_i\colon[0,1]\to(\R^3,\xist)$ be one of the segments of~$\aa$, parametrized compatibly with its orientation. Denote its components by
\[
\ll_i(s) = \big((\ll_i)_x(s),(\ll_i)_y(s),(\ll_i)_z(s)\big).
\]

We Legendrian realize $\ll_i$ by keeping its Lagrangian projection $(x,y)$ and adjusting its $z$–coordinate. Given a number $z_i\in\R$, define
\[
\overline{(\ll_i)_z}(s)
:= z_i - \int_0^s (\ll_i)_x(u)\,(\ll_i)_y'(u)\,du,
\]
and set
\[
\overline{\ll_i}(s)
:= \big((\ll_i)_x(s),(\ll_i)_y(s),\overline{(\ll_i)_z}(s)\big).
\]
By construction, $\overline{\ll_i}$ is a Legendrian curve with $\overline{(\ll_i)_z}(0)=z_i$.

\begin{remark}\label{patch_together}
    At this stage, the value of $z_i$ is arbitrary for each segment, so the family $\{\overline{\ll_i}\}_i$ does not yet patch into a closed Legendrian curve. We are currently only interested in the local shape of the front projection of $\overline{\ll_i}$ inside a handle; the global choice of the $z_i$ will be fixed later when we glue the segments together.
\end{remark}

Next, we describe how these segments $\overline{\ll_i}$ look and why they match (individually) the segments produced by the algorithm.

\subsubsection{1--handles}

\begin{lemma}\label{lem:1handles_displacement}
Let $H$ be a $1$--handle of the ribbon surface $\Sigma$, and let 
$\ll_i$ be the segments of $\aa$ contained in~$H$. After isotoping $\aa$ 
inside~$H$, we can assume the Legendrian realizations 
$\overline{\ll_i}$ satisfy the following:
\begin{enumerate}
\item For $H \neq H_\star$, the function
      \[
        h_i(s) := \overline{(\ll_i)_z}(s) - (\ll_i)_z(s)
      \]
      is constant on $[0,\tfrac{1}{3}]$ and $[\tfrac{2}{3},1]$, and monotone on 
      $[\tfrac{1}{3},\tfrac{2}{3}]$:
      \begin{itemize}
        \item if $\ll_i$ is the central (already Legendrian) segment, then $h_i$ is constant everywhere;
        \item otherwise, $h_i$ is strictly decreasing on $[\tfrac{1}{3},\tfrac{2}{3}]$ if $\ll_i$ is positively transverse, and strictly increasing if it is negatively transverse;
      \end{itemize}
\item The “jump”
      \[
        D_i := h_i(1) - h_i(0) 
        = h_i\Big(\tfrac{2}{3}\Big) - h_i\Big(\tfrac{1}{3}\Big)
      \]
      can be arranged, after a suitable isotopy of $\aa$ inside $H$, to satisfy
      \[
        D_i = \varepsilon_1\,\Delta(\ll_i),
      \]
      for some arbitrarily small $\varepsilon_1>0$.
\item For the distinguished handle $H_\star$, the same relation 
      $D_i = \varepsilon_1\,\Delta(\ll_i)$ holds after the additional adjustment 
      of the $\Delta(\ll_i)$ prescribed in~\eqref{eq:redefine_displ} and the 
      small isotopy described in Figure~\ref{fig:segments_H_1}.
\end{enumerate}
\end{lemma}

\begin{proof}
Let $H$ be a $1$--handle distinct from~$H_\star$. Isotope the segments of $\aa$ inside $H$ so that:
\begin{itemize}
    \item outside a small neighborhood of the unstable separatrix of the unique hyperbolic singularity of the characteristic foliation, each segment is tangent to the leaves of the foliation;
    \item inside that neighborhood, the segments become short arcs transverse to the foliation, arranged symmetrically around the hyperbolic point as in Figure~\ref{fig:segment_1handles}(right).
\end{itemize}

\begin{figure}[htbp]
    \centering
    \begin{overpic}[scale=0.8]{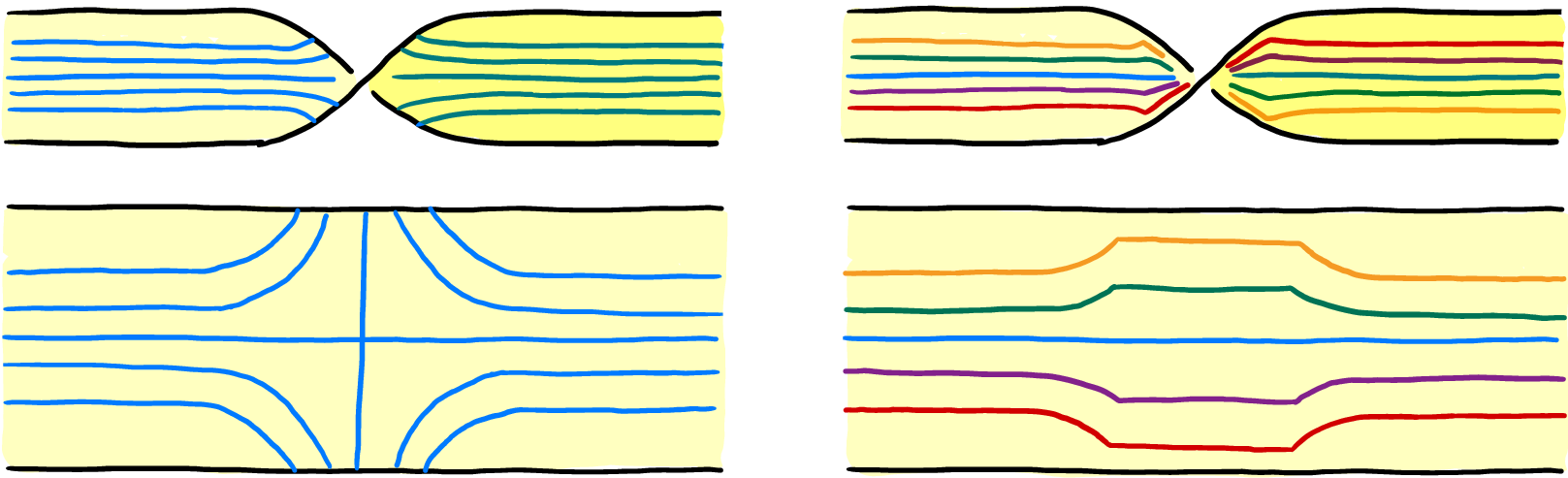}
    \end{overpic}
    \caption{Left: a $1$--handle and its characteristic foliation, with a single hyperbolic point. Right: the segments of $\aa$ arranged along the leaves, except near the unstable separatrix. Top pictures: front projections; bottom pictures: schematic versions.}\label{fig:segment_1handles}
\end{figure}

This clearly implies (1).

By pushing the segments of $\aa$ closer to the central one we can make all $|D_i|$ as small as desired. Moreover, segments that are further away from the center have larger $|D_i|$. In particular, for any small $\varepsilon_1>0$ we can arrange, after a suitable isotopy of $\aa$ inside $H$, that
\[
D_i =
\begin{cases}
-|\widetilde{\Delta}(\ll_i)|\,\varepsilon_1, & \text{if $\ll_i$ is positively transverse},\\[0.3em]
+|\widetilde{\Delta}(\ll_i)|\,\varepsilon_1, & \text{if $\ll_i$ is negatively transverse}.
\end{cases}
\]
It follows that
\begin{equation}\label{eq:displ_prop_to}
D_i = \varepsilon_1\,\Delta(\ll_i),
\end{equation}
because:
\begin{itemize}
    \item a positively transverse segment $\ll_i$ with $\widetilde{\Delta}(\ll_i)$ negative is oriented as $\boldsymbol{e}$;
    \item a positively transverse segment $\ll_i$ with $\widetilde{\Delta}(\ll_i)$ positive is not oriented as $\boldsymbol{e}$;
    \item a negatively transverse segment $\ll_i$ with $\widetilde{\Delta}(\ll_i)$ negative is not oriented as $\boldsymbol{e}$;
    \item a negatively transverse segment $\ll_i$ with $\widetilde{\Delta}(\ll_i)$ positive is oriented as $\boldsymbol{e}$.
\end{itemize}

This proves (2).

When $H=H_\star$, the values of $\Delta(\ll_i)$ were modified by~\eqref{eq:redefine_displ}. To keep~\eqref{eq:displ_prop_to} valid, we first arrange the segments in $H_\star$ symmetrically as in Figure~\ref{fig:segment_1handles}, and then apply a small isotopy that pushes all of them slightly up or down as in Figure~\ref{fig:segments_H_1}. This breaks the symmetry but allows us to realize the adjusted values of~$\Delta$ as in~\eqref{eq:displ_prop_to}.

\begin{figure}[htbp]
    \centering
    \begin{overpic}[scale=0.8]{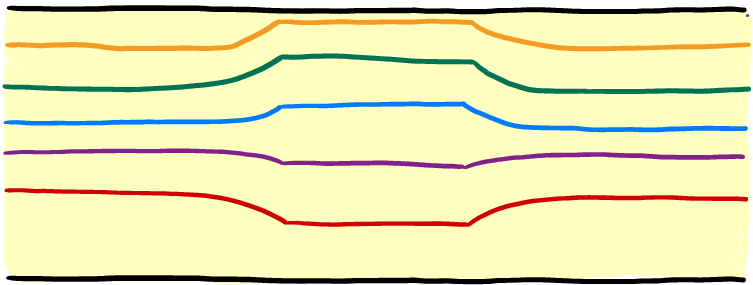}
    \end{overpic}
    \caption{Segments of $\aa$ in the distinguished handle $H_\star$ after the additional isotopy used to implement the modified gains.}\label{fig:segments_H_1}
\end{figure}

This proves (3).
\end{proof}

A good way to understand how $\overline{\ll_i}$ looks is as follows. Consider the ribbon surface $\Sigma$. Any surface obtained by translating $\Sigma$ in the $\partial_z$ direction is again a ribbon surface with all the properties of $\Sigma$, as $\partial_z$ is a contact vector field. Denote by $\Sigma(q)$ the ribbon surface $\Sigma$ translated vertically by an amount $q \in \mathbb{R}$.

The Legendrian segment $\overline{\ll_i}|_{[0, \frac{1}{3}]}$ is embedded in the ribbon surface $\Sigma(h_i(0))$, as this segment is just a copy of $\ll_i|_{[0, \frac{1}{3}]}$ translated vertically by $h_i(0)$. Similarly, the Legendrian segment $\overline{\ll_i}|_{[\frac{2}{3}, 1]}$ is embedded in the ribbon surface $\Sigma(h_i(1))$. On the other hand, the Legendrian segment $\overline{\ll_i}|_{[\frac{1}{3}, \frac{2}{3}]}$ is not a simple vertical translation of $\ll_i|_{[\frac{1}{3}, \frac{2}{3}]}$. Instead, it is a Legendrian segment that ``jumps'' by an amount $D_i$ from the ribbon surface $\Sigma(h_i(0))$ to the ribbon surface $\Sigma(h_i(1))=\Sigma(h_i(0)+D_i)$. See Figure~\ref{fig:H_and_D} for an example.

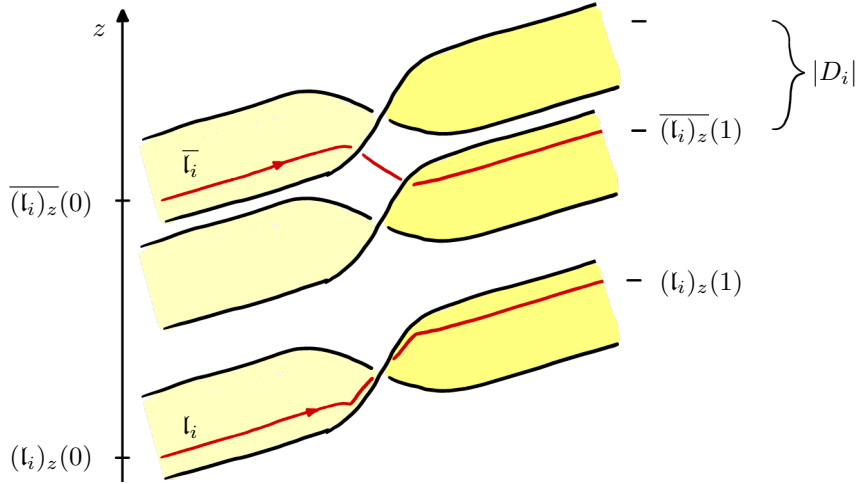
\begin{figure}[htbp]
    \centering
    \begin{overpic}[scale=1]{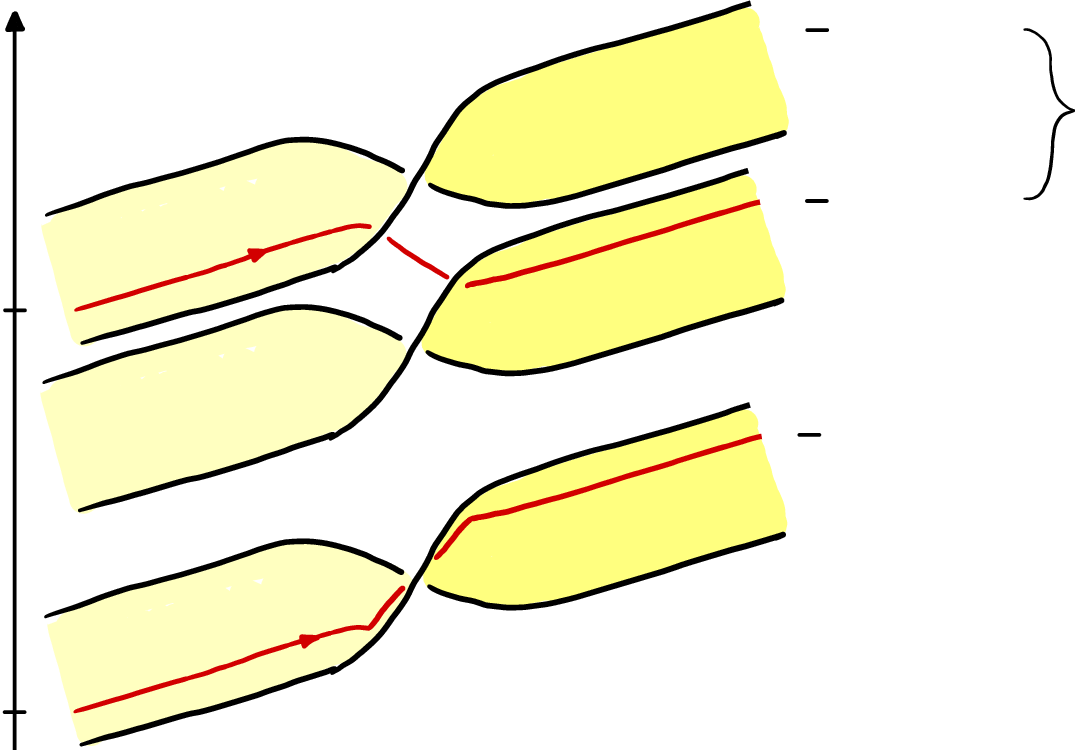}
        \put(-3,65){$z$}
        \put(-15,39.5){$\overline{(\ll_i)_z}(0)$}
        \put(-15,2.5){$(\ll_i)_z(0)$}
        \put(79,50){$\overline{(\ll_i)_z}(1)$}
        \put(79,28){$(\ll_i)_z(1)$}
        \put(101,58.5){$|D_i|$}
        \put(10,7){$\ll_i$}
        \put(10,45){$\overline{\ll_i}$}
    \end{overpic}
    \caption{Legendrian realization $\overline{\ll_i}$ of a segment $\ll_i$ in a $1$--handle. The lower surface is $\Sigma$, the others are vertical translates.}\label{fig:H_and_D}
\end{figure}

For example, if we take the segments in Figure~\ref{fig:segment_1handles} (right) as inputs, we can imagine that their Legendrian realizations will have ``jumps'' as shown in Figure~\ref{fig:segment_1handles1}, where the external segments (red and orange) ``jump'' more than the internal segments (purple and green). The blue segment is not drawn, as its Legendrian realization will simply be a copy of itself translated vertically by a certain amount.

\begin{figure}[htbp]
    \centering
    \begin{overpic}[scale=0.8]{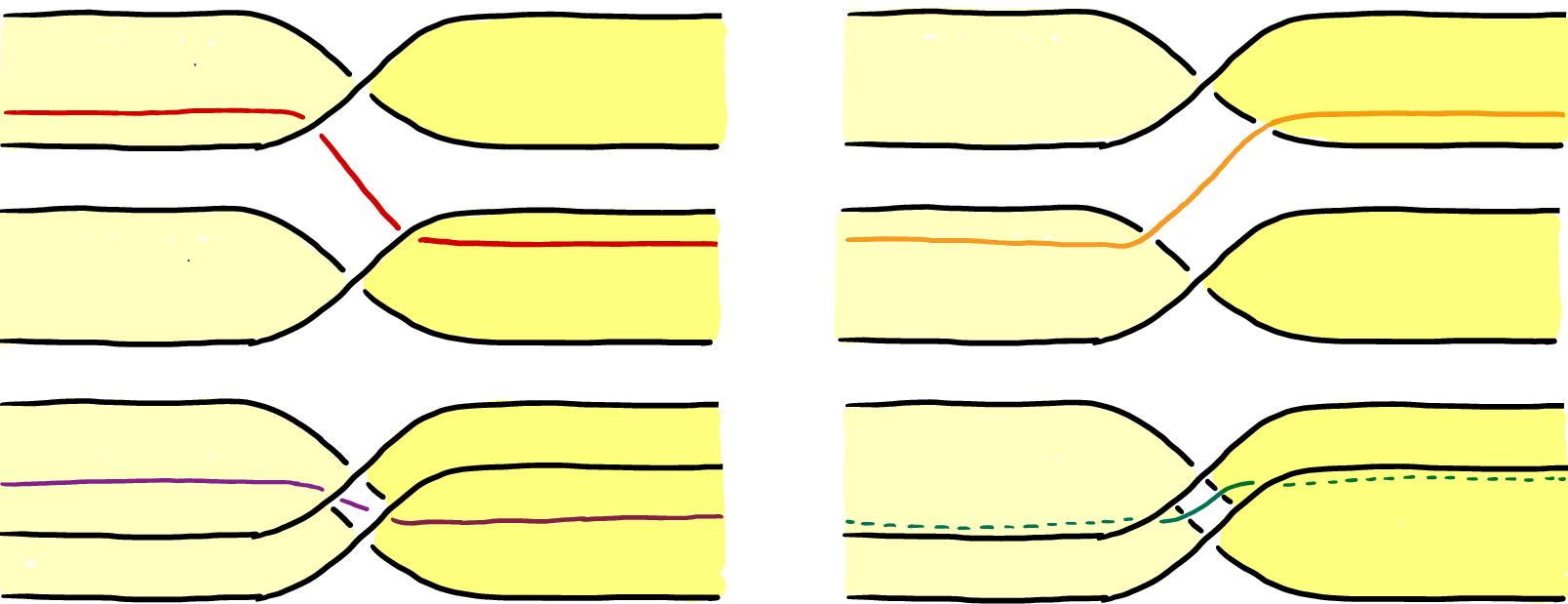}
    \end{overpic}
    \caption{The Legendrian realization of the segments in Figure~\ref{fig:segment_1handles}, right.}\label{fig:segment_1handles1}
\end{figure}

\subsubsection{0--handles}

\begin{lemma}\label{lem:0handles}
Let $H$ be a $0$--handle of $\Sigma$, and let $\ll_i$ be the segments of $\aa$ 
contained in~$H$. After isotoping the segments to follow the leaves of the 
characteristic foliation (except in a small neighborhood of the vertex 
$\boldsymbol{v}$) and modifying each segment near~$\boldsymbol{v}$ by a small 
“bump” as in Figure~\ref{fig:segment_0handles1}, the associated functions
\[
h_i(s) := \overline{(\ll_i)_z}(s) - (\ll_i)_z(s)
\]
satisfy
\[
h_i(0) = h_i(1)
\]
for all segments in $0$--handles.
\end{lemma}

\begin{proof}
We isotope the segments of $\aa$ in~$H$ so that they follow the leaves of the characteristic foliation, except in a small neighborhood of the vertex $\boldsymbol{v}$, where they are slightly displaced to avoid intersecting each other. For each segment $\ll_i$ we again consider
\[
h_i(s):=\overline{(\ll_i)_z}(s)-(\ll_i)_z(s).
\]
We can assume that $h_i$ is constant on $[0,\tfrac{1}{3}]$ and on $[\tfrac{2}{3},1]$, since the segment is already Legendrian there. By making each segment follow the leaves of the foliation more closely, we can ensure that $|h_i(0)-h_i(1)|<\varepsilon_0$ for any prescribed $\varepsilon_0>0$. We then correct this discrepancy to exactly $0$ by introducing a small ``bump'' near the vertex, as in Figure~\ref{fig:segment_0handles1}, whose Legendrian realization cancels the mismatch.

\begin{figure}[htbp]
    \centering
    \begin{overpic}[scale=1]{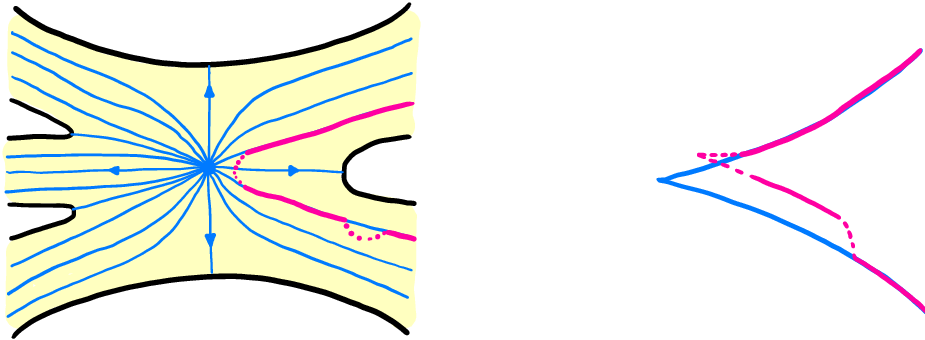}
    \end{overpic}
    \caption{Left: a segment $\ll_i$ in a $0$--handle, isotoped into the desired position (purple). Right: comparison between the leaves of the characteristic foliation (blue) and the Legendrian realization of $\ll_i$ (purple). The bump near the vertex corrects the mismatch in $z$ at the endpoints.}\label{fig:segment_0handles1}
\end{figure}

After this correction, we have $h_i(0)=h_i(1)$ for all segments in $0$--handles.
\end{proof}

\subsubsection{Gluing the segments together}\label{sec:putting_together}

\begin{lemma}\label{lem:glue_segments}
Let $\overline{\ll_i}$ be the Legendrian realizations of the segments of $\aa$ 
constructed above. Then:
\begin{enumerate}
\item There exist constants $z_i$ in the definition of $\overline{\ll_i}$ such that 
      the segments $\overline{\ll_i}$ glue to a Legendrian knot~$L'$.
\item The prominence~$P$ is a discrete version of the function $h$, in the sense that
      for the endpoints $\partial^\pm \ll_i$ of a segment $\ll_i$ we have
      \[
        h_i(0) = \varepsilon_1 P(\partial^-\ll_i)
        \quad \text{and} \quad
        h_i(1) = \varepsilon_1 P(\partial^+\ll_i),
      \]
      where $D_i = \varepsilon_1 \Delta(\ll_i)$.
\item The ribbon surface $\Sigma$ can be perturbed relative to $\partial\Sigma$ 
      so that it contains $L'$, and this perturbation can be chosen arbitrarily 
      $C^0$–small by isotoping $\aa$ closer to the cores of the handles.
\end{enumerate}
\end{lemma}

\begin{proof}
Set
\[
z_1 := (\ll_1)_z(0),
\qquad
z_{i+1} := \overline{(\ll_i)_z}(1)\quad\text{for }i=1,\dots,n-1.
\]
Now the segments glue together to give a Legendrian curve $L'$. To prove that $L'$ closes up to a knot we need to show that $\overline{(\ll_n)_z}(1)=(\ll_1)_z(0)$. This closure condition is equivalent to
\[
\oint_\aa x\,dy = 0,
\]
and one verifies using~\eqref{eq:displ_prop_to} and the definition of the prominence that this holds precisely because $\sum_i\Delta(\ll_i)=0$ (equation~\eqref{eq:displ_sum_zero}). This proves (1).

Assertion (2) is an easy consequence of Lemma~\ref{lem:1handles_displacement}(2) and (3), together with the fact that, after the choice of $z_i$ as above, $h_1(0)=0=P(\partial^-\ll_1)$.

The ribbon surface $\Sigma$ can indeed be perturbed as described in Theorem~\ref{thm:algorithm}, since $L'$ only differs from $\aa$ in the $z$ component, and the surface $\Sigma$ is transverse to $\partial_z$ and contains $\aa$ in its interior. An example of such a perturbation is shown schematically in Figure~\ref{fig:perturbation_sigma}.

\begin{figure}[htbp]
    \centering
    \begin{overpic}[scale=1]{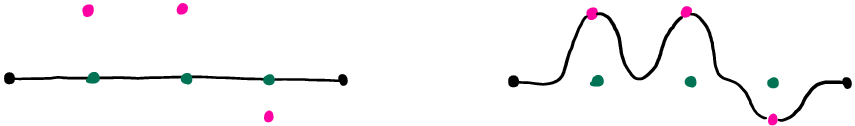}
    \end{overpic}
    \caption{Left: a section of the original ribbon~$\Sigma$: green points represent $\aa$, pink points represent~$L'$. Right: a small perturbation of~$\Sigma$ that contains~$L'$.}\label{fig:perturbation_sigma}
\end{figure}

Because the quantities $\varepsilon_1$ and $\varepsilon_0$ in the construction of the $h_i$ can be chosen arbitrarily small by isotoping $\aa$ closer to the cores of the handles, this perturbation can be made arbitrarily $C^0$–small. Figure~\ref{fig:perturbation_sigma2} shows schematically how bringing $\aa$ closer to the core reduces the vertical displacement between $\aa$ and~$L'$.

\begin{figure}[htbp]
    \centering
    \begin{overpic}[scale=1]{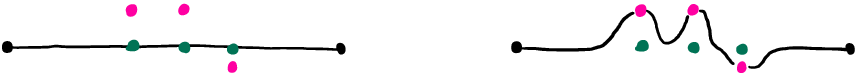}
    \end{overpic}
    \caption{Pushing $\aa$ closer to the core of the handle decreases the vertical distance between $\aa$ and its Legendrian realization~$L'$, allowing the perturbation of~$\Sigma$ to be arbitrarily $C^0$--small.}\label{fig:perturbation_sigma2}
\end{figure}

This proves (3).
\end{proof}

\begin{proof}[Proof of Proposition~\ref{prop:validity_algorithm}]
The fact that the algorithm produces a generic Legendrian knot $L$ is explained 
at the beginning of this section. Consider the Legendrian knot $L'$ given by Lemma~\ref{lem:glue_segments}(1). Assume for a moment that $L=L'$. Then Lemma~\ref{lem:glue_segments}(3) shows that $\Sigma$ can be perturbed relative to $\partial\Sigma$ so that it contains $L=L'$, and this perturbation can be chosen arbitrarily $C^0$–small by isotoping $\aa$ closer to the cores of the handles. Moreover, the isotopy class of $L=L'$ in the perturbation of $\Sigma$ clearly equals the isotopy class of $\aa$, since $L'$ and $\aa$ have the same Lagrangian projection and we have perturbed $\Sigma$ through surfaces transverse to $\partial_z$ ($L'$ is the actual Legendrian realization of $\aa$).

It remains to show that $L=L'$, that is, that the Legendrian segments produced by the algorithm match the Legendrian segments $\overline{\ll_i}$.

For segments of $\aa$ in a $1$--handle, we have already explained how their Legendrian realizations look, up to translation in the $\partial_z$ direction. Essentially, the Legendrian realization $\overline{\ll_i}$ of a segment $\ll_i$ initially looks like a copy of the Legendrian core $\boldsymbol{e}$, translated up or down relative to the Legendrian core by an amount $h_i(0)$. When $\ll_i$ is transverse, it changes its relative height with respect to the Legendrian core, and then it again looks like a copy of $\boldsymbol{e}$ at a relative height of $h_i(1)$.

Given a $1$--handle $H$ with $k$ segments of $\aa$, Step~5 requires us to draw $k$ different copies of $\boldsymbol{e}$ minus a small interval and then merge them together using a prescribed Legendrian braid. The connected components of the $k$ different copies of $\boldsymbol{e}$ minus a small interval correspond to the parts of the Legendrian realization of the segments where these realizations look like copies of $\boldsymbol{e}$. The strands of the braid that connect these components represent the Legendrian realizations of the transverse parts of the segments. This procedure produces Legendrian segments that properly represent the Legendrian realization of the $k$ segments, as can be verified by carefully examining all possible cases.

We illustrate this with a concrete example. Assume that $\ll_i$ and $\ll_j$ are two segments of $\aa$ oriented in the same direction in a $1$--handle $H$. Their relative position in the $1$--handle $H$ might be as shown in Figure~\ref{fig:Fredy1}(a) or Figure~\ref{fig:Fredy1}(b).\footnote{It is part of the example that the $1$--handle looks as in these figures.}

\begin{figure}[htbp]
    \centering
    \begin{overpic}[scale=1]{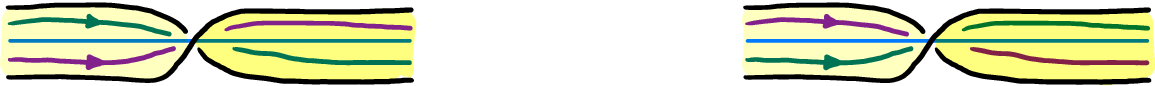}
    \put(-6,4){(a)}
    \put(58,4){(b)}
    \end{overpic}
    \caption{(a) The segment that starts higher (green) is $\ll_i$, while the lower segment (purple) is $\ll_j$. (b) The positions of $\ll_i$ and $\ll_j$ are reversed.}\label{fig:Fredy1}
\end{figure}

This relative position may be determined by the prominence of their endpoints. Without loss of generality, assume $h_i(0) \geq h_j(0)$. Then we have three cases: $h_i(1) > h_j(1)$, $h_i(1) = h_j(1)$, or $h_i(1) < h_j(1)$.

In the case where $h_i(1) > h_j(1)$ and $h_i(0) > h_j(0)$, Step~5 ensures that the Legendrian realization $\overline{\ll_i}$ of $\ll_i$ is always above the Legendrian realization $\overline{\ll_j}$ of $\ll_j$. This is consistent with the definition of the function $h$, as this function encodes how high the Legendrian realization of a given segment is with respect to the Legendrian core of the handle. See Figure~\ref{fig:Fredy2} for an example.

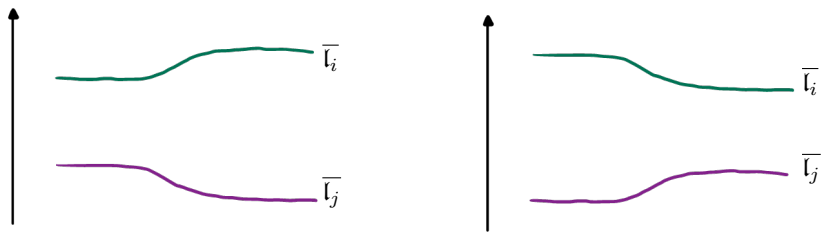
\begin{figure}[htbp]
    \centering
    \begin{overpic}[scale=1]{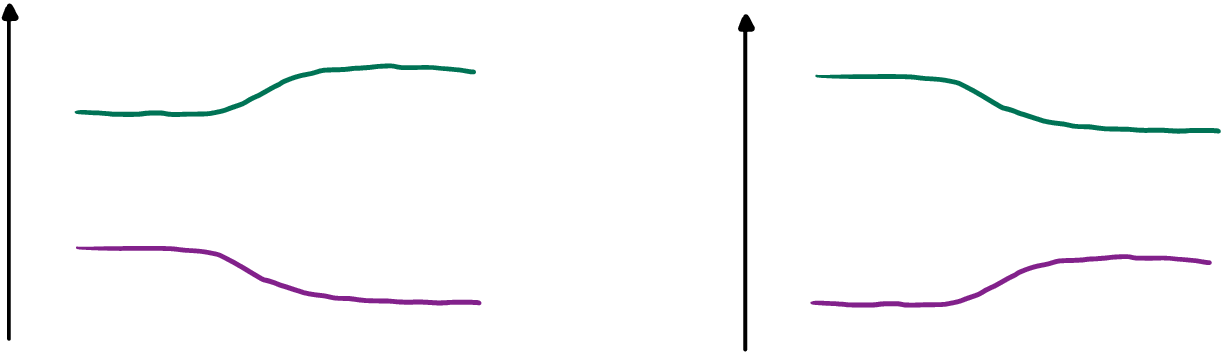}
    \put(40,21){$\overline{\ll_i}$}
    \put(40,4){$\overline{\ll_j}$}
    \put(101,18){$\overline{\ll_i}$}
    \put(101,7){$\overline{\ll_j}$}
    \end{overpic}
    \caption{Left: the Legendrian realization of $\ll_i$ and $\ll_j$ in the situation of Figure~\ref{fig:Fredy1}(a). Right: the Legendrian realization of $\ll_i$ and $\ll_j$ in the situation of Figure~\ref{fig:Fredy1}(b).}\label{fig:Fredy2}
\end{figure}

If $h_i(1) > h_j(1)$ and $h_i(0) = h_j(0)$, then again Step~5 tells us that the Legendrian realization of $\ll_i$ is a Legendrian segment always above the Legendrian realization of $\ll_j$. The reason for this is slightly less obvious, as $h_i(0) = h_j(0)$ indicates that the Legendrian realizations of the two segments start at the same relative height. More precisely, this means that the two Legendrian realizations are initially both embedded in a copy of the ribbon surface translated vertically by an amount $h_i(0) = h_j(0)$. The fact that $h_i(1) > h_j(1)$ indicates that the Legendrian realization of the segment $\ll_i$ then jumps to a copy of the ribbon surface that is higher than the copy of the ribbon surface to which the Legendrian realization of $\ll_j$ jumps. Because of this, it follows that the relative position of the two segments in $H$ is as in Figure~\ref{fig:Fredy1}(a), and their Legendrian realizations are as in Figure~\ref{fig:Fredy3}(a).

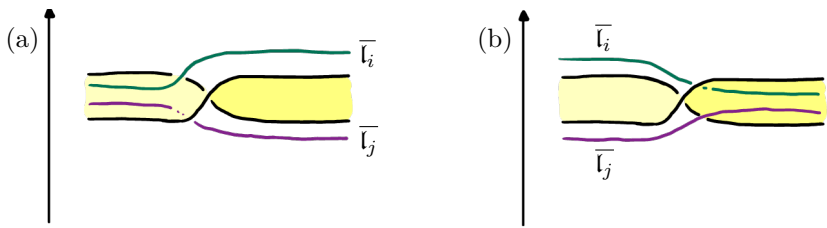
\begin{figure}[htbp]
    \centering
    \begin{overpic}[scale=1]{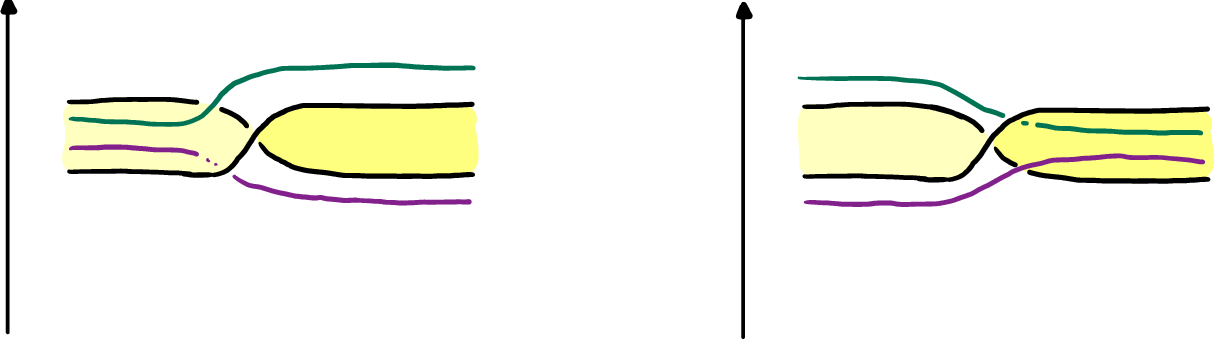}
    \put(-5,23){(a)}
    \put(55,23){(b)}
    \put(40,21){$\overline{\ll_i}$}
    \put(40,10){$\overline{\ll_j}$}
    \put(70,23){$\overline{\ll_i}$}
    \put(70,7){$\overline{\ll_j}$}   
    \end{overpic}
    \caption{(a) The Legendrian segments $\overline{\ll_i}$ and $\overline{\ll_j}$ are initially embedded in the same ribbon surface at height $h_i(0) = h_j(0)$. The $1$--handle of this ribbon surface is depicted. (b) The Legendrian segments $\overline{\ll_i}$ and $\overline{\ll_j}$ are eventually embedded in the same ribbon surface at height $h_i(1) = h_j(1)$. Again, the $1$--handle of this ribbon surface is depicted.}\label{fig:Fredy3}
\end{figure}

If $h_i(1) = h_j(1)$, then we have $h_i(0) > h_j(0)$, and it follows that the relative position of the two segments in $H$ is as in Figure~\ref{fig:Fredy1}(b). Therefore, the Legendrian realization of the segment $\ll_i$ lies entirely above the Legendrian realization of the segment $\ll_j$. See Figure~\ref{fig:Fredy3}(b).

If $h_i(1) < h_j(1)$ and $h_i(0) = h_j(0)$, then the relative position of the two segments in $H$ is as in Figure~\ref{fig:Fredy1}(a). In this case, the Legendrian realization of the segment $\ll_j$ lies entirely above the Legendrian realization of the segment $\ll_i$. See Figure~\ref{fig:Fredy4}(a).

\begin{figure}[htbp]
    \centering
    \begin{overpic}[scale=1]{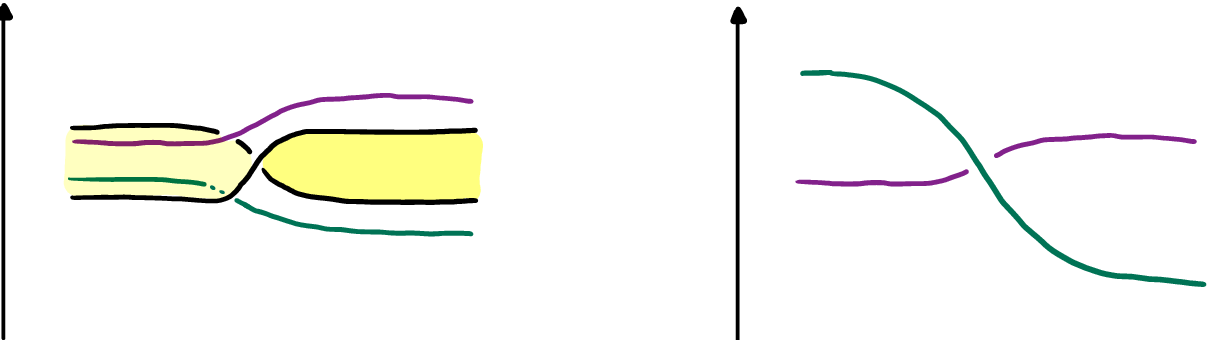}
    \put(-5,23){(a)}
    \put(55,23){(b)}
    \end{overpic}
    \caption{(a) The Legendrian segments $\overline{\ll_i}$ and $\overline{\ll_j}$ are initially embedded in the same ribbon surface at height $h_i(0) = h_j(0)$. The $1$--handle of this ribbon surface is depicted. (b) The Legendrian segments $\overline{\ll_i}$ and $\overline{\ll_j}$ cross.}\label{fig:Fredy4}
\end{figure}

The last remaining case is $h_i(1) < h_j(1)$ and $h_i(0) > h_j(0)$. In this case, the relative position of the two segments in $H$ is as in Figure~\ref{fig:Fredy1}(b), and the (front projection of the) Legendrian realizations of the two segments cross, as shown in Figure~\ref{fig:Fredy4}(b).

To fully justify why the braid created in Step~5 agrees with the front projection of the Legendrian realizations of the segments of $\aa$ in the $1$--handles, we would also need to analyze all other cases,\footnote{That is, cases where the two segments are oriented the other way around, have opposite orientations, or when the handle looks different from the one in Figure~\ref{fig:Fredy1}.} but the proof in all of these cases is very similar, so we omit it.

Next, we verify that Step~6 also produces Legendrian segments that match the Legendrian realization of the segments in a $0$--handle. For this, let $\ll_i$ and $\ll_j$ be two segments of $\aa$ in a $0$--handle $H$, and denote by $\boldsymbol{e_1}, \dots, \boldsymbol{e_n}$ the segments of the Legendrian skeleton of $H$. Step~6 claims that the Legendrian realization of a segment looks like a concatenation of two of the segments $\boldsymbol{e_1}, \dots, \boldsymbol{e_n}$, possibly perturbed and possibly translated vertically by a certain amount dictated by the prominence of these segments. 

If $P(\ll_i)>P(\ll_j)$, then it is clear that the Legendrian realization $\overline{\ll_i}$ of $\ll_i$ has to be higher than the Legendrian realization $\overline{\ll_j}$ of $\ll_j$, and therefore, independently of the precise shape of the two Legendrian segments, up to Legendrian isotopy, the relative position of these segments will be as explained in Step~6 and illustrated with an example in Figure~\ref{fig:justification_0handles}(b) (see also Figure~\ref{fig:Leg_real_0handle1}). If $P(\ll_i)=P(\ll_j)$, then the two Legendrian realizations $\overline{\ll_i}$ and $\overline{\ll_j}$ are both embedded in a perturbation of the same ribbon surface. Because of that, it is again clear why in Step~6 it is required that the segments $\overline{\ll_i}$ and $\overline{\ll_j}$ match the visual appearance of the segments $\ll_i$ and $\ll_j$ in the handle $H$. See Figure~\ref{fig:justification_0handles}(c) (see also Figure~\ref{fig:Leg_real_0handle2}) for an example showing this.

\begin{figure}[htbp]
    \centering
    \begin{overpic}[scale=1]{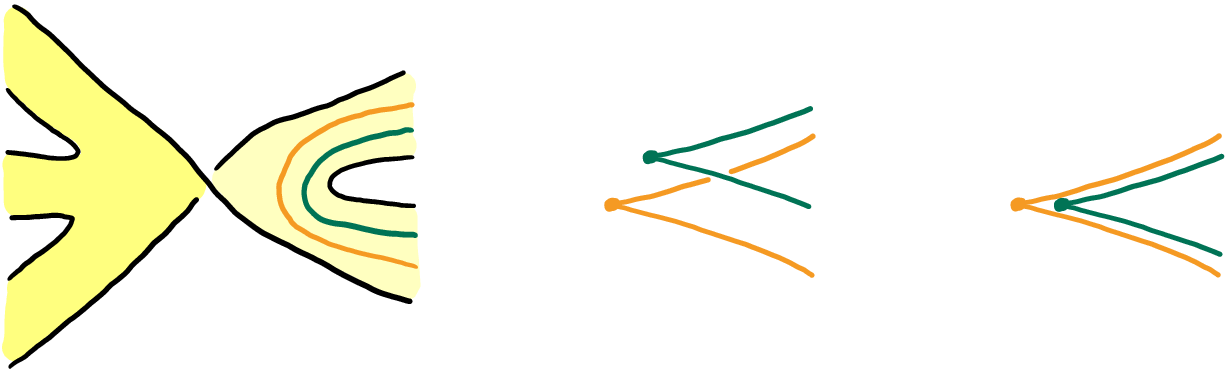}
        \put(-5,28){(a)}
        \put(45,28){(b)}
        \put(78,28){(c)}
    \end{overpic}
    \caption{(a) Segments of $\aa$ in a $0$--handle. (b) Legendrian realizations when the prominences are distinct: higher prominence means higher $z$--level. (c) Legendrian realizations when the prominences agree: all lie on the same level and are separated by a small perturbation.}\label{fig:justification_0handles}
\end{figure}

This proves all assertions in Proposition~\ref{prop:validity_algorithm} and hence 
completes the justification of the algorithm.
\end{proof}

\section{Applications: from open books to contact surgery diagrams}\label{sec:applications_OB}

The next result can be viewed as a refinement of Avdek's algorithm in~\cite{Avdek13}.

\begin{corollary}\label{corollary}
    Given an abstract open book decomposition $(\Sigma,\phi)$ whose monodromy $\phi$ is given explicitly as a composition of Dehn twists along homologically nontrivial simple closed curves, there exists an algorithm that produces a contact surgery diagram for the contact manifold supported by the open book $(\Sigma,\phi)$. 
    
    This algorithm produces a contact surgery link $\boldsymbol{L}=\boldsymbol{L_1}\cup\boldsymbol{L_2}$ for which the components of $\boldsymbol{L_1}$ are in bijection with the Dehn twists in the factorization of $\phi$, and all the components of $\boldsymbol{L_2}$ are $+1$–framed Legendrian unknots with Thurston–Bennequin number $tb=-1$.
\end{corollary}

Before delving into the proof, we first describe a suitable family of open books for $(S^3,\xist)$ whose pages have arbitrary genus and number of boundary components. We start from two basic open books:

\begin{itemize}
    \item one with page a genus~$1$ surface with connected boundary;
    \item one with page an annulus.
\end{itemize}

Each of these pages can be realized as the ribbon of a specific Legendrian graph in $(\R^3,\xist)$, shown in Figure~\ref{fig:corollary1}.

\begin{figure}[htbp]
    \centering
    \begin{overpic}[scale=1]{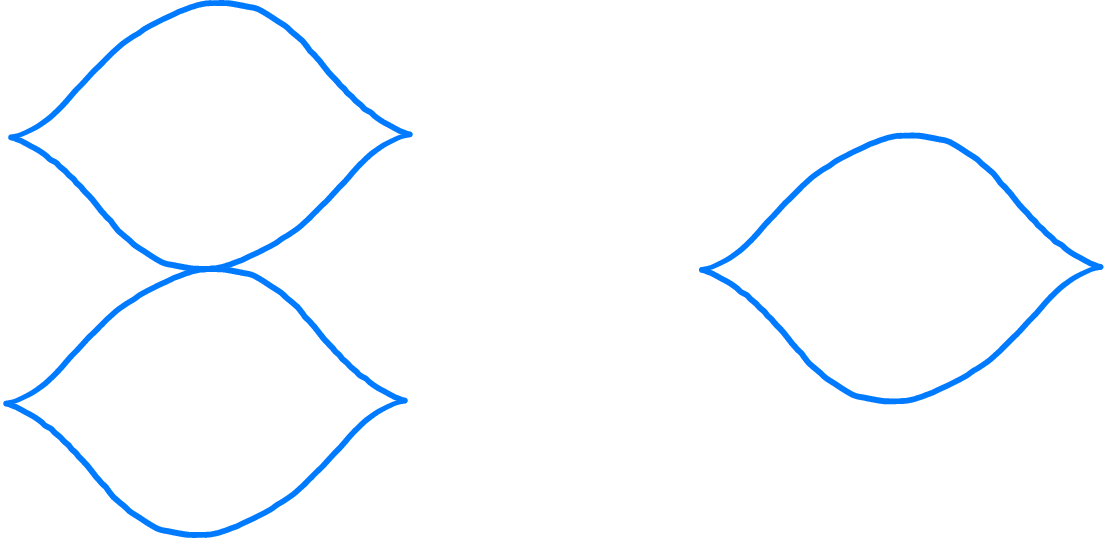}
        \put(30,44){$A$}
        \put(94,30){$B$}
    \end{overpic}
    \caption{Left: the Legendrian graph $A$ whose ribbon is a page of an open book with genus~$1$ and connected binding. Right: the Legendrian graph $B$ whose ribbon is an annular page.}\label{fig:corollary1}
\end{figure}

Let $A$ denote the graph on the left and $B$ the graph on the right. By assembling several copies of $A$ and $B$ as in Figure~\ref{fig:corollary2}, we obtain a Legendrian graph whose ribbon has genus equal to the number of $A$–pieces and number of boundary components equal to one plus the number of $B$–pieces.

\begin{figure}[htbp]
    \centering
    \begin{overpic}[scale=1]{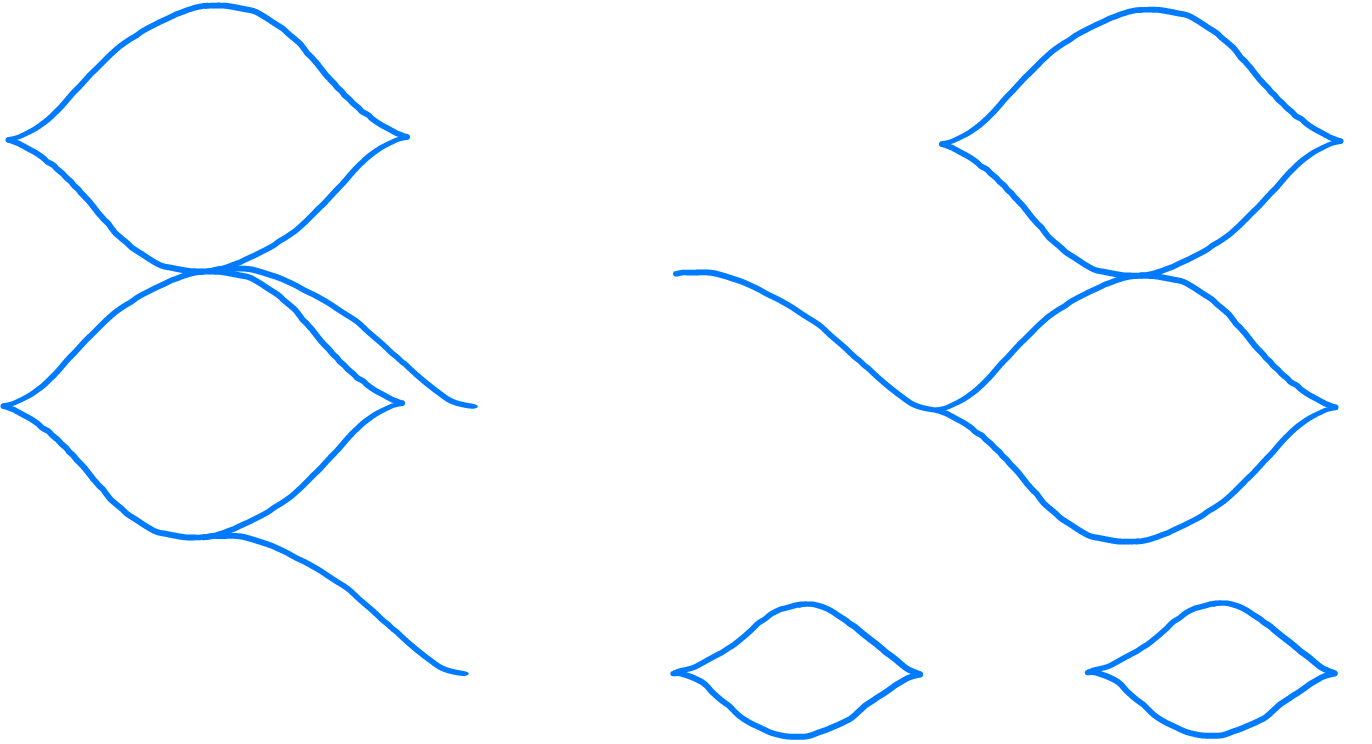}
        \put(39,30){$\boldsymbol\cdot$}
        \put(41,30){$\boldsymbol\cdot$}
        \put(43,30){$\boldsymbol\cdot$}
        \put(39,4){$\boldsymbol\cdot$}
        \put(41,4){$\boldsymbol\cdot$}
        \put(43,4){$\boldsymbol\cdot$}
        \put(70,4){$\boldsymbol\cdot$}
        \put(72,4){$\boldsymbol\cdot$}
        \put(74,4){$\boldsymbol\cdot$}
    \end{overpic}
    \caption{A Legendrian graph obtained by concatenating pieces of type $A$ and~$B$. Its ribbon has genus equal to the number of $A$–pieces and boundary components equal to one plus the number of $B$–pieces.}\label{fig:corollary2}
\end{figure}
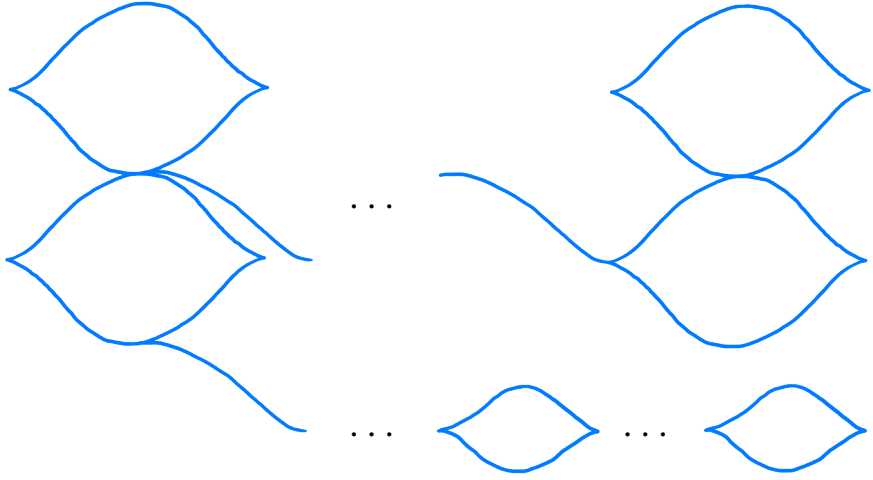

The ribbons of these graphs are pages of open books supporting $(S^3,\xist)$ (because of Theorem~\ref{thm:contact_cell}). Thus we have explicit models of open books for $(S^3,\xist)$ with any prescribed topology of the page.

\begin{proof}[Proof of Corollary~\ref{corollary}]
Let $(\Sigma,\phi)$ be an abstract open book as in the statement:
$\Sigma$ is a compact oriented surface with nonempty boundary, and
$\phi\in\mathrm{MCG}(\Sigma,\partial\Sigma)$ is given by a factorization
\[
  \phi
  \;=\;
  \tau_{\aa_m}^{\delta_m}\cdots\tau_{\aa_1}^{\delta_1},
  \qquad
  \delta_i\in\{\pm1\},
\]
where each $\aa_i$ is a homologically nontrivial simple closed curve
on~$\Sigma$.

\smallskip

\noindent
{\bf Step 1: Realize $\Sigma$ as a ribbon in $(S^3,\xist)$.}
By the construction in Figure~\ref{fig:corollary2}, we can build a
Legendrian graph $G\subset(\R^3,\xist)\subset(S^3,\xist)$ by
concatenating pieces of type $A$ and $B$ with the following property. The ribbon $R_G$ of
this graph is a compact surface with
\[
  R_G \cong \Sigma
\]
as oriented surfaces with boundary, and by construction it is transverse
to $\partial_z$. Moreover, the Legendrian graph $G$ is the
$1$--skeleton of a contact cell decomposition of $(S^3,\xist)$, so
Theorem~\ref{thm:contact_cell} implies that $R_G$ is a page of an open book
supporting $(S^3,\xist)$. We denote this open book by
\[
  (B,\pi)
  \quad\text{with corresponding abstract open book}\quad
  (R_G,\psi),
\]
where $\psi\in\MCG(R_G,\partial R_G)$ is the monodromy.

Now notice that the graph $A$ can be seen as two Legendrian unknots touching each other in one point. Call $A_1$ the top one and $A_2$ the bottom one. Since we may have many copies of $A$ and $B$ in our graph $G$, denote them by $A^1=A_1^1\cup A^1_2, \dots, A^k=A_1^k\cup A^k_2$ and $B^1,\dots, B^h$. With this notation, $\psi$ is given explicitly as a product of
right–handed Dehn twists along these curves (note that each $A^i_1$, $A^i_2$ and $B^j$ can be seen as a curve on $R_G$). More precisely, we
may write
\[
  \psi
  =\prod_{i=1}^k(\tau^+_{A_2^i}\tau^+_{A_1^i}) \prod_{j=1}^h
  \tau^+_{B^j},
\]
see \cite{Avdek13}.

Fix once and for all an orientation–preserving diffeomorphism
\[
  f\colon \Sigma \longrightarrow R_G.
\]
Via~$f$, we will freely
identify $\Sigma$ with $R_G$; in particular, the abstract open book
$(\Sigma,\phi)$ is equivalent to $(R_G,\phi')$, where
\[
  \phi' := f\phi f^{-1}
  \in\MCG(R_G,\partial R_G),
\]
and the given factorization of $\phi$ yields a factorization
\[
  \phi'
  =
  \tau_{\aa'_m}^{\delta_m}\cdots\tau_{\aa'_1}^{\delta_1},
  \qquad
  \aa'_i := f(\aa_i)\subset R_G.
\]
Since the $\aa_i$ are homologically nontrivial in $\Sigma$, the
curves $\aa'_i$ are homologically nontrivial in~$R_G$.

\smallskip

\noindent
{\bf Step 2: Legendrian realization of the curves on the ribbon.}
The ribbon $R_G$ is convex with dividing set equal to
$\partial R_G$. On such a surface, a simple closed curve
is nonisolating if and only if it is homologically nontrivial; in
particular, each $\aa'_i$ is nonisolating. By the Legendrian
realization principle and its explicit version given in our algorithm
(Theorem~\ref{thm:algorithm}), for each $i=1,\dots,m$ we can explicitly obtain a Legendrian knot
\[
  L_i \subset R_G\times\{t_i\} \subset (S^3,\xist)
\]
isotopic to $\aa'_i$ in $R_G$. We take care that $t_i>t_j$ if $i>j$ and that $t_i<0$ for each $i$.

\smallskip

\noindent
{\bf Step 3: Constructing a surgery link that encodes the monodromy.}
We now build a contact surgery link
\[
  \boldsymbol{L}
  =
  \Bigl(\bigcup_{i=1}^m L_i^{-\delta_i}\Bigr)
  \;\cup\;
  \Bigl(\bigcup_{i=1}^k (A^i_1\times\{2\varepsilon\})^{+}\cup(A^i_2\times\{\varepsilon\})^+ \Bigr)
  \;\cup\;
  \Bigl(\bigcup_{j=1}^h (B^j\times\{\varepsilon\})^{+}\Bigr)
  \subset (S^3,\xist).
\]

Thus on the components $L_i$ we perform contact
$(-\delta_i)$–surgery, and on the others contact $+1$–surgery. This is the desired contact surgery link.

\begin{remark}
Just for reference, the Legendrian link $(A^i_1\times\{2\varepsilon\})^{+}\cup(A^i_2\times\{\varepsilon\})^+$ is the linked configuration in which the component corresponding to $A_1^i$ lies above the component corresponding to $A_2^i$, as opposed to $(A^i_1\times\{\varepsilon\})^{+}\cup(A^i_2\times\{2\varepsilon\})^+$ where the vertical order is reversed (up to the surgery coefficients).
\end{remark}

We now explain why this procedure works.

\smallskip

\noindent
{\bf Effect of surgery on the monodromy.}
Proposition~\ref{prop:Dehn_surgery_Dehn_twist} expresses the effect of
contact $\pm1$–surgery on Legendrian knots contained in distinct pages of a
supporting open book in terms of Dehn twists on the page. In our
setting, notice that our link $\boldsymbol{L}$ is compatible with $R_G$, and so its components are embedded in pages of the open book $(B,\pi)$ (since $R_G$ is a page of this open book). Therefore the contact manifold represented by the contact surgery link $\boldsymbol{L}$ is supported by the open book
\[
  \bigl(R_G,\psi\circ\tau_{\boldsymbol{L}}\bigr),
\]
where $\tau_{\boldsymbol{L}}\in\MCG(R_G,\partial R_G)$ is the mapping
class obtained by composing Dehn twists along the curves represented by
the components of $\boldsymbol{L}$, with exponents determined by the
surgery coefficients. Concretely,
\[
  \tau_{\boldsymbol{L}}
  =
  \prod_{j=1}^h
  \tau^-_{B^j}\prod_{i=1}^k(\tau^-_{A_1^i}\tau^-_{A_2^i}) \;\tau_{\aa'_m}^{\delta_m}\cdots\tau_{\aa'_1}^{\delta_1}.
\]

Since, by the construction of $(R_G,\psi)$, we know that
\[
  \psi
  = \prod_{i=1}^k(\tau^+_{A_2^i}\tau^+_{A_1^i}) \prod_{j=1}^h
  \tau^+_{B^j},
\]
we obtain
\[
  \psi\circ\tau_{\boldsymbol{L}}
  = \psi\circ\psi^{-1}\circ\phi'
  = \phi'.
\]

Thus the contact manifold obtained by contact surgery on the Legendrian
link $\boldsymbol{L}\subset(S^3,\xist)$ is supported by the abstract
open book $(R_G,\phi')$.

\smallskip

\noindent
{\bf Identifying the resulting contact manifold.}
By construction, $(R_G,\phi')$ is equivalent to the original
abstract open book $(\Sigma,\phi)$ via the diffeomorphism $f\colon\Sigma\to R_G$. Therefore the contact manifold supported by $(R_G,\phi')$ is contactomorphic to the contact manifold supported by $(\Sigma,\phi)$.

Summarizing, starting from the explicit Dehn twist factorization of
$\phi$, the construction above:
\begin{itemize}
  \item builds a Legendrian graph $G$ whose ribbon is identified with
  $\Sigma$;
  \item realizes the Dehn twist curves in the factorization of
  $\phi$ (and those in the factorization of $\psi$) as explicit
  Legendrian knots on different slices $R_G\times\{t\}$, using the algorithm
  of Theorem~\ref{thm:algorithm};
  \item assigns to these Legendrian knots explicit surgery coefficients
  $\pm1$ so that the resulting surgery link $\boldsymbol{L}\subset
  (S^3,\xist)$ yields a contact manifold supported by
  $(\Sigma,\phi)$. The sublink coming from the Legendrian realization of the curves in the factorization of $\psi$ is used to cancel the monodromy of the open book $(B,\pi)$ for the standard contact sphere, while the other components are used to produce the monodromy $\phi$.
\end{itemize}
All steps are algorithmic and depend only on the data
$(\Sigma,\phi=\tau_{\aa_m}^{\delta_m}\cdots\tau_{\aa_1}^{\delta_1})$.
Hence we have produced the desired algorithm that associates to
$(\Sigma,\phi)$ a contact surgery diagram for the contact
manifold it supports, proving Corollary~\ref{corollary}.
\end{proof}

\printbibliography

\end{document}